\documentclass[a4paper,reqno]{amsart}
\usepackage[english]{babel}

\usepackage{amsmath,amstext,amssymb,amsthm}
\usepackage{mathrsfs}
\usepackage{bbm}
\usepackage{enumitem}

\usepackage{hyperref}

\usepackage{graphicx}

\usepackage{tikz}
\usetikzlibrary{arrows}

\newcommand{\DOIlink}[1]{\href{https://doi.org/#1}{\texttt{doi:#1}}}
\newcommand{\arXivlink}[1]{\href{https://arxiv.org/abs/#1}{\texttt{arXiv:#1}}}

\newcommand*\RR{\mathbb{R}}
\newcommand*\CC{\mathbb{C}}
\newcommand*\NN{\mathbb{N}}

\newcommand*{\Rnon}{\RR_+}
\newcommand*{\Rpos}{\mathring{\RR}_+}

\newcommand*{\loc}{\mathrm{loc}}
\newcommand*{\op}{\mathrm{op}}

\newcommand*\ind{\mathbbm{1}}

\newcommand{\sobolev}[2]{L^{#2}_{#1}}

\newcommand*{\Lip}{\mathrm{Lip}}
\newcommand*{\AC}{\mathrm{AC}}

\newcommand*{\Hom}{\mathrm{Hom}}
\newcommand*{\Imm}{\mathop{\mathrm{Im}}}

\DeclareMathOperator{\spann}{span}
\DeclareMathOperator{\sign}{sign}

\newcommand*\supp{\mathop{\mathrm{supp}}}
\newcommand*\arccosh{\mathop{\mathrm{arccosh}}}

\newcommand*{\lie}[1]{\mathfrak{#1}}

\newcommand*\id{\mathrm{id}}

\newcommand*{\tc}{\mathrel{:}}
\newcommand*{\defeq}{\mathrel{:=}}

\newcommand{\Beta}{\mathrm{B}}

\DeclareSymbolFont{t1letters}{T1}{cmr}{m}{it}
\DeclareMathSymbol{\DD}{0}{t1letters}{208}

\newcommand*{\dist}{\varrho}

\newcommand*\Sz{\mathcal{S}}
\newcommand*\Df{\mathcal{D}}

\newcommand*\fctE{\mathcal{E}}
\newcommand*\fctJ{\mathcal{J}}

\newcommand*\Four{\mathcal{F}}

\newcommand*\bdlE{\mathscr{E}}
\newcommand*\bdlF{\mathscr{F}}

\newcommand*\dd{\mathrm{d}}
\newcommand*\dmu{\dd\mu_\nu}
\newcommand*\dv{\dd v}
\newcommand*\du{\dd u}
\newcommand*\dz{\dd z}
\newcommand*\ddx{\frac{\dd x}{x}}
\newcommand*\ddt{\frac{\dd t}{t}}
\newcommand*\ddlam{\frac{\dd\lambda}{\lambda}}

\newcommand*\Dom{\mathfrak{D}}
\newcommand*\Dmax{\Dom_{\mathrm{max}}}
\newcommand*\Dmin{\Dom_{\mathrm{min}}}
\newcommand*\Dmaxc{\Dom_{\mathrm{max},c}}
\newcommand*\Dmaxmin{\Dom_{\mathrm{mix}}}
\newcommand*\DNeu{\Dom_{\mathrm{Neu}}}
\newcommand*\DDir{\Dom_{\mathrm{Dir}}}

\newcommand*\RRone{\RR^{\vec d_1}}
\newcommand*\RRtwo{\RR^{\vec d_2}}

\newcommand*\al{\alpha}
\newcommand*{\val}{|\alpha|}
\newcommand*\be{\beta}
\newcommand*\te{\theta}

\theoremstyle{plain}
\newtheorem{thm}{Theorem}[section]
\newtheorem{lm}[thm]{Lemma}
\newtheorem{prop}[thm]{Proposition}
\newtheorem{cor}[thm]{Corollary}
\theoremstyle{remark}
\newtheorem{rem}[thm]{Remark}

\numberwithin{equation}{section}

\begin{document}
\title[Multiplier theorem for solvable extensions]{A sharp multiplier theorem for solvable extensions of Heisenberg and related groups}

\author[A. Martini]{Alessio Martini}
\address[A. Martini]{Dipartimento di Scienze Matematiche ``G.L. Lagrange'' \\ Politecnico di Torino \\ Corso Duca degli Abruzzi 24 \\ 10129 Torino \\ Italy}
\email{alessio.martini@polito.it}

\author[P. Plewa]{Pawe\l{} Plewa}
\address[P. Plewa]{Dipartimento di Scienze Matematiche ``G.L. Lagrange'',
	Politecnico di Torino\\ 
	and Department of Pure and Applied Mathematics, 
	Wroc{\l}aw University of Science and Technology}        
\email{pawel.plewa@polito.it}

\begin{abstract}
Let $G$ be the semidirect product $N \rtimes \RR$, where $N$ is a stratified Lie group and $\RR$ acts on $N$ via automorphic dilations. Homogeneous left-invariant sub-Laplacians on $N$ and $\RR$ can be lifted to $G$, and their sum $\Delta$ is a left-invariant sub-Laplacian on $G$. In previous joint work of Ottazzi, Vallarino and the first-named author, a spectral multiplier theorem of Mihlin--H\"ormander type was proved for $\Delta$, showing that an operator of the form $F(\Delta)$ is of weak type $(1,1)$ and bounded on $L^p(G)$ for all $p \in (1,\infty)$ provided $F$ satisfies a scale-invariant smoothness condition of order $s > (Q+1)/2$, where $Q$ is the homogeneous dimension of $N$. Here we show that, if $N$ is a group of Heisenberg type, or more generally a direct product of M\'etivier and abelian groups, then the smoothness condition can be pushed down to the sharp threshold $s>(d+1)/2$, where $d$ is the topological dimension of $N$. The proof is based on lifting to $G$ weighted Plancherel estimates on $N$ and exploits a relation between the functional calculi for $\Delta$ and analogous operators on semidirect extensions of Bessel--Kingman hypergroups.
\end{abstract}



\thanks{The authors gratefully acknowledge the financial support of Compagnia di San Paolo. The first-named author is a member of Gruppo Nazionale per l'Analisi Matematica, la Probabilit\`a e le loro Applicazioni (GNAMPA) of Istituto Nazionale di Alta Matematica (INdAM)}

\maketitle

\section{Introduction}
Let $N$ be a stratified Lie group of step $r$. In other words, $N$ is a connected, simply connected nilpotent Lie group, whose Lie algebra $\lie{n}$ is decomposed as a direct sum $\lie{n}_1 \oplus \dots \oplus \lie{n}_r$ of nontrivial subspaces $\lie{n}_j$, called layers, in such a way that $[\lie{n}_1,\lie{n}_j] = \lie{n}_{j+1}$ for all $j=1,\dots,r$, where $\lie{n}_{r+1} = \{0\}$. As $N$ is nilpotent and simply connected, the exponential map is a diffeomorphism between $\lie{n}$ and $N$, and in exponential coordinates the Lebesgue measure on $\lie{n}$ corresponds to the (left and right) Haar measure on $N$. We denote by $d = \sum_{j=1}^r \dim \lie{n}_j$ and $Q = \sum_{j=1}^r j \dim\lie{n}_j$ the topological and homogeneous dimensions of $N$. Let $\Delta_N$ be a homogeneous left-invariant sub-Laplacian on $N$, that is an operator of the form
\begin{equation}\label{eq:N_sublaplacian}
\Delta_N = -\sum_{j=1}^{\dim\lie{n}_1} X_j^2,
\end{equation}
where the $X_j$ are left-invariant vector fields forming a basis of the first layer $\lie{n}_1$.

The stratified group $N$ is naturally equipped with a one-parameter group of automorphic dilations $(e^{uD})_{u \in \RR}$, where $D$ is the derivation of $\lie{n}$ which is $j$ times the identity on $\lie{n}_j$. We can then form the semidirect product $G = N \rtimes \RR$, where $\RR$ acts on $N$ by dilations; in other words, the product law on $G$ is given by
\[
(z,u) \cdot (\tilde z, \tilde u) = (z\cdot e^{u D} \tilde z, u+\tilde u)
\]
for all $(z,u),(\tilde z,\tilde u) \in G$. In these coordinates, the product measure $\dz \, \du$ of the Haar measure on $N$ and the Lebesgue measure on $\RR$ is the right Haar measure on $G$, while the left Haar measure is given by $e^{-Qu} \,\dz \,\du$; unless otherwise specified, we shall use the right Haar measure when integrating on $G$ and in the definition of the Lebesgue spaces $L^p(G)$.

The vector fields $\partial_u$ on $\RR$ and $X_j$ ($j=1,\dots,\dim \lie{n}_1$) on $N$ lift to left-invariant vector fields on the semidirect product $G$, given by
\begin{equation}\label{eq:liftedvectorfields}
X_j^\sharp = \begin{cases} \partial_u &\text{if } j=0,\\
e^u X_j & \text{if } j=1,\dots,\dim\lie{n}_1,
\end{cases}
\end{equation}
and we can consider the corresponding left-invariant sub-Laplacian
\begin{equation}\label{eq:G_sublaplacian}
\Delta = -\sum_{j=0}^{\dim \lie{n}_1} (X_j^\sharp)^2 = -\partial_u^2 + e^{2u} \Delta_N
\end{equation}
on $G$. This operator is essentially self-adjoint on $L^2(G)$; so, by the spectral theorem, the sub-Laplacian $\Delta$ has a Borel functional calculus, and for any bounded Borel function $F : [0,\infty) \to \CC$ the operator $F(\Delta)$ is bounded on $L^2(G)$.

Here we are interested in investigating what additional conditions are to be required on the spectral multiplier $F$ so that the operator $F(\Delta)$, initially defined on $L^2(G)$, extends to a bounded operator on $L^p(G)$ for some $p \neq 2$. More specifically, we look for relations between $L^p$-boundedness properties of $F(\Delta)$ and size and smoothness properties of $F$, such as Mihlin--H\"ormander type estimates of the form
\begin{equation}\label{eq:mh}
\| F(\Delta) \|_{L^p(G) \to L^p(G)} \lesssim_{p,s} \sup_{t>0} \| F(t \cdot) \chi \|_{\sobolev{s}{2}(\RR)}
\end{equation}
for appropriate values of $p$ and $s$, where $\sobolev{s}{2}(\RR)$ is the $L^2$ Sobolev space of order $s$ on $\RR$ and $\chi \in C^\infty_c(0,\infty)$ is a nontrivial cutoff. These are generalisations of the classical estimates for the Laplace operator on $\RR^n$, for which the analogue of \eqref{eq:mh} holds for any $p \in (1,\infty)$ and $s > n/2$, as a consequence of the Mihlin--H\"ormander theorem for Fourier multipliers \cite{Ho,Mi}.

We point out that $G$ is a solvable Lie group of exponential volume growth. In particular, the ``standard machinery'' providing $L^p$-bounds of Mihlin--H\"ormander type for the functional calculus for sub-Laplacians on Lie groups of polynomial growth and more general doubling metric-measure spaces (see, e.g., \cite{A,CoSi,DOSi,He95}) does not apply here; indeed, if we equip $G$ with the Carnot--Carath\'eodory distance associated with the vector fields \eqref{eq:liftedvectorfields}, then $G$ is locally doubling, but not globally. Nevertheless, somewhat surprisingly, the sub-Laplacian $\Delta$ on $G$ has a differentiable functional calculus on $L^p(G)$ \cite{CGHM,G,He93b,Mus}; this should be contrasted with what happens for other exponential solvable Lie groups and sub-Laplacians, which may be of holomorphic $L^p$-type \cite{ChMu,HeLM}. More recently, by developing and extending previous results and ideas in \cite{He99,HeSt,Va07}, the authors of \cite{MOV} proved a spectral multiplier theorem of Mihlin--H\"ormander type for $\Delta$, which can be stated as follows.

For a fixed cutoff $\chi \in C_c^\infty(0,\infty)$ with $\ind_{[1/2,2]} \leq \chi \leq \ind_{[1/4,4]}$,
we define for $s\geq 0$ and a bounded Borel function $F : [0,\infty) \to \CC$ the quantities
\begin{equation*}
\Vert F\Vert_{0,s} =\sup_{0<t \leq 1} \Vert F(t\cdot)\chi\Vert_{\sobolev{s}{2}(\RR)}, \qquad
\Vert F\Vert_{\infty,s} =\sup_{t\geq 1} \Vert F(t\cdot)\chi\Vert_{\sobolev{s}{2}(\RR)}.
\end{equation*}
These are variants of the scale-invariant Sobolev norm in the right-hand side of \eqref{eq:mh}, focusing on the ``local part'' and the ``part at infinity'' of the function $F$.

\begin{thm}[\cite{HeSt,MOV}]\label{thm:mov}
Suppose that both $s_0,s_\infty>3/2$, and that moreover $s_\infty>(Q+1)/2$. If a bounded Borel function $F : [0,\infty) \to \CC$ satisfies $\Vert F\Vert_{0,s_0}<\infty$ and $\Vert F\Vert_{\infty,s_\infty}<\infty$, then $F(\Delta)$ extends to an operator of weak type $(1,1)$ and bounded on $L^p(G)$ for $p \in (1,\infty)$, bounded from $H^1(G)$ to $L^1(G)$ and from $L^\infty(G)$ to $BMO(G)$.
\end{thm}

We refer to \cite{MOV,Va} for the precise definitions of the atomic Hardy space $H^1(G)$ and the dual space $BMO(G)$, which are adapted to the Calder\'on--Zygmund structure of $G$ in the sense of \cite{HeSt}.

We point out that the restriction $s_\infty>3/2$ in the above statement is implied by $s_\infty > (Q+1)/2$ except when $Q=1$, in which case $N$ must be abelian.
In general, the above result requires different orders of smoothness $s_0$ and $s_\infty$ on the local part and the part at infinity of the multiplier $F$, related to the ``pseudodimension'' $3$ and the ``local doubling dimension'' $Q+1$ of $G$; here the pseudodimension plays the role that the ``dimension at infinity'' plays in \cite{A} for groups of polynomial growth.
In any case, Theorem \ref{thm:mov} implies the validity of the estimate \eqref{eq:mh} for any $p \in (1,\infty)$ and $s > \max\{3,Q+1\}/2$.

When $N$ is abelian, i.e.\ $N$ has step $1$, then $N \cong \RR^d$ and the above result is already contained in \cite{HeSt}. In that case, $\Delta_N$ is just the Laplacian on $\RR^d$, and $\Delta$ is elliptic; moreover, $\lie{n} = \lie{n}_1$ and $Q = d$. A classical transplantation argument \cite{KST} then shows that the smoothness condition $s_\infty > (Q+1)/2$ is sharp,  in the sense that the threshold $(Q+1)/2$, which in that case equals half the topological dimension $(d+1)/2$ of $G$, cannot be replaced by a smaller quantity.

Instead, when $N$ is not abelian, the operators $\Delta_N$ and $\Delta$ are not elliptic, and $Q>d \geq 3$. Here it is meaningful to ask about the sharpness of the condition $s_\infty > (Q+1)/2$ in the previous theorem. Indeed, by the results of \cite{MMNG}, we know that the threshold $(Q+1)/2$ cannot be replaced by anything smaller than $(d+1)/2$, but a gap remains between the two quantities. The main result of this paper shows that, at least for certain classes of nonabelian stratified groups $N$, the condition can indeed be pushed down to $s_\infty > (d+1)/2$, thus leading to a sharp result, which improves Theorem \ref{thm:mov} in this case.

We recall that the stratified group $N$ is said to be a \emph{M\'etivier group} \cite{Me} if $N$ has step $2$ and, for all $\mu \in \lie{n}_2^* \setminus \{0\}$, the skewsymmetric bilinear form $\mu[\cdot,\cdot] : \lie{n}_1 \times \lie{n}_1 \to \RR$ is nondegenerate. Heisenberg groups and, more generally, groups of Heisenberg type in the sense of Kaplan \cite{Ka} are M\'etivier groups, but there exist also M\'etivier groups which are not of Heisenberg type \cite{MuSe}. Our improvement of Theorem \ref{thm:mov} applies in particular to the case where $N$ is a M\'etivier group.

\begin{thm}\label{thm:1}
Assume that the $2$-step group $N$ is a direct product of finitely many M\'etivier and abelian groups.	Suppose that $s_0>3/2$ and $s_\infty>(d+1)/2$. If a bounded Borel function $F : [0,\infty)\to\CC$ satisfies $\Vert F\Vert_{0,s_0}<\infty$ and $\Vert F\Vert_{\infty,s_\infty}<\infty$, then $F(\Delta)$ extends to an operator of weak type $(1,1)$ and bounded on $L^p(G)$ for $p \in (1,\infty)$, bounded from $H^1(G)$ to $L^1(G)$ and from $L^\infty(G)$ to $BMO(G)$.
\end{thm}

As a consequence, under the assumptions of this theorem, the estimate \eqref{eq:mh} holds true for all $p \in (1,\infty)$ and $s>(d+1)/2$. We point out once again that the threshold $(d+1)/2$, corresponding to half the topological dimension of $G$, is sharp. In this respect, this result can be considered as part of a programme aimed at determining the sharp threshold in Mihlin--H\"ormander estimates for non-elliptic ``Laplace-like'' operators in a variety of settings; we refer to \cite{CCM,CoSi,DalMa,MM16,MMNG,MuSt} for a more extensive discussion and further references. The relevance of Theorem \ref{thm:1} in this context is that it appears to be the first sharp result of this type where the underlying manifold has exponential volume growth.

By a contraction argument (see, e.g., \cite[Theorem 5.2]{Ma17}), Theorems \ref{thm:mov} and \ref{thm:1} imply corresponding results for the direct product $\tilde G = N \times \RR$, with smoothness conditions $s_0=s_\infty > (Q+1)/2$ and $s_0=s_\infty > (d+1)/2$ respectively. However, these results for $\tilde G$ are already available in the literature. Indeed, the direct product $\tilde G$ is a stratified group itself, and the results deduced by contraction from Theorems \ref{thm:mov} and \ref{thm:1} correspond to the Christ--Mauceri--Meda theorem \cite{Ch,MaMe} for homogeneous sub-Laplacians on stratified groups, and its improvement due to Hebisch \cite{He93} and M\"uller and Stein \cite{MuSt} for Heisenberg and related groups.

It is still an open problem whether the improvement to half the topological dimension in the Christ--Mauceri--Meda theorem is always possible for an arbitrary stratified group (see, e.g., the discussion in \cite{MM14,MM16,MMNG}). As a consequence, the question whether the additional assumption on $N$ in Theorem \ref{thm:1} can be dropped altogether appears to be out of reach at this time, as the analogous and apparently easier question for the direct product $\tilde G$ is still open. Nevertheless, the technique developed in this paper can be used to improve Theorem \ref{thm:mov} for a larger class of nonabelian stratified groups $N$ than the one considered in Theorem \ref{thm:1}, also including some groups of step higher than $2$.

We refer to \cite[Section 3]{Ma12} for the definition of \emph{$h$-capacious} stratified group.  Here we just recall that any stratified group $N$ is $0$-capacious, but, if $N$ is M\'etivier, then it is also $(Q-d)$-capacious. Moreover, if $N$ is the direct product of an $h_1$-capacious and an $h_2$-capacious group, then $N$ is $(h_1+h_2)$-capacious. Furthermore, if $N$ has step $r$ and $\dim\lie{n}_r=1$, then $N$ is $1$-capacious \cite[Proposition 3.9]{Ma12}. So the following result properly extends Theorem \ref{thm:1}.

\begin{thm}\label{thm:capacious}
Assume that $N$ is $h$-capacious for some $h \in \NN$.	Suppose that $s_0,s_\infty>3/2$ and $s_\infty>(Q-h+1)/2$. If a bounded Borel function $F : [0,\infty)\to\CC$ satisfies $\Vert F\Vert_{0,s_0}<\infty$ and $\Vert F\Vert_{\infty,s_\infty}<\infty$, then $F(\Delta)$ extends to an operator of weak type $(1,1)$ and bounded on $L^p(G)$ for $p \in (1,\infty)$, bounded from $H^1(G)$ to $L^1(G)$ and from $L^\infty(G)$ to $BMO(G)$.
\end{thm}

This result can be compared with that in \cite[Corollary 6.1]{Ma12} for homogeneous sub-Laplacians on $h$-capacious stratified groups. The proof of Theorem \ref{thm:capacious} is a relatively straightforward modification of the proof of Theorem \ref{thm:1}, but requires several adjustments and changes of notation. In order not to affect the clarity of the presentation, below we only discuss the details of the proof of Theorem \ref{thm:1}.

\subsection*{Proof strategy}

By duality and interpolation, each of the spectral multiplier theorems for $\Delta$ stated above reduces to endpoint estimates of the form
\begin{equation}\label{eq:endpoint}
\begin{aligned}
\|F(\Delta)\|_{L^1(G) \to L^{1,\infty}(G)} &\lesssim_\varepsilon \|F\|_{0,\varsigma_0+\varepsilon} + \|F\|_{\infty,\varsigma_\infty+\varepsilon}, \\
\|F(\Delta)\|_{H^1(G) \to L^1(G)} &\lesssim_\varepsilon \|F\|_{0,\varsigma_0+\varepsilon} + \|F\|_{\infty,\varsigma_\infty+\varepsilon}
\end{aligned}
\end{equation}
for all $\varepsilon > 0$, where $\varsigma_0,\varsigma_\infty \geq 0$ are appropriate thresholds. As in other works on the subject, by means of Calder\'on--Zygmund theory, the estimates \eqref{eq:endpoint} can essentially be reduced to $L^1$-estimates for the convolution kernels $K_{F(t\Delta)}$ of the operators $F(t\Delta)$, corresponding to rescaled versions of a multiplier $F$ with compact support $\supp F \subseteq [-4,4]$:
\[
\sup_{0 < t \leq 1} \|K_{F(t\Delta)}\|_{L^1(G)} \lesssim_\varepsilon \|F\|_{\sobolev{\varsigma_0+\varepsilon}{2}(\RR)},  \quad \sup_{t \geq 1} \|K_{F(t\Delta)}\|_{L^1(G)} \lesssim_\varepsilon \|F\|_{\sobolev{\varsigma_\infty+\varepsilon}{2}(\RR)}.
\]
In turn, via a frequency decomposition, these can be deduced, at least when $\varsigma_0 \leq \varsigma_\infty$, from an estimate of the form
\begin{equation}\label{eq:l1l2}
\|K_{F(\sqrt{\Delta})}\|_{L^1(G)} \lesssim r^{[\varsigma_\infty,\varsigma_0]} \left( \int_0^\infty |F(\lambda)|^2 \, \lambda^{[2\varsigma_0,2\varsigma_\infty]} \ddlam \right)^{1/2}
\end{equation}
for all $r > 0$ and $F \in \fctE_r$; here $\fctE_r$ is the set of the even Schwartz functions $F : \RR \to \CC$ whose Fourier supports are contained in $[-r,r]$, and we write
\begin{equation}\label{eq:bipower}
\lambda^{[a,b]}= \begin{cases}
\lambda^a &\text{if } \lambda \leq 1,\\
\lambda^b &\text{if } \lambda \geq 1,
\end{cases}
\end{equation}
for any $a,b \in \RR$ and $\lambda >0$.

When $\varsigma_0 = 3/2$ and $\varsigma_\infty = (Q+1)/2$, the estimate \eqref{eq:l1l2} can be proved by combining the Plancherel formula for the functional calculus for $\Delta$, namely
\begin{equation}\label{eq:plancherel}
\|K_{F(\sqrt{\Delta})}\|_{L^2(G)}^2 \simeq \int_0^\infty |F(\lambda)|^2 \, \lambda^{[3,Q+1]} \ddlam,
\end{equation}
and the finite propagation speed property for $\Delta$, namely
\[
\supp K_{F(\sqrt{\Delta})} \subseteq \overline{B_G}(0_G,r)
\]
whenever $F \in \fctE_r$, where $\overline{B_G}(0_G,r)$ is the closed ball of radius $r$ centred at the origin of $G$ with respect to the Carnot--Carath\'eodory distance. Indeed, by finite propagation speed and the Cauchy--Schwarz inequality,
\[
\|K_{F(\sqrt{\Delta})}\|_{L^1(G)} \leq |\overline{B_G}(0_G,r)|^{1/2} \|K_{F(\sqrt{\Delta})}\|_{L^2(G)}
\]
when $F \in \fctE_r$, which in view of the Plancherel formula \eqref{eq:plancherel} gives \eqref{eq:l1l2} when $r \leq 1$, because $|\overline{B_G}(0_G,r)| \simeq r^{Q+1}$ in this case. This argument does not work as it is for $r \geq 1$, since $|\overline{B_G}(0_G,r)| \simeq \exp(Qr)$ for large $r$. However, one can fix this by introducing a suitable weight when applying the Cauchy--Schwarz inequality, in order to kill the exponential volume growth; finite propagation speed and radiality properties of $K_{F(\sqrt{\Delta})}$ can then be used, roughly speaking, to get rid of the extra weight in the $L^2$ norm, and obtain \eqref{eq:l1l2} for large $r$. This is broadly the approach used in \cite{HeSt,MOV} to prove Theorem \ref{thm:mov}.

In order to improve the result, i.e., to push down the threshold $\varsigma_\infty$, here we introduce, also for small $r$, an appropriate weight $w= w(z)$ in the application of the Cauchy--Schwarz inequality:
\[
\|K_{F(\sqrt{\Delta})}\|_{L^1(G)} \leq \left(\int_{\overline{B_G}(0_G,r)} w(z)^{-2} \,dz \,du \right)^{1/2} \|w \, K_{F(\sqrt{\Delta})}\|_{L^2(G)}.
\]
If $w$ is chosen so that $\int_{\overline{B_G}(0_G,r)} w(z)^{-2} \,dz \,du \simeq r^{2\varsigma_\infty}$ for $r \leq 1$, then the problem of obtaining \eqref{eq:l1l2} is essentially reduced, at least for small $r$, to the proof of a ``weighted Plancherel estimate'' of the form
\begin{equation}\label{eq:wplancherel}
\|w \, K_{F(\sqrt{\Delta})}\|_{L^2(G)}^2 \lesssim \int_0^\infty |F(\lambda)|^2 \, \lambda^{[3,2\varsigma_\infty]} \ddlam.
\end{equation}

In the case $\varsigma_\infty = (\nu+1)/2$ for some integer $\nu$, in light of \eqref{eq:plancherel} we can equivalently rewrite the previous inequality as
\begin{equation}\label{eq:wplancherel_inter}
\|w \, K_{F(\sqrt{\Delta})}\|_{L^2(G)} \lesssim \| K_{F(\sqrt{\Delta_\nu})} \|_{L^2(G_\nu)},
\end{equation}
where $G_\nu$ and $\Delta_\nu$ are the analogues of $G$ and $\Delta$ when $N$ is replaced by $\RR^\nu$. As $w$ only depends on the variable $z$, the relation between the functional calculi of $\Delta$ and $\Delta_N$ allows one to reduce \eqref{eq:wplancherel_inter} to a similar estimate on the stratified group $N$:
\begin{equation}\label{eq:Nplancherel_inter}
\|w \, K_{F(\sqrt{\Delta_N})}\|_{L^2(N)} \lesssim \| K_{F(\sqrt{\Delta_{\RR^\nu}})} \|_{L^2(\RR^\nu)},
\end{equation}
or equivalently,
\begin{equation}\label{eq:Nplancherel}
\|w \, K_{F(\sqrt{\Delta_N})}\|_{L^2(N)}^2 \lesssim \int_0^\infty |F(\lambda)|^2 \, \lambda^{\nu} \ddlam.
\end{equation}
As it turns out, when $N$ is a direct product of M\'etivier and abelian groups, the ``weighted Plancherel estimate'' \eqref{eq:Nplancherel} on $N$ holds true for any $\nu \in (d,Q]$, with a suitable weight $w$ depending on $\nu$. Indeed, such an estimate is the fundamental ingredient used in \cite{He93} to sharpen the Christ--Mauceri--Meda theorem for this class of stratified groups. The intermediate steps \eqref{eq:Nplancherel_inter} and \eqref{eq:wplancherel_inter} allow one to lift the weighted estimate to $G$ and deduce \eqref{eq:wplancherel} for $\varsigma_\infty = (\nu+1)/2$, where $\nu \in (d,Q] \cap \NN$.

In order to prove Theorem \ref{thm:1} with the sharp condition $s_\infty > (d+1)/2$, in principle we would like to take $\varsigma_\infty = (d+1)/2$, that is $\nu=d$, but the weighted Plancherel estimate \eqref{eq:Nplancherel} on $N$ fails in that case. As a workaround, we can take instead $\nu = d+\varepsilon$ for arbitrarily small $\varepsilon>0$, as \eqref{eq:Nplancherel} then holds true, but this requires us to work with non-integer $\nu$. The problem then becomes how to make sense of the intermediate steps \eqref{eq:Nplancherel_inter} and \eqref{eq:wplancherel_inter} in the lifting argument leading to \eqref{eq:wplancherel}, as $\RR^\nu$ and $G_\nu = \RR^\nu \rtimes \RR$ are not defined when $\nu$ is fractional.

The solution that we adopt here is to replace the Laplacian $\Delta_{\RR^\nu}$ on $\RR^\nu$ with the Bessel operator $L_\nu = -\partial_x^2 - (\nu-1) x^{-1} \partial_x$ on the half-line $X_\nu = [0,\infty)$ equipped with the measure $x^{\nu-1} \,\dd x$. When $\nu$ is an integer, the Bessel operator $L_\nu$ is just the radial part of the Laplacian on $\RR^\nu$, but of course $L_\nu$ and $X_\nu$ make sense for fractional $\nu$ as well, and provide the following replacement for \eqref{eq:Nplancherel_inter}:
\begin{equation}\label{eq:Nplancherel_inter2}
\|w \, K_{F(\sqrt{\Delta_N})}\|_{L^2(N)} \lesssim \| K_{F(\sqrt{L_\nu})} \|_{L^2(X_\nu)}.
\end{equation}
Similarly, the lifted estimate \eqref{eq:wplancherel_inter} becomes meaningful for fractional $\nu$ by taking as $G_\nu$ the semidirect product $X_\nu \rtimes \RR$, and as $\Delta_\nu$ the operator $-\partial_u^2 + e^{2u} L_\nu$ thereon.

One of the technical problems in implementing this strategy is that the half-line $X_\nu$ is not a Lie group; however, it is a hypergroup \cite{BH,J} (more precisely, $X_\nu$ is known as a Bessel--Kingman hypergroup), so there is a convolution structure in terms of which the convolution kernels $K_{F(\sqrt{L_\nu})}$ are defined. Additionally, $X_\nu$ has a group of automorphic dilations,
so the semidirect product $G_\nu = X_\nu \rtimes \RR$ can be made sense of in terms of hypergroup theory \cite{HeyKa,W}, and again the operators in the functional calculus for $\Delta_\nu = -\partial_u^2 + e^{2u} L_\nu$ are convolution operators on $G_\nu$. Thus, one of the tasks that we undertake here is to develop the theory of the operators $L_\nu$ and $\Delta_\nu$, mirroring the classical one for $\Delta_{\RR^\nu}$ and $-\partial_u^2 + e^{2u} \Delta_{\RR^\nu}$ when $\nu$ is an integer, and establish all the properties that we need in order to run the aforementioned argument when $\nu$ is fractional (e.g., finite propagation speed, explicit formula for the control distance, radiality of convolution kernels in the functional calculus, Plancherel formula...). A fundamental tool for this is the theory developed in \cite{G}, which explicitly relates the heat semigroups generated by $L_\nu$ and $\Delta_\nu$, thus providing the link between their functional calculi that justifies the lifting of \eqref{eq:Nplancherel_inter2} to \eqref{eq:wplancherel_inter} and eventually yields the required estimates for the proof of Theorem \ref{thm:1}.

\subsection*{Structure of the paper}

In Section \ref{s:groups} we recall the main properties of the stratified groups $N$, their solvable extensions $G = N \rtimes \RR$, and the sub-Laplacians thereon. We also state and prove a suitable version of the weighted Plancherel estimate on $N$, extending those available in the literature, and show that the weight appearing in that estimate has the correct integrability properties on $G$.

Sections \ref{s:hypergroupsX} and \ref{s:hypergroupsG} are devoted to the discussion of the hypergroups $X_\nu$ and $G_\nu$, and the corresponding operators $L_\nu$ and $\Delta_\nu$. Of course, there is a vast literature on the Bessel operator $L_\nu$, so we only recall some of the main properties, in connection with the Hankel transform on $X_\nu$. In comparison, not so many results appear to be available for the operator $\Delta_\nu$ on the semidirect product hypergroup $G_\nu$, so we spend some time to derive the required properties. One of the issues we need to deal with is the fact that $X_\nu$ and $G_\nu$ are manifolds with boundary, and indeed boundary conditions play a role (at least for $\nu$ small) in the choice of self-adjoint extensions of $L_\nu$ and $\Delta_\nu$, which in turn determine their functional calculi.

By using the aforementioned hypergroup structures, in Section \ref{s:multipliers} we lift to $G$ the weighted Plancherel estimate on $N$ obtained in Section \ref{s:groups} and prove our main result, Theorem \ref{thm:1}.

Finally, we devote an Appendix (Section \ref{s:appendix}) to recall some terminology and results about self-adjoint extensions of divergence-form differential operators with Dirichlet and Neumann boundary conditions, as well as the finite propagation speed property for the associated wave equation, using the ``first-order approach'' from \cite{McIMo}.

\subsection*{Some remarks and open questions}

The technique developed here hinges on lifting weighted Plancherel estimates on $N$ of the form \eqref{eq:Nplancherel} to the semidirect product $G = N \rtimes \RR$. As already mentioned, similar estimates were originally obtained in \cite{He93,MuSt} as a tool to sharpen the Christ--Mauceri--Meda multiplier theorem on $N$. However, not all stratified groups $N$ are amenable to this approach. More recently \cite{Ma15,MM14} a different approach was developed, which applies to other classes of stratified groups and is based on variants of \eqref{eq:Nplancherel} where, roughly speaking, the right-hand side may also contain derivatives of $F$. This alternative approach can be used on a number of $2$-step stratified groups $N$, other than direct products of M\'etivier and abelian groups, to push down the condition in the Christ--Mauceri--Meda theorem to half the topological dimension. For those stratified groups $N$, one may expect that an analogous improvement of Theorem \ref{thm:mov} can be obtained on the corresponding semidirect product $G = N \rtimes \RR$, but the approach presented here does not directly apply and new ideas would appear to be needed.

Another natural question is the validity of $L^p$ estimates of Miyachi--Peral type \cite{Miy,P} for the wave equation associated with the sub-Laplacian $\Delta$ on $G$; indeed sharp estimates of this type would likely imply via subordination \cite{Mu} a sharp multiplier theorem for $\Delta$. Such estimates on $G$ are known \cite{MuTh} in the case $N$ is abelian (see also the extension in \cite{MuVa} for an elliptic Laplacian on Damek--Ricci spaces), but the case of nonabelian $N$ and nonelliptic $\Delta$ appears to be wide open. For the homogeneous sub-Laplacian $\Delta_N$ on $N$, sharp Miyachi--Peral estimates are known when $N$ is of Heisenberg type \cite{MuSe2}, and one may wonder whether similar estimates also hold for $\Delta$ on the corresponding solvable extension $G$.

\subsection*{Notation}
$\NN$ denotes the set of natural numbers, including $0$.
We write $\Rnon$ and $\Rpos$ for the closed and open half-lines $[0,\infty)$ and $(0,\infty)$.
For two nonnegative quantities $A$ and $B$, we write $A \lesssim B$ to denote that there exists a constant $C \in \Rpos$ such that $A \leq C B$. We also write $A \simeq B$ for the conjunction of $A \lesssim B$ and $B \lesssim A$. Subscripted variants such as $\lesssim_p$ or $\simeq_p$ indicate that the constant may depend on the parameter $p$.
Finally, we write $\ind_S$ for the characteristic function of a set $S$.

\section{Stratified Lie groups and their solvable extensions}\label{s:groups}

\subsection{Weighted Plancherel estimate for products of M\'etivier groups}\label{ss:metivier}

In this section our aim is to justify the weighted Plancherel estimate \eqref{eq:Nplancherel} that is at the core of our argument. Versions of this estimate can be found elsewhere in the literature (see, e.g., \cite[Lemma 1.4]{He93} or \cite[Proposition 3.5]{Ma11}), but they are stated and proved under the assumption that the multiplier $F$ is supported in a fixed compact subset of $\Rpos$. The proof that we present here shows that, in fact, this restriction on the support of $F$ can be dropped for certain homogeneous weights $w$ in the left-hand side of \eqref{eq:Nplancherel}, provided one chooses the appropriate homogeneity degree of the measure in the right-hand side. This version of the estimate without support restrictions turns out to be crucial in our proof of Theorem \ref{thm:1}.

Let $N$ be a $2$-step stratified Lie group. So the Lie algebra $\lie{n}$ of $N$ is the direct sum of two layers, $\lie{n}=\lie{n}_1\oplus\lie{n}_2$. If we denote by $d_1$ and $d_2$ the dimensions of $\lie{n}_1$ and $\lie{n}_2$, then $d=d_1+d_2$ and $Q=d_1+2d_2$ are the topological and homogeneous dimensions of $N$. Through the exponential map we can identify $N$ with $\lie{n}$; via this identification, Lebesgue measure on $\lie{n}$ is a left and right Haar measure on $N$. The same identification allows us to define the Schwartz class $\Sz(N)$ of rapidly decaying smooth functions on $N$.

We choose a system $\{X_j\}_{j=1}^{d_1}$ of left-invariant vector fields on $N$ which form a basis of $\lie{n}_1$,
and let $\Delta_N$ as in \eqref{eq:N_sublaplacian}
be the corresponding homogeneous sub-Laplacian on $N$. 
Since $\Delta_N$ with domain $C^\infty_c(N)$ is essentially self-adjoint on $L^2(N)$, a Borel functional calculus for $\Delta_N$ is defined via the spectral theorem. 
For any bounded Borel function $F : \Rnon \to \CC$, the $L^2$-bounded operator $F(\Delta_N)$ is left-invariant, so, by the Schwartz kernel theorem, there exists a convolution kernel $K_{F(\Delta_N)} \in \Sz'(N)$ such that
\begin{equation*}
	F(\Delta_N)f= f\ast K_{F(\Delta_N)},\qquad f\in \Sz(N).
\end{equation*}
As is well-known (see, e.g., \cite{CM,Mel,Si3} or Proposition \ref{prop:8} below), $\Delta_N$ satisfies the finite propagation speed property:
\begin{equation}\label{eq:fps_N}
\supp K_{\cos(t\sqrt{\Delta})} \subseteq \overline{B_N}(0_N,|t|), \qquad t \in \RR,
\end{equation}
where $\overline{B_N}(z,r)$ is the closed ball of centre $z$ and radius $r$ relative to the Carnot--Carath\'eodory distance associated with $\Delta_N$, and $0_N$ is the identity element of $N$.
Furthermore, a Plancherel formula holds for $\Delta_N$:
\begin{equation}\label{eq:plancherel_deltaN}
\|K_{F(\Delta)}\|_{L^2(N)}^2 \simeq \int_0^\infty |F(\lambda)|^2 \,\lambda^{Q/2} \ddlam
\end{equation}
(see, e.g., \cite{Ch,DeMa,HuJ}). This corresponds to the unweighted version of \eqref{eq:Nplancherel}, and holds without any assumptions on $N$.

Let $\langle\cdot,\cdot\rangle$ be the inner product on $\lie{n}_1$ which makes $\{X_j\}_{j=1}^{d_1}$ an orthonormal basis.
For all $\mu\in \lie{n}_2^\ast$ we define the skewsymmetric linear operator $J_\mu : \lie{n}_1\to\lie{n}_1$ by
\[
\langle J_\mu x,y\rangle = \mu[x,y]  \qquad \forall x,y\in\lie{n}_1.
\]

Let us choose linear coordinates on $\lie{n}_1$ and $\lie{n}_2$. So we shall write $z=(z',z'')\in N\simeq \RR^{d_1}\times \RR^{d_2}$, and also $\mu\in\lie{n}_2^\ast\simeq \RR^{d_2}$. Observe that $|J_\mu z'|^2$ is a polynomial in $(z',\mu)$, separately homogeneous of order $2$ in both $z'$ and $\mu$; the norm on $\lie{n}_1$ appearing in the expression $|J_\mu z'|^2$ is the one induced by the inner product $\langle \cdot,\cdot\rangle$.

Let $Z''=-i\nabla_{z''}=(-i\partial_{z''_1},\ldots,-i\partial_{z''_{d_2}} )$ be the vector of second-layer derivatives in $N$, and $P=(z'_1,\ldots,z'_{d_1})$ be the vector of multiplication operators by $z'_j$. The components of $P$ and $Z''$ together form a strongly commuting system of self-adjoint operators on $L^2(N)$ (cf.\ \cite[pp.~1229-1230]{Ma12}). Since $|J_\mu z'|^2$ is a polynomial, we consider $|J_{Z''} z'|^2 \defeq |J_{Z''} P|^2$ as a differential operator on $N$ of order $2$ with polynomial coefficients.

\begin{lm}\label{lm:9}
Let $f\in\Sz(N)$ be such that $f\ast K_{e^{-t\Delta_N}}= K_{e^{-t\Delta_N}}\ast f$ for all $t>0$. Then, for all $k\in\NN$,
\begin{equation*}
	\Vert |J_{Z''} P|^{2k} f \Vert_{L^2(N)} \lesssim_k \Vert \Delta_N^k f \Vert_{L^2(N)}.
\end{equation*}
\end{lm}
\begin{proof}
For every smooth differential operator $D$ on $N$ we define $D^\circ$ by
\begin{equation*}
(Df)^\ast = D^\circ f^\ast,
\end{equation*}
where $f^\ast(z)=\overline{f(z^{-1})}$ denotes the involution on $N$.
If $D$ is left-invariant, then $D^\circ$ is right-invariant, and vice versa.
By \cite[Lemma~3.4]{Ma12} we have
\begin{equation}\label{eq:23}
|J_{Z''} z'|^2 =\sum_{j=1}^{d_1} (X_j+X_j^\circ)^2.
\end{equation}
As the above expression contains a mixture of left- and right-invariant differential operators, it is convenient to analyse it by means of a lifting to the direct product $N \times N$; this idea has already been exploited in \cite{HeZi,Ma12}.

The Lie algebra of the direct product $N \times N$ is canonically identified with $\lie{n} \oplus \lie{n}$. In particular, we can think of $N \times N$ as a $2$-step group too, with first and second layers given by $\lie{n}_1 \oplus \lie{n}_1$ and $\lie{n}_2 \oplus \lie{n}_2$.
Let $\xi$ be the unitary representation of $N \times N$ on $L^2(N)$ defined by
\[
\xi(x,y) f(z) = f(y^{-1} z x).
\]
for all $x,y,z \in N$ and $f : N \to \CC$.
If $D$ is a left-invariant differential operator on $L^2(N \times N)$, then $\dd\xi(D)$ is a smooth differential operator on $N$. Moreover,
\begin{equation*}
\dd\xi(D^\bullet)=\dd\xi(D)^\circ
\end{equation*}
(see \cite[p.\ 1228]{Ma12}), where $D\mapsto D^\bullet$ is the correspondence on the algebra of left-invariant differential operators on $N \times N$ defined as the unique conjugate-linear automorphism extending the Lie algebra automorphism $(X,Y)\mapsto(Y,X)$ of $\lie{n}\oplus\lie{n}$.

Now, notice that $X_j =\dd\xi(\tilde{X}_j)$, $X^\circ_j =\dd\xi(\tilde{X}^\bullet_j)$, where $\tilde{X}_j=X_j\otimes\id$ is the lifting of $X_j$ to $N\times N$, and $\tilde{X}_j,\tilde{X}_j^\bullet$ are left-invariant $1$-homogeneous vector fields on the $2$-step group $N\times N$. In particular, by \eqref{eq:23}, for all $k \in \NN$,
\begin{equation}\label{eq:Joperator}
	|J_{Z''} P|^{2k} = \dd\xi(D_k),
\end{equation}
where 
\begin{equation*}
	D_k=\left(\sum_{j=1}^{d_1} (\tilde{X}_j+\tilde{X}^\bullet_j)^2\right)^k
\end{equation*}
is a left-invariant $2k$-homogeneous differential operator on $N\times N$.
	
Define $A=\frac{1}{2}(\Delta_N +\Delta_N^\circ)$. Notice that $A=\dd\xi(\tilde{A})$, where 
\begin{equation*}
	\tilde{A}= \frac{1}{2}(\tilde{\Delta}_N+\tilde{\Delta}_N^\bullet)
	=\frac{1}{2} (\Delta_N\otimes \id + \id\otimes\Delta_N )
\end{equation*}
is a $2$-homogeneous sub-Laplacian on the $2$-step group $N\times N$. In particular,
\begin{equation}\label{eq:Aoperator}
	A^k=\dd\xi(\tilde{A}^k),
\end{equation}
where $\tilde{A}^k$ is a $2k$-homogeneous positive Rockland operator on $N\times N$.
	
Hence, by the Helffer--Nourrigat theorem \cite{HelNo} (see also \cite[Theorem 2.5]{tElRo} and \cite[p.~32]{HeZi}), from \eqref{eq:Joperator} and \eqref{eq:Aoperator} we deduce, for any $k \in \NN$, the estimate
\begin{equation*}
	\Vert |J_{Z''} P|^{2k} f \Vert_{L^2(N)} 
	\lesssim_k \Vert A^k f \Vert_{L^2(N)}, \qquad f\in \Sz(N).
\end{equation*}
	
Finally, by \cite[Lemma~3.2]{Ma12}, for any $f\in\Sz(N)$ commuting with heat kernels associated with $\Delta_N$, we have $A^k f= \Delta_N^k f$.
\end{proof}

Notice that $\Delta_N$ and the components of $Z''$ are also a system of strongly commuting self-adjoint differential operators on $L^2(N)$. Actually, they form a \emph{homogeneous weighted subcoercive system}, or a \emph{Rockland system} (see \cite[Theorem 5.2]{Ma11} and \cite[Theorem 3.5]{Cal}). In particular, by \cite[Corollary 3.3]{Ma11}, $\Delta_N$ and $Z''$ admit a joint functional calculus on $L^2(N)$ and, for any bounded Borel function $H : \RR\times\RR^{d_2}\to\CC$, the $L^2(N)$-bounded operator $H(\Delta_N,Z'')$ is left-invariant. Moreover, there exists a Plancherel measure $\sigma$ for this joint functional calculus (see \cite[Theorem 3.10]{Ma11}), that is, a regular Borel measure on $\RR \times \RR^{d_2}$, supported on the joint spectrum (so in particular $\supp\sigma\subseteq \Rnon\times\RR^{d_2}$), such that
\begin{equation}\label{eq:jointplancherel}
\Vert K_{H(\Delta_N,Z'')} \Vert_{L^2(N)} = \int_{\RR\times\RR^{d_2}} |H(\lambda,\mu)|^2 \,\dd\sigma(\lambda,\mu).
\end{equation}
As $\Delta_N$ and the components of $Z''$ are $2$-homogeneous, \cite[Proposition 5.1]{Ma11} yields that, for all $t>0$,
\begin{equation*}
K_{H(t^2\Delta_N,t^2 Z'')}(z',z'')= t^{-Q}K_{H(\Delta_N,Z'')}(t^{-1}z',t^{-2}z'').
\end{equation*}
Consequently,
\begin{equation}\label{eq:25}
\int_{\RR\times\RR^{d_2}} H(t^2\lambda,t^2\mu)\,\dd\sigma(\lambda,\mu) = t^{-Q} \int_{\RR\times\RR^{d_2}} H(\lambda,\mu)\,\dd\sigma(\lambda,\mu).
\end{equation}

\begin{rem}\label{rem:2}
If $\varphi\in L^1_{\loc}(\RR^{d_2})$ is non-negative, then the push-forward via the projection $(\lambda,\mu)\mapsto\lambda$ of the measure $\varphi(\mu)\,\dd\sigma(\lambda,\mu)$ is a regular Borel measure $\sigma_\varphi(\lambda)$ on $\RR$, supported on $\Rnon$ (see \cite[Lemma 3.7]{Ma11}). Moreover, if $\varphi$ is homogeneous of degree $\omega\in\RR$, then for any $F\in C_c(\RR)$ and $t>0$ we have
\[\begin{split}
	\int_0^\infty F(t^2\lambda)\,\dd\sigma_\varphi(\lambda) 
	&= \int_{\RR\times\RR^{d_2}} F(t^2\lambda) \varphi(\mu)\, \dd\sigma(\lambda,\mu) \\
	&=t^{-Q}  \int_{\RR\times\RR^{d_2}} F(\lambda) \varphi(t^{-2}\mu)\, \dd\sigma(\lambda,\mu)\\
	&=t^{-Q-2\omega}  \int_{0}^\infty F(\lambda) \, \dd\sigma_\varphi(\lambda),
\end{split}\]
where the last equality follows from \eqref{eq:25}. Thus, the measure $\sigma_\varphi$ is homogeneous and there exists a positive constant $C_\varphi$ such that
\begin{equation*}
	\dd\sigma_\varphi(\lambda)=C_\varphi \lambda^{(Q+2\omega)/2}\,\ddlam.
\end{equation*}
In the case $\varphi \equiv 1$ we have $\omega = 0$, showing that \eqref{eq:jointplancherel} is consistent with \eqref{eq:plancherel_deltaN}.
\end{rem}

We remark that, by a multivariate extension of Hulanicki's theorem \cite{Hu} (see, e.g., \cite[Proposition 4.2.1]{MaPhD}), if $H\in\Sz(\RR\times\RR^{d_2})$, then $K_{H(\Delta_N,Z'')}\in \Sz(N)$; so any partial (Euclidean) Fourier transform of $K_{H(\Delta_N,Z'')}$ is also a Schwartz function. 

\begin{lm}\label{lm:10}
If $\gamma\geq 0$ and $H\in\Sz(\RR\times\RR^{d_2})$, then
\begin{equation*}
	\int_{\RR^{d_1}} \int_{\RR^{d_2}} | |J_\mu z'|^{2\gamma} \hat{K}_{H(\Delta_N,Z'')}(z',\mu)|^2 \,\dd z'\,\dd\mu
	\lesssim_\gamma \int_{\RR\times\RR^{d_2}} | H(\lambda,\mu)|^2 |\lambda|^\gamma \,\dd\sigma(\lambda,\mu),
\end{equation*} 
where $\hat{K}_{H(\Delta_N,Z'')}(z',\mu)$ is the partial Fourier transform in $z''$ of $K_{H(\Delta_N,Z'')}(z',z'')$.
\end{lm}
\begin{proof}
Clearly, $K_{H(\Delta_N,Z'')}$ commutes with the heat kernels $K_{e^{-t\Delta_N}}$, as the operators $H(\Delta_N,Z'')$ and $e^{-t\Delta_N}$ are in the joint functional calculus of $\Delta_N$ and $Z''$. Thus, assuming firstly that $\gamma\in\NN$, the Plancherel theorem in $\RR^{d_2}$ and Lemma \ref{lm:9} give
\[\begin{split}
	\int_{\RR^{d_1}} \int_{\RR^{d_2}} | |J_\mu z'|^{2\gamma}  \hat{K}_{H(\Delta_N,Z'')}(z',\mu)|^2 \,\dd z'\,\dd\mu
	&\simeq \int_{N} || J_{Z''} z'|^{2\gamma}  K_{H(\Delta_N,Z'')}(z)|^2 \,\dz \\
	&\lesssim_\gamma \int_{N}  |\Delta_N^\gamma K_{H(\Delta_N,Z'')}(z)|^2 \,\dz\\
	& = \int_{\RR\times\RR^{d_2}} |H(\lambda,\mu)|^2 |\lambda|^{2\gamma}\,\dd\sigma(\lambda,\mu).
\end{split}\]
By interpolation \cite{St} we extend the obtained inequality to non-integer $\gamma$.
\end{proof}

From now on we assume that $N$ satisfies the assumption of Theorem \ref{thm:1}, that is, $N$ is the direct product
\begin{equation}\label{eq:product_N}
N=N^{(0)}\times N^{(1)}\times\ldots\times N^{(\ell)}
\end{equation}
of an abelian group $N^{(0)}$ and finitely many M\'etivier groups $N^{(j)}$, $j=1,\ldots,\ell$, for some $\ell\in\NN \setminus \{0\}$. The factor $N^{(j)}$ is a $2$-step group for each $j=1,\dots,\ell$, so $\lie{n}^{(j)} = \lie{n}_1^{(j)} \oplus \lie{n}_2^{(j)}$, and we set $d_i^{(j)} = \dim \lie{n}_i^{(j)}$ for $i=1,2$. Since $N^{(0)}$ is abelian, it is a $1$-step group, so $\lie{n}^{(0)} = \lie{n}^{(0)}_1$, and we just set $d^{(0)} = \dim \lie{n}^{(0)}$.  We emphasise that with a proper interpretation one can allow $d^{(0)}=0$ and consider $N=N^{(1)}\times\ldots\times N^{(\ell)}$.

As $N$ is a direct product, the first and second layers of $\lie{n}$ are given by
\begin{equation}\label{eq:layer_dec}
\lie{n}_1 = \lie{n}^{(0)} \oplus \lie{n}_1^{(1)} \oplus \dots \oplus \lie{n}_1^{(\ell)}, \qquad \lie{n}_2 = \lie{n}_2^{(1)} \oplus \dots \oplus \lie{n}_2^{(\ell)}.
\end{equation}
In particular, if we define $\vec{d_i}=(d_i^{(1)},\ldots,d_i^{(\ell)})$ and $| \vec{d_i}|=\sum_{j=1}^\ell d_i^{(j)}$ for $i=1,2$, then the dimensions of the two layers of $\lie{n}$ are $d_1 = d^{(0)} + |\vec{d_1}|$ and $d_2 =|\vec{d_2}|$. Consequently, the homogeneous and topological dimensions of $N$ are $Q=d^{(0)}+ |\vec{d_1}| +2 |\vec{d_2}|$ and $d = d^{(0)}+ |\vec{d_1}| +|\vec{d_2}|$. 

We choose linear coordinates for all the layers $\lie{n}_i^{(j)}$, so we can identify $N^{(0)} \simeq \RR^{d^{(0)}}$ and $N^{(j)} \simeq \RR^{d_1^{(j)}}\times \RR^{d_2^{(j)}}$ as manifolds, for $1\leq j\leq \ell$, and correspondingly $N \simeq \RR^{d^{(0)}} \times \RR^{\vec d_1} \times \RR^{\vec d_2}$, where $\RR^{\vec{d_i}}\defeq \RR^{d_i^{(1)}}\times\ldots\times \RR^{d_i^{(\ell)}}$.

We emphasise that the vector fields $X_j$ on $N$ that define the sub-Laplacian $\Delta_N$ need not be compatible with the direct product decomposition. In other words, we do not require the sub-Laplacian $\Delta_N$ to be the sum of sub-Laplacians on the factors $N_j$, nor do we require the decomposition of $\lie{n}_1$ in \eqref{eq:layer_dec} to be an orthogonal decomposition with respect to the inner product $\langle \cdot, \cdot \rangle$ associated with $\Delta_N$.

For a multi-index $\al=(\al_1,\ldots,\al_\ell)\in\RR^\ell$ we write $|\al|=\al_1+\ldots+\al_\ell$ for its length. We also use the symbols $\prec,\preccurlyeq,\succ,\succcurlyeq$ to denote componentwise inequalities between multi-indices: for instance we write $\al\preccurlyeq \be$ for $\al,\be\in\RR^\ell$ whenever $\al_j\leq\be_j$ for all $j=1,\ldots,\ell$. Also, we denote $\vec{0}=(0,\ldots,0)\in\RR^\ell$. Unless stated otherwise, all the multi-indices we use have $\ell$ components.

The above structural assumption on $N$ yields the following crucial estimate (cf., e.g., \cite[Proposition 3.9]{Ma12}).

\begin{lm}\label{lm:metivier}
For all $z'=(z'_0,z'_1\ldots,z'_\ell) \in \lie{n}_1$, $\mu=(\mu_1,\ldots,\mu_\ell) \in \lie{n}_2$,
\begin{equation}\label{eq:metivierproduct}
|J_\mu z'|\gtrsim \sum_{j=1}^\ell |\mu_j| |z'_j|.
\end{equation}
In particular, for any $\al \succeq \vec{0}$,
\begin{equation}\label{eq:24}
|J_\mu z'|^{\val}\gtrsim_\al \prod_{j=1}^\ell (|\mu_j||z_j'|)^{\al_j}.
\end{equation}
Moreover,
\begin{equation}\label{eq:metivier_dim}
\vec d_2 \prec \vec d_1.
\end{equation}
\end{lm}
\begin{proof}
As $N$ is a direct product,
\[
\langle J_\mu x,y\rangle = \mu[x,y] = \sum_{j=1}^\ell \mu_j [x_j,y_j]
\]
for all $x=(x_0,x_1,\dots,x_\ell),y=(y_0,y_1,\dots,y_\ell) \in \lie{n}_1$, $\mu=(\mu_1,\ldots,\mu_\ell) \in \lie{n}_2^*$. Consequently,
\begin{equation}\label{eq:norm_equivalence}
|J_\mu x| \simeq |\mu[x,\cdot]| \simeq \sum_{j=1}^\ell |\mu_j [x_j,\cdot]| \gtrsim \sum_{j=1}^\ell |\mu_j| |x_j|,
\end{equation}
where $\mu[x,\cdot]$ and $\mu_j [x_j,\cdot]$ are thought of as elements of $\lie{n}_1^*$ and $(\lie{n}^{(j)}_1)^*$ respectively, and in the last inequality of \eqref{eq:norm_equivalence} we used that each of the $N_j$ is M\'etivier. This proves \eqref{eq:metivierproduct}, and \eqref{eq:24} follows immediately.

Finally, the (well-known) inequality $d_2^{(j)} < d_1^{(j)}$ between the dimensions of the layers of a M\'etivier group is due to the fact that, for any nonzero $x \in \lie{n}_1^{(j)}$, the linear map $[x,\cdot] : \lie{n}_1^{(j)} \to \lie{n}_2^{(j)}$ is surjective (due to the M\'etivier condition), but not injective, as its kernel contains $x$.
\end{proof}

By combining Lemmas \ref{lm:10} and \ref{lm:metivier}, we can finally obtain the desired weighted Plancherel estimate on $N$.

\begin{prop}\label{prop:3}
For $\vec{0}\preccurlyeq\al\prec \vec{d_2}$ and all $F\in\Sz(\RR)$,
\begin{equation*}
	\int_N  |K_{F(\Delta_N)}(z)|^2\prod_{j=1}^\ell |z_j'|^{\al_j} \,\dz 
	\lesssim_{\alpha} \int_0^\infty |F(\lambda)|^2 \lambda^{(Q-\val)/2}\, \ddlam. 
\end{equation*}
\end{prop}
\begin{proof}
For any $H\in\Sz(\RR\times\RRtwo)$ and $\al\succcurlyeq \vec{0}$, Lemma \ref{lm:10} and \eqref{eq:24} yield
\[\begin{split}
	&\int_{\RR\times\RRtwo} | H(\lambda,\mu)|^2 |\lambda|^{\val/2} \,\dd\sigma(\lambda,\mu)\\
	&\gtrsim_\al \int_{\RR^{d^{(0)}}\times\RRone} \int_{\RRtwo}  |\hat{K}_{H(\Delta_N,Z'')}(z',\mu)|^2  \prod_{j=1}^{\ell }(|\mu_j | |z'_j |)^{\al_j}  \,\dd\mu \,\dd z'\\
	&=\int_{N}  |Z''_1 |^{\al_1}\ldots |Z''_\ell|^{\al_\ell} | K_{H(\Delta_N,Z'')}(z)|^2 \prod_{j=1}^{\ell }|z'_j |^{\al_j}  \,\dz\\
	&=\int_{N}   | K_{|Z''_1 |^{\al_1}\ldots |Z''_\ell|^{\al_\ell} H(\Delta_N,Z'')}(z)|^2 \prod_{j=1}^{\ell } |z'_j |^{\al_j} \,\dz.
\end{split}\]
	
Let $F\in\Sz(\RR)$. We apply the above for $H(\lambda,\mu)=F(\lambda)\chi(\mu) \prod_{j=1}^\ell|\mu_j|^{-\al_j/2}$ where $\chi(\mu)=\prod_{j=1}^\ell\chi_j(\mu_j)$ for some $\chi_j\in C_c^\infty(\RR^{d^{(j)}_2}\setminus\{0\})$, obtaining	
\begin{multline}\label{eq:26}
	\int_{N}    | K_{ F(\Delta_N)\chi(Z'')}(z)|^2 \prod_{j=1}^{\ell } |z'_j |^{\al_j} \,\dz \\
	\lesssim_\alpha \Vert\chi\Vert_{L^\infty}^2 
	\int_{\RR\times\RRtwo} | F(\lambda)|^2 |\lambda|^{\val/2} \prod_{j=1}^{\ell } |\mu_j |^{-\al_j} \,\dd\sigma(\lambda,\mu).
\end{multline}
	
Now we apply Remark \ref{rem:2} to $\varphi(\mu) =\prod_{j=1}^{\ell } |\mu_j |^{-\al_j} $ and denote the corresponding measure $\sigma_\varphi$ by $\sigma_\al$. We restrict $\al\prec \vec{d_2}$, so that $\varphi$ is locally integrable. Since $\varphi$ is homogeneous of degree $-\val$, we obtain $\sigma_\al(\lambda) = C_\al \lambda^{(Q-2\val)/2}\,\ddlam$.
	
Hence, \eqref{eq:26} gives
\begin{equation}\label{eq:27}
	\int_{N}  | K_{ F(\Delta_N)\chi(Z'')}(z)|^2  \prod_{j=1}^{\ell } |z'_j |^{\al_j}\,\dz
	\lesssim_\alpha \Vert\chi\Vert_{L^\infty}^2 \int_0^\infty |F(\lambda)|^2 \lambda^{(Q-\val)/2}\,\ddlam.
\end{equation}
	
Now we choose a sequence $\chi^{(n)} = \chi^{(n)}_1 \otimes \dots \otimes \chi^{(n)}_\ell$ of smooth cutoffs such that $0\leq \chi^{(n)}\leq 1$ and $\chi_n$ converges monotonically to $1$ on $\RRtwo \setminus A$, where
\begin{equation*}
	A\defeq (\{0\}\times \RR^{d_2^{(2)}}\times\ldots\times \RR^{d_2^{(\ell)}})\cup\ldots\cup( \RR^{d_2^{(1)}}\times\ldots\times\RR^{d_2^{(\ell-1)}}\times\{0\}).
\end{equation*}
We claim that $\sigma(\Rnon \times A)=0$. Indeed, for $b>a>0$ we have 
\[
	\sigma( (a,b)\times A) = \Vert K_{\ind_{(a,b)\times A}(\Delta_N,Z'')}\Vert_{L^2(N)} 
	=\Vert \ind_{(a,b)}(\Delta_N)\ind_A(Z'')\Vert_{L^1(N)\to L^2(N)} = 0,
\]
because $\ind_A(Z'')$ vanishes: in the chosen coordinates on $N$, the operator $\ind_A(Z'')$ is just the Euclidean Fourier multiplier operator on $\RR^{d_1} \times \RR^{d_2}$ with symbol $\ind_{\RR^{d_1} \times A}$, which vanishes almost everywhere.
Hence, $F(\lambda)\chi^{(n)}(\mu)\to F(\lambda)$ in $L^2(\dd\sigma)$, by the dominated convergence theorem. Thus $K_{F(\Delta_N)\chi^{(n)}_n(Z'')}\to K_{F(\Delta_N)}$ in $L^2(N)$ by the Plancherel formula \eqref{eq:jointplancherel}, and, up to extracting a subsequence, also almost everywhere on $N$. 
	
Finally, by Fatou's lemma and \eqref{eq:27} we arrive at
\[\begin{split}
	  \int_N |K_{F(\Delta_N)}(z)|^2  \prod_{j=1}^{\ell }|z'_j |^{\al_j} \,\dz 
		&\leq \liminf_{n\to\infty} \int_N  |K_{F(\Delta_N)\chi^{(n)}(Z'')}(z)|^2 \prod_{j=1}^{\ell } |z'_j |^{\al_j}\,\dz \\
	&\lesssim_\alpha \int_0^\infty |F(\lambda)|^2 \lambda^{(Q-\val)/2}\, \ddlam,
\end{split}\]
as desired.
\end{proof}

\subsection{The semidirect product extension}

As in the introduction, we set $G=N \rtimes\RR$, where $\RR$ acts on $N$ via automorphic dilations.
The left-invariant vector fields $X_j$, $j=1\ldots, d_1$, forming a basis of the first layer of $N$, and the standard basis $X_0 = \partial_u$ of the Lie algebra of $\RR$, can be lifted, as in \eqref{eq:liftedvectorfields}, to left-invariant vector fields $\{X^\sharp_j\}_{j=0}^{d_1}$ on $G$, which generate the Lie algebra of $G$. As in \eqref{eq:G_sublaplacian}, let $\Delta$
be the corresponding left-invariant sub-Laplacian on $G$. 
	
The group $G$ and the sub-Laplacian $\Delta$ are among those considered in \cite{MOV}, to which we refer for a more extensive discussion. Here we just recall that, much as in the case of $\Delta_N$ discussed earlier, the operator $\Delta$ on the domain $C^\infty_c(G)$ is essentially self-adjoint on $L^2(G)$, and for any bounded Borel function $F : \Rnon\to\CC$ we define $F(\Delta)$ via the spectral theorem. Since it is left-invariant, by the Schwartz kernel theorem there exists a (possibly distributional) convolution kernel $K_{F(\Delta)}$ such that
\begin{equation*}
	F(\Delta)f= f\ast K_{F(\Delta)},\qquad f\in C^\infty_c(G).
\end{equation*}
Moreover, a Plancherel formula holds for $\Delta$ \cite[Corollary 4.6]{MOV}, which says that
\begin{equation*}
	\Vert K_{F(\Delta)}\Vert_{L^2(G)}^2 \simeq \int_0^\infty |F(\lambda)|^2 \,\lambda^{[3/2,(Q+1)/2]} \,\ddlam,
\end{equation*}
where we use the notation \eqref{eq:bipower}.
Next, if $|z|_N$ denotes the Carnot--Carath\'eodory distance of $z \in N$ from the identity element $0_N$, then, by \cite[Proposition 2.7]{MOV}, the distance of $(z,u) \in G$ from the identity $0_G$ on $G$ is given by
\begin{equation}\label{eq:Gdistance}
|(z,u)|_G=\arccosh\left(\cosh u +\frac{|z|_N^2}{2e^{u}} \right).
\end{equation}
Furthermore, we have finite propagation speed for $\Delta$:
\begin{equation}\label{eq:fps_G}
\supp K_{\cos(t\sqrt{\Delta})} \subseteq \overline{B_G}(0_G,|t|), \qquad t \in \RR,
\end{equation}
where $\overline{B_G}(0_G,r)$ is the closed ball of radius $r$ centred at the identity of $G$ with respect to the Carnot--Carath\'eodory distance.

All the above results hold without any assumptions on the $2$-step group $N$. We now restrict to the case where $N$ is a product of abelian and M\'etivier groups as in \eqref{eq:product_N}.
Recall the notation $z=(z',z'')\in N$, $z'=(z'_0,z'_1,\ldots,z'_\ell)$ from Section \ref{ss:metivier}.
The following results are variations of \cite[Proposition 2.9]{MOV}, where we also include the reciprocals of the weights appearing in Proposition \ref{prop:3}.

\begin{lm}\label{lm:2}
Let $f,g : \Rnon \to \Rnon$ be measurable functions, and $\vec{0}\preccurlyeq\al\prec \vec{d_1}$. Then
\begin{multline*}
	\int_G  f(|(z,u)|_G) g(|z|_N ) \prod_{j=1}^{\ell } |z'_j |^{-\al_j}\, \dz\, \du \\
	= C_{N,\al} \int_\RR\int_0^\infty f( \arccosh(\cosh u + e^{-u}t/2)) g(\sqrt{t}) \, t^{\frac{Q-\val}2} \,\ddt\,\du
\end{multline*}
for a suitable constant $C_{N,\alpha} \in \Rpos$.
\end{lm}
\begin{proof}
Let $S=\{z\in N: |z|_N=1 \}$. Then there exists a Borel measure  $\tau$ on $S$ such that
\begin{equation}\label{eq:polar}
\int_N h(z) \,\dz = \int_0^\infty \int_S h(t\omega) \,t^Q \, \ddt \,\dd\tau(\omega)
\end{equation}
(see, e.g., \cite[Proposition 1.15]{FoSt}).
Hence, by \eqref{eq:Gdistance},
\begin{multline}\label{eq:polar_computation}
	\int_G  f(|(z,u)|_G) g(|z|_N) \prod_{j=1}^{\ell } |z'_j |^{-\al_j}\, \dz\, \du\\
	= C_N \int_\RR\int_S \int_0^\infty f(\arccosh(\cosh u + e^{-u} |t|^2/2)) g(t) t^Q \prod_{j=1}^\ell(t|\omega_j'|)^{-\al_j} \, \ddt \,\dd\tau(\omega) \,\du.
\end{multline} 
Since $\al\prec \vec{d_1}$,  the function $h(z) = \ind_{[1,2]}(|z|_N) \prod_{j=1}^\ell |z_j'|^{-\alpha}$ is integrable on $N$, so from \eqref{eq:polar} we also deduce that $\int_S \prod_{j=1}^\ell |\omega_j'|^{-\al_j}\,\dd\tau(\omega)<\infty$. A simple change of variables in the right-hand side of \eqref{eq:polar_computation} then completes the proof.
\end{proof}

\begin{cor}\label{cor:3}
Let $\vec{0}\preccurlyeq\al\prec \vec{d_1}$. If $r\in(0,1]$, then
\begin{equation*}
	\int_{B_G(0_G,r)} \prod_{j=1}^{\ell } |z'_j |^{-\al_j}\, \dz\, \du \lesssim_\al r^{Q-\val+1}.
\end{equation*}
If $r\geq 1$, then
\begin{equation*}
	\int_{B_G(0_G,r)} (1+|z|_N^{Q-\val})^{-1} \prod_{j=1}^{\ell } |z'_j |^{-\al_j}\, \dz\, \du 
	\lesssim_\al r^{2}.
\end{equation*}
\end{cor}
\begin{proof}
In the first case we apply Lemma \ref{lm:2} with $f=\ind_{[0,r]}$ and $g\equiv 1$:
\[\begin{split}
	&\int_{B_G(0_G,r)} \prod_{j=1}^{\ell } |z'_j |^{-\al_j}\, \dz\, \du \\
	&\simeq_\al \int_\RR \int_0^\infty \ind_{[0,r]}( \arccosh(\cosh u + e^{-u}t/2)) \, t^{\frac{Q-\val}2} \ddt\,\du\\
	&= \int_{-r}^r \int_0^{2e^u(\cosh r - \cosh u)} t^{\frac{Q-\val}2} \ddt\,\du\\
	& \leq \int_{-r}^r \int_0^{8e^u (r+u)(r-u)} t^{\frac{Q-\val}2} \ddt\,\du\\
	&\simeq_\alpha \int_{0}^{2r} ((2r-v)v)^{\frac{Q-\val}{2}}\,\dv \simeq_\alpha r^{Q-\val+1},
\end{split}\]
where we used that $|u| \leq r \leq 1$ in the domain of integration.
	
In the second case we again use Lemma \ref{lm:2}, this time with  $f(x)=\ind_{[0,r]}(x)$ and $g(x)= (1+x^{Q-\val})^{-1}$. Thus,
\[\begin{split}
	&\int_{B_G(0_G,r)} \frac{\prod_{j=1}^{\ell } |z'_j |^{-\al_j}}{1+|z|_N^{Q-\val}} \, \dz\, \du \\
	&\simeq_\alpha \int_\RR \int_0^\infty \ind_{[0,r]}(\arccosh(\cosh u + e^{-u}t/2))  \frac{t^{\frac{Q-\val}2}}{1+t^{\frac{Q-\val}2}}\, \ddt\,\du \\
	&\leq \int_{-r}^r \int_0^{2e^{u+r}} \frac{t^{\frac{Q-\val}2}}{1+t^{\frac{Q-\val}2}}\,\ddt\,\du \\
	&\leq \int_{-r}^r \int_0^{1} t^{\frac{Q-\val}2}\,\ddt\,\du + \int_{-r}^r \int_{1}^{2e^{u+r}} \,\ddt\,\du \\
	&\simeq_\alpha \int_{-r}^{r} (1+u+r) \,\du \simeq r^2,
\end{split}\]
as $r \geq 1$.
\end{proof}

\section{The Bessel--Kingman hypergroup}\label{s:hypergroupsX}

\subsection{Bessel--Kingman hypergroup and Hankel transform}

We gather here some known facts about the Hankel convolution and the Hankel transform. 
Further details can be found, e.g., in 
\cite{GoSt,Guy,Hai,Hir,T}
and references therein.

Let $\nu\geq 1$ and denote $X_\nu=(\Rnon,\mu_\nu)$, where $\dmu(x) = x^{\nu-1} \,\dd x$. The \emph{Hankel convolution} of suitable functions $f,g : X_\nu\to\CC$ is given by
\begin{equation}\label{eq:7}
f\ast_\nu g(x)=\int_0^\infty \tau_\nu^{[x]} f(y) \, g(y)\,\dmu(y) 
= \int_0^\infty  f(y) \, \tau_\nu^{[x]} g(y)\,\dmu(y),
\end{equation}
where $\tau_\nu^{[x]}f$ is the \emph{Hankel translation} of $f$,
\begin{multline*}
\tau_\nu^{[x]}f(y) = \tau_\nu^{[y]}f(x) \\
=\begin{cases}
\Beta(\frac{\nu-1}{2},\frac{1}{2})^{-1} \int_0^\pi f(\sqrt{x^2+y^2-2xy\cos\omega}) \sin^{\nu-2}\omega\,\dd\omega &\text{if } \nu > 1,\\
\frac{1}{2} [f(x+y)+f(|x-y|)] &\text{if } \nu = 1,
\end{cases}
\end{multline*}
for all $x,y \in X_\nu$, and $\Beta$ is the Beta function. The Hankel translations are normalised so that if $f\equiv 1$, then $\tau_\nu^{[x]} f\equiv 1$. Moreover, the $\tau_\nu^{[x]}$ commute pairwise, are self-adjoint on $L^2(X_\nu)$ and contractions on $L^p(X_\nu)$, $p\in[1,\infty]$, for all $x \geq 0$.

If $\nu$ is an integer, then the Hankel convolution $*_\nu$ corresponds to the restriction of the Euclidean convolution $*$ on $\RR^\nu$ to radial functions; more precisely, if $C_\nu=2\pi^{\nu/2}/\Gamma(\nu/2)$ denotes the surface measure of the unit sphere in $\RR^\nu$, then
\begin{equation}\label{eq:conv_conv}
f\ast_\nu g(|z|) = C_\nu^{-1} F\ast G(z), 
\end{equation}
where $F,G : \RR^\nu\to\CC$ are radial and $F(z)=f(|z|)$, $G(z)=g(|z|)$.  Nonetheless, the convolution $\ast_\nu$ is well defined for non-integer $\nu$ as well.

For an arbitrary $\nu \geq 1$, the convolution $\ast_\nu$ is commutative, associative and satisfies Young's inequality, that is
\begin{equation}\label{eq:young_Xnu}
	\Vert f\ast_\nu g\Vert_{L^r(X_\nu)} \leq \Vert f\Vert_{L^p(X_\nu)} \Vert g\Vert_{L^q(X_\nu)},
\end{equation}
where $p,q,r \in [1,\infty]$ and $1+1/r=1/p+1/q$ \cite[Theorem 2b]{Hir}.

The Hankel convolution can be extended to measures; as usual, here we identify $f \in L^1(X_\nu)$ with the measure $f \dmu$. In particular, if $\delta_x$ denotes the Dirac delta measure at $x \in X_\nu$, then $\delta_x *_\nu \delta_y$ is the probability measure on $X_\nu$ given by
\begin{equation}\label{eq:hyperdelta}
\delta_x \ast_\nu \delta_y(A) 
= \begin{cases}
\Beta(\frac{\nu-1}{2},\frac{1}{2})^{-1} \int_0^\pi \ind_A(\sqrt{x^2+y^2-2xy\cos\omega}) \sin^{\nu-2}\omega\,\dd\omega &\text{if } \nu > 1,\\
\frac{1}{2}(\delta_{x+y}+\delta_{|x-y|})(A) &\text{if } \nu =1,
\end{cases}
\end{equation}
for all Borel subsets $A$ of $X_\nu$, and has support in $[|x-y|,x+y]$. Moreover, Hankel translations can be expressed as convolution operators:
\[
\tau^{[x]}_\nu f = \delta_x *_\nu f.
\]
Equipped with this convolution structure, $X_\nu$ is a \emph{hypergroup} \cite{BH,J}, known as a \emph{Bessel--Kingman hypergroup} \cite{Ki}. It is a commutative hypergroup, whose unit element is $\delta_0$, whereas the involution is the identity (i.e.\ $X_\nu$ is a hermitian hypergroup) and the Haar measure is $\mu_\nu$. The hypergroup $X_\nu$ is a particular case of a \emph{Sturm--Liouville hypergroup} \cite{Zeu}, and more specifically a \emph{Ch\'ebli--Trim\`eche hypergroup} \cite{AcTr,Che}, associated with the Bessel operator $L_\nu$ we discuss in Section \ref{ss:bessel_op}.

The \emph{(modified) Hankel transform} $H_\nu$ on $X_\nu$ is given by
\begin{equation}\label{eq:hankeltr}
H_\nu f(x)=\int_0^\infty f(y) \, j^\nu_x (y)\,\dmu(y),\qquad x\geq 0, \qquad f\in L^1(X_\nu),
\end{equation}
where
\begin{equation}\label{eq:bessel}
j_x^\nu(y)=j_y^\nu(x) \defeq j^\nu(xy), \qquad j^\nu(t) \defeq \kappa_\nu \frac{J_{\nu/2 -1}(t)}{t^{\nu/2 -1}}, \qquad \kappa_\nu \defeq 2^{\nu/2 -1}\Gamma(\nu/2)
\end{equation}
and $J_s$ denotes the Bessel function of the first kind and order $s$. We point out that $j^\nu$ is an even analytic function, satisfying
\begin{equation}\label{eq:bessel_norm}
|j^\nu(t)| \leq 1 \qquad\forall t \in \RR,
\end{equation}
as $\nu \geq 1$ \cite[eq.\ (10.14.4)]{DLMF}.

For integer $\nu$ the Hankel transform is related to the Fourier transform: indeed, if $\nu \in \NN$, then, for radial functions $F:\RR^\nu\to\CC$, denoting $F(z)=f(|z|)$, we have
\begin{equation*}
H_\nu f(|\xi|)=  C_\nu^{-1} \Four F(\xi),
\end{equation*}
where $C_\nu$ is as in \eqref{eq:conv_conv} 
and $\Four F$ stands for the Fourier transform of $F$ in $\RR^\nu$,
\begin{equation*}
\Four F(\xi)=\int_{\RR^\nu} F(z) \, e^{-iz\cdot\xi}\,\dz.
\end{equation*}
We point out that many works in the literature use the order $\nu/2-1$ of the Bessel function in \eqref{eq:hankeltr} in place of $\nu$ as a parameter for the corresponding Hankel transform and Bessel--Kingman hypergroup; as in \cite{GaSe}, we prefer our choice of the parameter as it is immediately linked to the dimension of the related Euclidean space when $\nu$ is integer.

For an arbitrary $\nu \geq 1$, by \eqref{eq:bessel_norm} it follows that the Hankel transform maps $L^1(X_\nu)$ into $C \cap L^\infty(X_\nu)$. Moreover, the functions $j_x^\nu$ satisfy
\begin{equation}\label{eq:8}
\tau_\nu^{[y]} j_x^\nu (z)= j_x^\nu(y) \, j_x^\nu(z)
\end{equation}
(cf.\ \cite[Section 11.41]{Wa}), whence one deduces that
\begin{equation}\label{eq:trans_modul}
H_\nu(\tau_\nu^{[z]} f) = j_z^\nu H_\nu f, \qquad H_\nu (j_z^\nu f) = \tau_\nu^{[z]} H_\nu f.
\end{equation}
In other words, the Hankel transform intertwines Hankel translations with certain multiplication operators, which we can think of as analogues of modulations.

Let $\Sz_e(\Rnon)$ be the space of restrictions of even Schwartz functions on $\RR$ to $\RR_+$. Then $\Sz_e(\Rnon)$ is closed under Hankel translations and convolution, and the Hankel transform maps $\Sz_e(\Rnon)$ onto itself (cf.\ \cite{CCTV,Stk,T}).
Moreover, the Hankel transform extends to an isomorphism $H_\nu : L^2(X_\nu)\to L^2(X_\nu)$ satisfying 
\begin{equation}\label{eq:hankel_inv}
H_\nu^{-1}= \kappa_\nu^{-2} H_\nu
\end{equation}
(cf.\ \cite[Section 14.4]{Wa}) and such that Parseval's theorem holds, namely
\begin{equation*}
	\int_0^\infty H_\nu f(x) \, \overline{H_\nu g(x)}\,\dmu(x) = \kappa_\nu^2 \int_0^\infty f(y) \, \overline{g(y)}\,\dmu(y)
\end{equation*}
for all $f,g\in L^2(X_\nu)$ (cf.\ \cite{Mac}).
In particular, 
\begin{equation}\label{eq:16}
\Vert H_\nu f\Vert_{L^2(X_\nu)} = \kappa_\nu \Vert f\Vert_{L^2(X_\nu)}.
\end{equation}
By interpolating this fact with the boundedness of $H_\nu : L^1(X_\nu)\to L^\infty(X_\nu)$ we get the Hausdorff--Young inequality
\begin{equation}\label{eq:hankel_hy}
	\Vert H_\nu f\Vert_{L^{p'}(X_\nu)}\lesssim_{\nu,p} \Vert f\Vert_{L^{p}(X_\nu)},\qquad f\in L^{p}(X_\nu),
\end{equation}
where $p\in [1,2]$ and $1/p+1/p'=1$.

Observe that, by \eqref{eq:8},
\begin{equation}\label{eq:conv_prod}
H_\nu( f\ast_\nu g) = H_\nu f \cdot H_\nu g
\end{equation}
for $f,g\in L^1(X_\nu)$ \cite[Theorem 2d]{Hir}. In fact, by \eqref{eq:young_Xnu} and \eqref{eq:hankel_hy} this can be extended to the case where $f\in L^p(X_\nu)$ and $g\in L^q(X_\nu)$, provided $p,q\in [1,2]$ satisfy $1/p+1/q\in[3/2,2]$.

\subsection{The Bessel operator}\label{ss:bessel_op}

Let $\mathring{X}_\nu=(\Rpos,\mu_\nu)$ be the interior of $X_\nu$. We now introduce the \emph{Bessel operator} $L_\nu$, that is the second order differential operator with smooth coefficients on $\mathring{X}_\nu$ given by
\begin{equation*}
L_\nu = -\partial_x^2 - \frac{\nu-1}{x} \partial_x =D_\nu^+ D_\nu,
\end{equation*}
where $D_\nu=\partial_x$ and $D_\nu^+$ is its formal adjoint with respect to $\mu_\nu$. Notice that, although $D_\nu$ does not depend on $\nu$, the adjoint does.  If $\nu\in\NN$, then $L_\nu$ is the radial part of the standard Laplacian in $\RR^\nu$.

As the Bessel operator $L_\nu = D_\nu^+ D_\nu$ is a smooth divergence-form second-order differential operator on the manifold $\mathring{X}_\nu$, we can discuss its self-adjoint extensions on $L^2(\mathring{X}_\nu)$ with the language introduced in the Appendix, to which we refer for the definitions of Dirichlet and Neumann domains used here.
On the other hand, as $X_\nu$ and $\mathring{X}_\nu$ differ by a null set, we shall identify $L^2(\mathring{X}_\nu) = L^2(X_\nu)$, and consider $L_\nu$ as an unbounded operator on $L^2(X_\nu)$.

Observe that $L_\nu$ is symmetric and positive on the domain $\Sz_e(\RR_+)$, which $L_\nu$ maps into itself. Moreover, it can be checked \cite[eq.\ (10.13.4)]{DLMF} that 
\[
L_\nu j^\nu_x = x^2 j_x^\nu.
\]
Thus, for $f\in \Sz_e(\RR_+)$,
\begin{equation}\label{eq:1}
H_\nu (L_\nu f)(x) = x^2 H_\nu f(x).
\end{equation}
This easily allows us to define a self-adjoint extension of $L_\nu$ that is compatible with the Hankel transform.

\begin{lm}\label{lm:11}
The operator $L_\nu$ is essentially self-adjoint on $\Sz_e(\RR_+)$, and the identity \eqref{eq:1} holds for any $f$ in the domain of the self-adjoint extension and almost all $x \in X_\nu$. Moreover, this self-adjoint extension is equal to the operator $D_\nu^+ D_\nu$ with Neumann domain.
\end{lm}

This result is certainly known to experts. We include some details of the proof as they will be useful in Section \ref{ss:liftedbessel} below, when discussing the ``lifting'' of the Bessel operator to the semidirect product hypergroup.

\begin{proof}
The formula \eqref{eq:1} shows that $H_\nu$ intertwines $L_\nu$ on the domain $\Sz_e(\Rnon)$ with the multiplication operator by $M(x) = x^2$. The latter is self-adjoint on the natural domain
\[
	\{ f \in L^2(X_\nu) \tc M f \in L^2(X_\nu)\},
\]
and $\Sz_e(\Rnon)$ is clearly dense in this domain with respect to the graph norm. By the invariance of $\Sz_e(\RR_+)$ with respect to $H_\nu$, we conclude that $L_\nu$ is essentially self-adjoint on $\Sz_e(\RR_+)$, and that the intertwining relation \eqref{eq:1} remains true on the domain of the self-adjoint extension.
	
In light of the definition \eqref{eq:Neu} of the Neumann domain $\DNeu(D_\nu^+ D_\nu)$, in order to justify the remaining part of the statement, it suffices to verify that, if $f\in \Sz_e(\RR_+)$, then $D_\nu f \in \Dmin(D_\nu^+)$. As $D_\nu f = f' \in \Sz_o(\RR_+)$, the space of restrictions of odd Schwartz functions on $\RR$ to $\RR_+$, it is enough to show that any element of $g \in \Sz_o(\RR_+)$ can be approximated in the graph norm of $D_\nu^+$ by a sequence of functions in $C^\infty_c(\Rpos)$.

Let $g\in \Sz_o(\Rnon)$. Then $D_\nu^+g(x) = -g'(x)-(\nu-1)g(x)/x$, and both $g'(x)$ and $g(x)/x$ are in $\Sz_e(\Rnon)$.
Take smooth functions $\psi,\phi:\Rnon \to[0,1]$ such that $\psi(x) = 1$  for $x\geq 2$, $\psi(x) = 0$ for $x \leq 1$, $\phi(x) = 1$ for $x\leq 2$, and $\phi(x) = 0$ for $x\geq 4$.
We set $\psi_n(x)=\psi(nx)$, $\phi_n(x)=\phi(x/n)$, and define the approximating sequence $g_n= g\psi_n\phi_n \in C^\infty_c(\Rpos)$.

Clearly, $g_n\to g$ and $g_n(x)/x \to g(x)/x$ in $L^2(X_\nu)$, by dominated convergence. Moreover,
\[
g_n'(x) = \psi(nx) \phi(x/n) g'(x) +  [nx \psi'(nx)] \phi(x/n) g(x)/x + \psi(nx) \phi'(x/n)  g(x)/n.
\]
Again, the first summand tends to $g'$ in $L^2(X_\nu)$ by dominated convergence, while the second and third summands tend to $0$; for the second summand, we use that $g(x)/x$ is in $L^2(X_\nu)$, while the multiplying factor is in $L^\infty(X_\nu)$ uniformly in $n$ and vanishes outside $[1/n,2/n]$. Thus $g_n' \to g'$ in $L^2(X_\nu)$, which combined with the previous observations gives that $D_\nu^+ g_n \to D_\nu^+ g$ in $L^2(X_\nu)$ and $g_n \to g$ in $\Dmax(D_\nu^+)$.
\end{proof}

\begin{rem}
Lemma \ref{lm:11} remains true if $\Sz_e(\RR_+)$ is replaced by the space $\Df_e(\RR_+)$ of	the restrictions to $\Rnon$ of even, smooth, and compactly supported functions on $\RR$; indeed, by using appropriate cutoffs, $\Df_e(\RR_+)$ is easily seen to be dense in $\Sz_e(\RR_+)$ in the Fr\'echet structure of the latter, so also in the graph norm of $L_\nu$. On the other hand, on the domain $C_c^\infty(\Rpos)$ the operator $L_\nu$ is essentially self-adjoint if and only if $\nu\geq 4$, see \cite[p.~161]{RS2}; mind that, in contrast to those of $\Df_e(\RR_+)$, elements of $C_c^\infty(\Rpos)$ must vanish near $0$. See also the discussion in \cite[Section 2]{HaSi}.
\end{rem}

From now on, we write $L_\nu$ for the self-adjoint extension of the Bessel operator discussed in Lemma \ref{lm:11}.
From \eqref{eq:trans_modul} and \eqref{eq:1} it follows immediately that $L_\nu$ commutes with Hankel translations:
\begin{equation}\label{eq:L_trinv}
\tau_\nu^{[y]} L_\nu f = L_\nu \tau_\nu^{[y]} f
\end{equation}
for any $f$ in the domain of $L_\nu$, i.e., $L_\nu$ is translation-invariant on $X_\nu$.

As $L_\nu$ is self-adjoint and nonnegative, by the spectral theorem we can define, for any bounded Borel function $F : \Rnon\to \CC$, the bounded operator $F(L_\nu)$ on $L^2(X_\nu)$, which, in light of \eqref{eq:1}, satisfies
\begin{equation*}
H_\nu (F(L_\nu)g)(x)=F(x^2)H_\nu g(x)
\end{equation*}
for all $g \in L^2(X_\nu)$.
Again, from \eqref{eq:trans_modul} we deduce that
\begin{equation}\label{eq:2}
\tau_\nu^{[y]} F(L_\nu)=F(L_\nu) \tau_\nu^{[y]},
\end{equation}
so $F(L_\nu)$ is translation-invariant on $X_\nu$.

If additionally $x \mapsto F(x^2)$ is in $L^2(X_\nu)$, equivalently $F\in L^2(\Rnon, \lambda^{\nu/2}\,\ddlam)$, then, by \eqref{eq:conv_prod}, $F(L_\nu)$ is a $\ast_\nu$-convolution operator,
\begin{equation}\label{eq:4}
F(L_\nu)g=g \ast_\nu K_{F(L_\nu)}, \qquad K_{F(L_\nu)}= H^{-1}_\nu(x \mapsto F(x^2)).
\end{equation}
Notice that \eqref{eq:16} implies a Plancherel formula for these kernels:
\begin{equation}\label{eq:17}
\Vert K_{F(L_\nu)}\Vert^2_{L^2(X_\nu)}= \frac{1}{2\kappa_\nu^2} \int_0^\infty |F(\lambda)|^2 \lambda^{\nu/2}\,\ddlam. 
\end{equation}

Now we consider the heat semigroup $\{e^{-t L_\nu} \}_{t>0}$ associated with $L_\nu$. In the following statement we record a few basic facts about it.

\begin{lm}\label{lm:1}
The heat semigroup associated with $L_\nu$ is a family of $\ast_\nu$-convolution operators, where the corresponding kernels are given by
\begin{equation}\label{eq:13}
	K_{e^{-t L_\nu}}(x) = \frac{2}{\Gamma(\nu/2) (4 t)^{\nu/2}} \exp\left( -\frac{x^2}{4t}\right).
\end{equation}
Moreover, $\Vert K_{e^{-tL_\nu}}\Vert_{L^1(X_\nu)}=1$ for all $t>0$. Consequently, the operators $e^{-t\Delta_\nu}$ are contractions on $L^p(X_\nu)$, $p\in[1,\infty]$, for all $t>0$. 
\end{lm}
\begin{proof}
The formula \eqref{eq:13} follows from \eqref{eq:4}, \eqref{eq:hankel_inv} and \cite[eq.\ (10.22.51)]{DLMF}.
Integration of this formula immediately gives  $\Vert K_{e^{-tL_\nu}}\Vert_{L^1(X_\nu)}=1$, and the $L^p$-contraction property then follows by Young's inequality \eqref{eq:young_Xnu} for $*_\nu$.
\end{proof}

We conclude with a few remarks on the dilation structure on $X_\nu$ and the related homogeneity of $L_\nu$.
Denote $\delta_\al g(x)=g(\al x)$ for $\al>0$. As $\dmu(x) = x^{\nu-1} \,\dd x$, these dilations are multiples of isometries on $L^p(X_\nu)$ for any $p \in [1,\infty]$. It is straightforward to verify that the Bessel operator $L_\nu$ is $2$-homogeneous, i.e.,
\begin{equation}\label{eq:L_homog}
L_\nu (\delta_\al g)= \al^2\delta_\al(L_\nu g), \qquad \alpha>0,
\end{equation}
for all $g$ in the domain of $L_\nu$. Thus, by spectral calculus, for any bounded Borel functions $F : \Rnon \to\CC$,
\begin{equation}\label{eq:3}
F(L_\nu) \delta_\al= \delta_\al F(\al^2 L_\nu), \qquad \alpha>0.
\end{equation}

\section{The semidirect product hypergroup}\label{s:hypergroupsG}

\subsection{Extension of the Bessel--Kingman hypergroup}
For any $u \in \RR$, the dilation $\gamma_u : x \mapsto e^u x$ is an automorphism of the Bessel--Kingman hypergroup $X_\nu$, in the sense that, by \eqref{eq:hyperdelta},
\[
\delta_{\gamma_u(x)} *_\nu \delta_{\gamma_u}(y) = \gamma_u(\delta_x *_\nu \delta_y)
\]
for all $x,y \in X_\nu$, and the right-hand side is the push-forward of the measure $\delta_x *_\nu \delta_y$ via $\gamma_u$. In addition, $\gamma_{u+u'} = \gamma_u \gamma_{u'}$ and $\gamma_0$ is the identity; in other words, $u \mapsto \gamma_u$ is an action of the group $\RR$ by automorphisms on the hypergroup $X_\nu$. Through this action we can define the \emph{semidirect product} hypergroup $G_\nu\defeq X_\nu \rtimes\RR$ (see \cite[Proposition 4.1 and Section 4.1, Example 4]{HeyKa} and also \cite[Definition 3.1]{W}).

Convolution of measures on $G_\nu$ is defined by setting
\[
\delta_{(x,u)} \diamond_\nu \delta_{(y,v)} = (\delta_{x} *_\nu \delta_{e^u y}) \otimes \delta_{u+v}
\]
for all $(x,u),(y,v) \in G_\nu$. The semidirect product $G_\nu$ is a noncommutative hypergroup, with unit element $(0,0)$ and involution given by $(x,u)^- = (e^{-u}x,-u)$.
The measures $e^{-\nu u}\,\dmu(x)\,\du$ and $\dmu(x) \,\du$ are the left and right Haar measures, and the modular function is
\begin{equation}\label{eq:modular}
m(x,u)\defeq e^{-\nu u};
\end{equation}
notice that the left and right Haar measures are one the push-forward of the other via the involution on $G_\nu$.

Unless otherwise specified, we shall use Lebesgue spaces $L^p(G_\nu)$ on $G_\nu$ defined in terms of the right Haar measure. This must be kept in mind when comparing the results below with the literature on hypergroups, as most references define Lebesgue spaces with respect to the left Haar measure.

We define left and right translations of functions $f$ on $G_\nu$ by
\begin{equation}\label{eq:35}
	\ell_{(x,u)} f(y,v) = r_{(y,v)}f(x,u) = \int_{G_\nu} f \,\dd (\delta_{(x,u)} \diamond_\nu \delta_{(y,v)})
\end{equation}
for all $(x,u),(y,v) \in G_\nu$; more explicitly,
\begin{equation}\label{eq:Gnu_transl}
\begin{aligned}
\ell_{(x,u)} f(y,v)&=\tau_\nu^{[x]}f(e^u y,u+v),\\
 r_{(x,u)}f(y,v) &=\tau_\nu^{[e^{v}x]}f(y,u+v),
\end{aligned}
\end{equation}
where the Hankel translation acts on the first variable of $f$. 
With these definitions, convolution of functions on $G_\nu$ is given by
\begin{equation}\label{eq:11}
\begin{split}
f\diamond_\nu g(x,u)&=\int_{G_\nu}   f((y,v)^{-})r_{(x,u)}g(y,v)\,\dmu(y)\, \dv \\
&=\int_{G_\nu}   \ell_{(x,u)}f((y,v)^-) g(y,v) \,\dmu(y)\, \dv.
\end{split}
\end{equation}

From the theory of hypergroups we deduce a number of properties of $\diamond_\nu$, such as Young's inequality.

\begin{lm}\label{lm:young}
For all $p,q,r\in[1,\infty]$ satisfying $1+\frac{1}{r}=\frac{1}{p}+\frac{1}{q}$,
\begin{equation}\label{eq:young}
	\Vert (f m^{-1/q'}) \diamond_\nu g \Vert_{L^r(G_\nu)} \leq \Vert f\Vert_{L^p(G_\nu)} \Vert g\Vert_{L^q(G_\nu)},
\end{equation}
where $1/q+1/q'=1$. In particular, for any $p\in[1,\infty]$,
\begin{equation}\label{eq:22}
	\Vert f\diamond_\nu g\Vert_{L^p(G_\nu)} \leq \Vert f\Vert_{L^p(G_\nu)} \Vert g\Vert_{L^1(G_\nu)}.
\end{equation}
\end{lm}
\begin{proof}
We follow the proof of the analogous result for groups (see, e.g, \cite[Lemma 2.1]{KlRu}).
From \cite[Theorem 6.2E]{J} we deduce
\[
\|f \diamond_\nu g\|_{L^\infty(G_\nu)} \leq \| f^- \|_{L^{q'}(G_\nu)} \| g \|_{L^{q}(G_\nu)} = \| f m^{1/q'} \|_{L^{q'}(G_\nu)} \| g \|_{L^{q}(G_\nu)},
\]
where $f^-(x,u) \defeq f((x,u)^{-})$; this inequality can be rewritten as
\begin{equation}\label{eq:first_young}
\|(fm^{-1/q'}) \diamond_\nu g\|_{L^\infty(G_\nu)} \leq \| f \|_{L^{q'}(G_\nu)} \| g \|_{L^{q}(G_\nu)}.
\end{equation}
From \cite[Theorems 5.3C and 6.2C]{J} we also deduce
\[
\| (fm^{1/q}) \diamond_\nu (g m^{1/q}) \|_{L^q(G_\nu)} =  \| (f \diamond_\nu g) m^{1/q} \|_{L^q(G_\nu)} \leq \| f m \|_{L^1(G_\nu)} \| g m^{1/q} \|_{L^q(G_\nu)},
\]
which can be equivalently rewritten as
\begin{equation}\label{eq:second_young}
\| (f m^{-1/q'}) \diamond_\nu g \|_{L^q(G_\nu)} \leq \| f \|_{L^1(G_\nu)} \| g \|_{L^q(G_\nu)}.
\end{equation}
Interpolation of \eqref{eq:first_young} and \eqref{eq:second_young} gives \eqref{eq:young}, and taking $q=1$ gives \eqref{eq:22}.
\end{proof}

We also have the following mapping properties of translations on $L^p$ spaces.

\begin{lm}\label{lm:transl_Gnu}
Let $p\in[1,\infty]$ and $f\in L^p(G_\nu)$. Then:
\begin{enumerate}[label=(\roman*)]
	\item\label{en:tr_bd_l} $\Vert  r_{(x,u)}f\Vert_{L^p(G_\nu)} \leq \Vert f\Vert_{L^p(G_\nu)}$ for any $(x,u)\in G_\nu$;
	\item\label{en:tr_bd_r} $\Vert  \ell_{(x,u)}f\Vert_{L^p(G_\nu)} \leq m(x,u)^{1/p} \Vert f\Vert_{L^p(G_\nu)}$ for any $(x,u)\in G_\nu$.
\end{enumerate}
In addition, if $p<\infty$, or if $p=\infty$ and $f \in C_0(G_\nu)$, then:
\begin{enumerate}[label=(\roman*),resume]
	\item\label{en:tr_cn_l} $\Vert r_{(x,u)}f - r_{(\bar x,\bar u)} f \Vert_{L^p(G_\nu)} \to 0$ as $(x,u) \to (\bar x,\bar u)$ in $G_\nu$;
	\item\label{en:tr_cn_r} $\Vert \ell_{(x,u)}f - \ell_{(\bar x,\bar u)} f \Vert_{L^p(G_\nu)} \to 0$ as $(x,u) \to (\bar x,\bar u)$ in $G_\nu$.
\end{enumerate}
\end{lm}
\begin{proof}
Notice that $\ell_{(x,u)} f = \delta_{(x,u)^-} \diamond_\nu f$ and $r_{(x,u)} f = f \diamond_\nu \delta_{(x,u)^-}$ \cite[Section 4.2]{J}.

The bounds \ref{en:tr_bd_l} and \ref{en:tr_bd_r} are consequences of the invariance properties of the right Haar measure under translations; specifically, \ref{en:tr_bd_l} can be found in \cite[Theorem 6.2B]{J}, while \ref{en:tr_bd_r} can be deduced from the former by using the fact that $\ell_{(x,u)} f = (r_{(x,u)^-} f^-)^-$ \cite[Lemma 4.2K]{J}, and also $r_{(x,u)} (m^\alpha f) = m(x,u)^{\alpha} \, m^{\alpha} \, r_{(x,u)} f$ for any $\alpha \in \CC$, due to the properties of the modular function \cite[Theorem 5.3C]{J}.

Finally, the continuity property \ref{en:tr_cn_l} follows from \cite[Lemmas 2.2B, 4.2F and 5.4H]{J}, and by involution we also deduce the case $p=\infty$ of \ref{en:tr_cn_r}; the case $p<\infty$ of \ref{en:tr_cn_r} then follows by a density argument, as in the proof of \cite[Lemma 5.4H]{J}.
\end{proof}

\subsection{Lifting of the Bessel operator}\label{ss:liftedbessel}

Write $\mathring{G}_\nu=(\Rpos\times\RR,\dmu(x)\,\du )$ for the interior of $G_\nu$. On the manifold $\mathring{G}_\nu$ we introduce the following differential operator with smooth coefficients:
\begin{equation}\label{eq:Deltanu}
\Delta_\nu = -\partial_u^2+e^{2u}L_\nu=\nabla_\nu^+\nabla_\nu,
\end{equation}
where 
\begin{equation}\label{eq:28}
\nabla_\nu=\begin{pmatrix}
\partial_u\\
e^u \partial_x
\end{pmatrix}
\end{equation}
and $\nabla_\nu^+$ denotes the formal adjoint of $\nabla_\nu$ on $\mathring{G}_\nu$.
Much as in Section \ref{ss:bessel_op}, the operator $\nabla_\nu$ does not depend on $\nu$, but its formal adjoint $\nabla_\nu^+$ does.
As $\Delta_\nu$ is a smooth divergence-form second-order differential operator on the manifold $\mathring{G}_\nu$, the discussion of the Appendix applies to it as well.
Furthermore, since $G_\nu$ and $\mathring{G}_\nu$ differ by a set of measure zero, we can think of $\Delta_\nu$ as an unbounded operator on $L^2(G_\nu)$.
We define $\Delta_\nu$ initially on $\Sz_e(\Rnon)\otimes C_c^\infty(\RR)$, i.e.\ the space of finite linear combinations of tensor products $f\otimes g$, where $f\in \Sz_e(\Rnon)$, $g\in C_c^\infty(\RR)$.

\begin{lm}\label{lm:12}
The operator $\Delta_\nu$ is essentially self-adjoint on $\Sz_e(\Rnon)\otimes C_c^\infty(\RR)$. Moreover, the self-adjoint extension is equal to $\nabla_\nu^+ \nabla_\nu$ with Neumann domain.
\end{lm}
\begin{proof}
To justify the essential self-adjointness,
we proceed in a way similar to \cite[Section 3.2]{DalMa}.
By \cite[pp.~256-257]{RS1} the operator $\Delta_\nu$ with domain $\Sz_e(\RR_+)\otimes C_c^\infty(\RR)$ is essentially self-adjoint if and only if $\Imm((\Delta_\nu + i)|_{\Sz_e(\RR_+)\otimes C_c^\infty(\RR)})^\perp=\{0\}$. In other words, it is sufficient to show that, if $f\in L^2(G_\nu)$ satisfies
\begin{equation*}
	\int_{G_\nu} f(x,u) \, \overline{ (\Delta_\nu +i) (\varphi_1 \otimes \varphi_2)(x,u)}\, \dmu(x)\,\du=0 
\end{equation*}
for all $\varphi_1\in\Sz_e(\RR_+)$ and $\varphi_2\in C_c^\infty(\RR)$, then $f=0$ almost everywhere.
	
Take $f$ as above, and define $g(x,u)=H_\nu (f(\cdot,u))(x)$. By \eqref{eq:1} and since the Hankel transform is, up to a constant, an isometry on $L^2(X_\nu)$, and also a bijection on $\Sz_e(\RR_+)$, the function $g$ is in $L^2(G_\nu)$ and satisfies
\begin{equation}\label{eq:33}
	\int_{G_\nu} g(x,u) \, \overline{ (-\partial^2_u+e^{2u}x^2 +i)\widetilde{\varphi_1}(x)\varphi_2(u)}\, \dmu(x)\,\du=0
\end{equation}
for all $\widetilde{\varphi_1}\in\Sz_e(\RR_+)$, $\varphi_2\in C_c^\infty(\RR)$ (here $\widetilde{\varphi_1}=H_\nu \varphi_1$). The task is to show that $g=0$ almost everywhere.
	
Fix $\varphi_2 \in C^\infty_c(\RR)$, and define
\begin{equation*}
	h(x)=(1+x^2)^{-1}\int_\RR g(x,u) \, \overline{ (-\partial^2_u+e^{2u}x^2 +i)\varphi_2(u)}\,\du, \qquad x \in \RR_+.
\end{equation*}
Since $g \in L^2(G_\nu)$ and $\varphi_2 \in C^\infty_c(\RR)$, it is easily checked that $h \in L^2(X_\nu)$. Moreover, by \eqref{eq:33},
\begin{equation*}
	\int_0^\infty h(x) \, \overline{(i+x^2)\widetilde{\varphi_1}(x)}\,\dmu(x)=0,\qquad \widetilde{\varphi_1}\in\Sz_e(\RR_+).
\end{equation*}
Since $(i+x^2)\widetilde{\varphi_1}(x) = H_\nu((i+L_\nu)\varphi_1)(x)$ and $H_\nu$ is a multiple of an isometry on $L^2(X_\nu)$, the latter condition can be rewritten as
\[
\langle H_\nu h, (i+L_\nu) \varphi_1 \rangle_{L^2(X_\nu)} = 0, \qquad \varphi_1\in\Sz_e(\RR_+).
\]
The essential self-adjointness of $L_\nu$ on $\Sz_e(\RR_+)$, discussed in Lemma \ref{lm:11}, then implies that $H_\nu h=0$ almost everywhere. Consequently, $h=0$ almost everywhere.
	
As $\varphi_2$ in the definition of $h$ was arbitrary, this means that, for all $\varphi_2\in C_c^\infty(\RR)$,
\begin{equation}\label{eq:34}
	\int_\RR g(x,u) \, \overline{(i+T_x)\varphi_2(u)}\,\du=0
\end{equation}
for almost all $x \in \RR_+$, where $T_x \defeq -\partial_u^2 +x^2e^{2u}$ is a Schr\"odinger operator with nonnegative smooth potential on $\RR$. It is known that, as an operator on $L^2(\RR)$, the Schr\"odinger operator $T_x$ is essentially self-adjoint on $C_c^\infty(\RR)$.
	
Let $M \subseteq C_c^\infty(\RR)$ be countable and dense, and set $V = \spann M$. Thus, \eqref{eq:34} holds for all $\varphi_2\in V$ and $x\in\RR_+\setminus E$, where $E$ is a Lebesgue null set independent of $\varphi_2$. 
So, for all $x \in \RR_+ \setminus E$, we have $g(x,\cdot) \in \Imm((i+T_x)|_{V})^\perp$.
Since each $T_x$ is essentially self-adjoint on $V$, we conclude that, for all $x\in\RR_+\setminus E$, we have $g(x,\cdot)=0$ almost everywhere on $\RR$. This simply means that $g$ vanishes almost everywhere on $G_\nu$, thus proving the essential self-adjointness of $\Delta_\nu$.

It remains to prove that the self-adjoint extension coincides with the operator $\nabla_\nu^+ \nabla_\nu$ with Neumann domain, in the sense of \eqref{eq:Neu}. As $\nabla_\nu$ maps $\Sz_e(\Rnon)\otimes C_c^\infty(\RR)$ into $(\Sz_e(\Rnon)\otimes C_c^\infty(\RR))\oplus (\Sz_o(\Rnon)\otimes C_c^\infty(\RR))$, it is sufficient to show that 
\[
(\Sz_e(\Rnon)\otimes C_c^\infty(\RR))\oplus (\Sz_o(\Rnon)\otimes C_c^\infty(\RR)) \subseteq \Dmin(\nabla_\nu^+).
\]
On the other hand, as 
\[
\nabla_\nu^+ = \begin{pmatrix} -\partial_u & e^u D_\nu^+ \end{pmatrix} = \begin{pmatrix} -1 \otimes \partial_u & D_\nu^+ \otimes e^u \end{pmatrix},
\]
and both $\partial_u$ and $e^u$ preserve $C^\infty_c(\RR)$, this reduces to showing that
any element of $\Sz_e(\Rnon)$ can be approximated in $L^2(X_\nu)$ by a sequence in $C^\infty_c(\Rpos)$, and also that any element of $\Sz_o(\RR_+)$ can be approximated by a sequence in $C^\infty_c(\Rpos)$ in the graph norm of $D_\nu^+$. The first fact is a triviality, while the second one is discussed in the proof of Lemma \ref{lm:11}.
\end{proof}

From now on, when writing $\Delta_\nu$ we always mean the self-adjoint extension. Notice that \eqref{eq:L_trinv} and \eqref{eq:L_homog}, together with \eqref{eq:Gnu_transl} and \eqref{eq:Deltanu}, imply that $\Delta_\nu$ is left-invariant, namely
\begin{equation}\label{eq:Deltanu_left}
\ell_{(y,v)} \Delta_\nu f  = \Delta_\nu \ell_{(y,v)} f  \qquad\forall (y,v) \in G_\nu,
\end{equation}
for all $f$ in the domain of $\Delta_\nu$; this is easily checked first for $f \in \Sz_e(\Rnon) \otimes C^\infty_c(\RR)$, and then arguing by density.
By the spectral theorem, from \eqref{eq:Deltanu_left} we deduce the left-invariance of any operator in the functional calculus:
\begin{equation}\label{eq:FDeltanu_left}
\ell_{(y,v)} F(\Delta_\nu)  = F(\Delta_\nu) \ell_{(y,v)}   \qquad\forall (y,v) \in G_\nu,
\end{equation}
for all bounded Borel functions $F : \Rnon \to \CC$; this should be compared with \eqref{eq:2} and \eqref{eq:3}.
As we shall see, the $F(\Delta_\nu)$ are actually right $\diamond_\nu$-convolution operators.

\subsection{Heat semigroup and Plancherel formula}

We consider the heat semigroup $\{e^{-t\Delta_\nu}\}_{t>0}$ associated with $\Delta_\nu$. We shall make use of \cite[Theorem 2.1]{G}, where the relation between heat kernels associated with certain operators, in our case $L_\nu$, and their lifted versions, in our case $\Delta_\nu$, is established. 

\begin{prop}\label{prop:6}
The heat semigroup associated with $\Delta_\nu$ is given by a family of $\diamond_\nu$-convolution operators, namely
\begin{equation}\label{eq:deltanu_heat_conv}
	e^{-t\Delta_\nu}f= f\diamond_\nu K_{e^{-t\Delta_\nu}}
\end{equation}
for all $t>0$, where the kernels are given by
\begin{equation}\label{eq:6}
	K_{e^{-t\Delta_\nu}}(x,u)=\frac{2}{\Gamma(\nu/2) } \int_0^\infty \Psi_t(\xi) \, (2\xi e^u)^{-\nu/2} \exp\left(-\frac{\cosh u}\xi-\frac{x^2}{2\xi e^u}\right)\,\dd\xi,
\end{equation}
and
\begin{equation}\label{eq:psi}
	\Psi_t(\xi)
	=\frac{e^{\frac{\pi^2}{4t}}}{\xi^2 \sqrt{4\pi^3 t}} \int_0^\infty \sinh(\te) \sin\left(\frac{\pi\te}{2t}\right) \exp\left(-\frac{\te^2}{4t}-\frac{\cosh \te}{\xi}\right)\,\dd\te,\qquad \xi>0.
\end{equation}
	Moreover, for all $t>0$, $K_{e^{-t\Delta_\nu}}\geq 0$ on $G_\nu$, and $\Vert K_{e^{-t\Delta_\nu}}\Vert_{L^1(G_\nu)}=1$. Thus, the $e^{-t\Delta_\nu}$ are positivity-preserving contractions on $L^p(G_\nu)$, $p\in[1,\infty]$, for all $t>0$.
\end{prop}
\begin{proof}
To justify the formula \eqref{eq:6}, we want to apply the theory of \cite{G}. Specifically, the operator $\Delta_\nu$ here corresponds to the self-adjoint operator $T$ of \cite[Section 2]{G}, where $L$ in \cite{G} is taken to be our $L_\nu$; with this choice, the domain of $T$ in \cite{G} clearly contains $\Sz_e(\Rnon) \otimes C^\infty_c(\RR)$, and as the latter is a core for $\Delta_\nu$, by Lemma \ref{lm:12}, the two operators are indeed the same.
So, by \cite[Theorem 2.1]{G},
\begin{equation}\label{eq:12}
		e^{-t\Delta_\nu}f(x,u)=\left(\int_\RR  \int_0^\infty \Psi_t(\xi) \, \exp\left(-\frac{\cosh(u-v)}{\xi}\right) e^{-\xi e^{u+v} L_\nu/2} f_v \,\dd\xi\,\dv\right)(x),
\end{equation}
for all $f \in L^1 \cap L^2(G_\nu)$ and almost all $(x,u) \in G_\nu$, where we use the notation $f_v(\cdot)=f(\cdot,v)$. Notice that, for any $t>0$, the function $\Psi_t$ satisfies the bound
\begin{equation}\label{eq:15}
		|\Psi_t(\xi)|\lesssim_t \xi^{-2} \qquad\forall \xi > 0,
\end{equation}
see \cite[Theorem 2.1]{G}. By \eqref{eq:15} and \eqref{eq:4} we can rewrite \eqref{eq:12} as
\begin{equation}\label{eq:first_heat_exp}
\begin{split}
	&e^{-t\Delta_\nu}f(x,u) \\
	&= \int_\RR  \int_0^\infty \Psi_t(\xi) \, \exp\left(-\frac{\cosh(u-v)}{\xi}\right) \left( f_v\ast_\nu K_{e^{-\xi e^{u+v} L_\nu/2}} \right) (x) \,\dd\xi \,\dv\\
	&= \int_\RR  f_v \ast_\nu \left( \int_0^\infty \Psi_t(\xi) \, \exp\left(-\frac{\cosh(u-v)}{\xi}\right)  K_{e^{-\xi e^{u+v} L_\nu/2}}(\cdot)\,\dd\xi  \right) (x) \,\dv.
\end{split}
\end{equation}
By Lemma \ref{lm:1} the inner integral is
\[\begin{split}
 &\int_0^\infty \Psi_t(\xi) \, \exp\left(-\frac{\cosh(u-v)}{\xi}\right)  K_{e^{-\xi e^{u+v} L_\nu/2}}(y) \,\dd\xi \\
	&=\frac{2}{\Gamma(\nu/2) } \int_0^\infty \Psi_t(\xi) \, (2\xi e^{u+v})^{-\nu/2} \exp\left(-\frac{\cosh(u-v)}{\xi} -\frac{y^2}{2\xi e^{u+v}}\right)  \,\dd\xi  \\
	&= e^{-v\nu} K_{e^{-t\Delta_\nu}}(e^{-v}y,u-v) ,
\end{split}\]
where $K_{e^{-t\Delta_\nu}}$ is the function defined in \eqref{eq:6}.
Plugging this into \eqref{eq:first_heat_exp} and expanding the Hankel convolution as in \eqref{eq:7} gives
\[\begin{split}
	e^{-t\Delta_\nu}f(x,u)&= \int_{G_\nu}  \tau_\nu^{[x]}f(y,v)  K_{e^{-t\Delta_\nu}}(e^{-v}y,u-v) e^{-v\nu}\,\dmu(y)\,\dv\\
	&=  \int_{G_\nu}  \tau_\nu^{[x]}f(e^{u-v} y,u-v) K_{e^{-t\Delta_\nu}}(y,v) \,\dmu(y)\,\dv\\
	&= \int_{G_\nu}  \ell_{(x,u)} f((y,v)^-) K_{e^{-t\Delta_\nu}}(y,v) \,\dmu(y)\,\dv \\
	&= f\diamond_\nu K_{e^{-t\Delta_\nu}}(x,u),
\end{split}\]
by \eqref{eq:Gnu_transl} and \eqref{eq:11}. This confirms that $K_{e^{-t\Delta_\nu}}$ given by \eqref{eq:6} is indeed the $\diamond_\nu$-convolution kernel of $e^{-t\Delta_\nu}$.

Now we move on to justifying that $K_{e^{-t\Delta_\nu}} \geq 0$ for all $t>0$.
By \eqref{eq:15} it is easily checked that the integral in \eqref{eq:6} is absolutely convergent and $K_{e^{-t\Delta_\nu}} \in C \cap L^1(G_\nu)$ for any $t>0$. Moreover, from \eqref{eq:deltanu_heat_conv}, \eqref{eq:11} and \eqref{eq:35} we deduce that  $e^{-t\Delta_\nu}$ is an integral operator on $G_\nu$,
\[
e^{-t\Delta_\nu} f(x,u) = \int_{G_\nu} K_t((x,u),(y,v)) \, f(y,v) \,\dmu(y) \,\dv,
\]
with integral kernel $K_t$ given by
\begin{equation}\label{eq:heat_conv_int}
\begin{split}
K_t((x,u),(y,v)) 
&= e^{-\nu v} r_{(x,u)} K_{e^{-t\Delta_\nu}}((y,v)^-) \\	
&= e^{-\nu v} \ell_{(y,v)^-} K_{e^{-t\Delta_\nu}}(x,u)
\end{split}
\end{equation}
(see also \cite[Theorem 2.3]{G}).
In particular, by Lemma \ref{lm:transl_Gnu}, $K_t(\cdot,(y,u)) \to K_{e^{-t\Delta_\nu}}$ in $L^1(G_\nu)$ as $(y,u) \to (0,0)$; thus, in order to prove that $K_{e^{-t\Delta_\nu}} \geq 0$, it is enough to show that $K_t \geq 0$ almost everywhere on $G_\nu \times G_\nu$, that is, that $e^{-t\Delta_\nu}$ is positivity-preserving.

Now, by \cite[Theorem 1.3.2]{Da}, the fact that the heat semigroup $e^{-t\Delta_\nu}$ is positivity-preserving is equivalent to the property that, for any real-valued $f$ in the domain $\Dom(\sqrt{\Delta_\nu})$ of $\sqrt{\Delta_\nu}$, there holds $|f|\in\Dom(\sqrt{\Delta_\nu})$ and
\begin{equation}\label{eq:29}
	\langle \sqrt{\Delta_\nu}|f|,\sqrt{\Delta_\nu}|f|\rangle 
	\leq \langle \sqrt{\Delta_\nu} f,\sqrt{\Delta_\nu} f\rangle.
\end{equation}
By Lemma \ref{lm:12}, here $\Delta_\nu$ is the operator $\nabla_\nu^+ \nabla_\nu$ with Neumann domain (see \eqref{eq:Neu}), that is, $\Delta_\nu = \nabla_\nu^* \nabla_\nu$ where $\nabla_\nu$ is given the maximal domain. Thus, 
\[
\Dom(\sqrt{\Delta_\nu}) = \Dmax(\nabla_\nu) = \{ f \in L^2(G_\nu) \tc |\nabla_\nu f| \in L^2(G_\nu) \}
\]
and, for all $f\in\Dom(\sqrt{\Delta_\nu})$,
\[
\langle \nabla_\nu f,\nabla_\nu f\rangle = \langle \sqrt{\Delta_\nu}f,\sqrt{\Delta_\nu}f\rangle
\]
(cf.\ the proof of \cite[Theorem X.25]{RS2}). Moreover, if $f \in \Dom(\sqrt{\Delta_\nu})$ is real-valued, then $f,|\nabla_\nu f|\in L^1_{\loc}(G_\nu)$, so by applying \cite[Lemma 7.6]{GiTr} we obtain that $\nabla_\nu |f| = \sign(f) \nabla_\nu f$. Hence, $|f| \in \Dom(\sqrt{\Delta_\nu})$ and \eqref{eq:29} is satisfied.
	
We are left with calculating the $L^1(G_\nu)$ norm of the heat kernels. Observe that
\[\begin{split}
	&\Vert K_{e^{-t\Delta_\nu}}\Vert_{L^1(G_\nu)}\\
	&=\int_\RR \int_0^\infty \Psi_t(\xi) \, \exp\left(-\frac{\cosh u}\xi\right)\frac{2}{\Gamma(\nu/2) } \int_0^\infty  \left(\frac{x^2}{2\xi e^u}\right)^{\nu/2} \exp\left(-\frac{x^2}{2\xi e^u}\right)\,\ddx\,\dd\xi\,\du\\
	&= \int_\RR  \int_0^\infty \Psi_t(\xi) \, \exp\left(-\frac{\cosh u}\xi\right)\,\dd\xi\,\du,
\end{split}\]
since the heat kernels $K_{e^{-sL_\nu}}$ have $L^1(X_\nu)$-norm equal to $1$ for any $s>0$. Furthermore, by \eqref{eq:psi}, the above expression can be rewritten as
\begin{equation*}
	 \frac{e^{\frac{\pi^2}{4t}}}{\sqrt{4\pi^3 t}} \int_\RR \int_0^\infty \sinh(\te) \sin\left(\frac{\pi\te}{2t}\right) e^{-\frac{\te^2}{4t}} \int_0^\infty \xi^{-2}\exp\left(-\frac{\cosh \te+\cosh u}{\xi}\right)\,\dd\xi\,\dd\te\,\du.
\end{equation*} 
Notice that the integral over $\xi$ equals  $(\cosh\te+\cosh u)^{-1}$. Moreover,
\begin{equation*}
	\int_\RR \frac{\du}{\cosh\te + \cosh u} = \frac{2\te}{\sinh\te}.
\end{equation*}
Hence, combining the above we arrive at
\[\begin{split}
	\int_\RR \int_0^\infty K_{e^{-t\Delta_\nu}}(x,u)\,\dmu(x)\,\du 
	&= \frac{e^{\frac{\pi^2}{4t}}}{\sqrt{\pi^3 t}}  \int_0^\infty  \te \sin\left(\frac{\pi\te}{2t}\right) e^{-\frac{\te^2}{4t}}\,\dd\te\\
	&= \frac{e^{\frac{\pi^2}{4t}}}{\sqrt{4\pi t}}  \int_\RR \cos\left(\frac{\pi\te}{2t}\right) e^{-\frac{\te^2}{4t}}\,\dd\te,
\end{split}\]
which is equal to $1$ by the formula for the Fourier transform of a Gaussian.
	
Therefore, $\Vert K_{e^{-t\Delta_\nu}}\Vert_{L^1(G_\nu)}=1$ for all $t>0$. Consequently, by Young's convolution inequality \eqref{eq:22} we obtain that the $e^{-t\Delta_\nu}$ are contractions on $L^p(G_\nu)$ for any $p\in[1,\infty]$. 
\end{proof}

The heat kernel formula of Proposition \ref{prop:6} is our starting point to derive a number of properties of the functional calculus for $\Delta_\nu$, including the fact that it is given by convolution operators, and the validity of a Plancherel formula relating the $L^2$ norm of the convolution kernel with a suitable $L^2$ norm of the corresponding spectral multiplier.

We start with establishing a basic property, which is related to the left-invariance \eqref{eq:FDeltanu_left} of the operators in the functional calculus (cf.\ \cite[Theorem 5]{Pav}). Recall that $m$ is the modular function on $G_\nu$, given by \eqref{eq:modular}.

\begin{lm}\label{lm:8}
Let $F : \RR_+\to\CC$ be a bounded Borel function and $f\in L^1(G_\nu)$ be such that $fm^{1/2}\in L^1(G_\nu)$. Then
\begin{equation*}
	F(\Delta_\nu) (f\diamond_\nu g)= f\diamond_\nu F(\Delta_\nu)g,\qquad g\in L^2(G_\nu).
\end{equation*}
\end{lm}
\begin{proof}
Fix $f$ as above.	Associativity of $\diamond_\nu$ on $L^1(G_\nu)$ and Young's inequality (Lemma \ref{lm:young}), together with a density argument, show that
\[
(f \diamond_\nu g) \diamond_\nu h = f \diamond_\nu (g \diamond_\nu h)
\]
for all $g \in L^2(G_\nu)$ and all $h \in L^1(G_\nu)$.
By Proposition \ref{prop:6} we can apply this to $h = K_{e^{-\Delta_\nu}}$ and obtain
\begin{equation*}
	e^{-\Delta_\nu} (f\diamond_\nu g)=  f\diamond_\nu e^{-\Delta_\nu}g,\qquad g\in L^2(G_\nu).
\end{equation*}
As the left-convolution operator $g \mapsto f \diamond_\nu g$ is $L^2$-bounded, 
by the spectral calculus for $e^{-\Delta_\nu}$ we get that, for all bounded Borel functions $G :\RR_+\to\CC$,
\begin{equation*}
	G(e^{-\Delta_\nu})(f\diamond_\nu g)= f\diamond_\nu G(e^{-\Delta_\nu}) g, \qquad g\in L^2(G_\nu).
\end{equation*}
Finally, taking $G(\lambda)=F(-\log \lambda)$ finishes the proof.
\end{proof}

For $t>0$ and $u\in\RR$ we define the function $M_{t,u} : \RR_+\to\RR$ by
\begin{equation}\label{eq:def_M}
M_{t}(\lambda,u)\defeq M_{t,u}(\lambda) = \int_0^\infty \Psi_t(\xi) \exp\left( -\frac{\cosh u}\xi - \frac{\xi\lambda e^u}2\right)\, \dd\xi
\end{equation}
for all $\lambda > 0$, where $\Psi_t$ is defined in \eqref{eq:psi}.

We emphasize that $M_{t,u}$ is independent of $\nu$, which is of paramount importance. This fact was already exploited in the literature, for instance in \cite{He99,MOV}.

\begin{lm}\label{lm:6}
	For all $t>0$ and $u\in\RR$, the function $M_{t,u}$ is bounded. Moreover, $M_{t,u}\in L^2(\RR_+, \lambda^{\alpha/2}\,\ddlam)$ for all $\alpha>0$.
\end{lm}
\begin{proof}
By \eqref{eq:15} we have
\begin{equation*}
		\|M_{t,u}\|_\infty \lesssim_t \int_0^\infty \xi^{-2} e^{-1/\xi}\, \dd\xi = 1,
\end{equation*}
and also
\[\begin{split}
		\Vert M_{t,u}\Vert^2_{L^2(\RR_+, \lambda^{\alpha/2}\,\ddlam)} 
		&\lesssim \int_0^\infty \left(\int_0^\infty \xi^{-2}  \exp\left( -\frac{\cosh u}\xi - \frac{\xi\lambda e^u}2\right)\,\dd\xi \right)^2 \lambda^{\alpha/2}\,\ddlam.
\end{split}\]
We split the integration over $\lambda$ onto the intervals $(0,1)$ and $(1,\infty)$. In the former case we have
\[
	\int_0^1 \left(\int_0^\infty \xi^{-2}  \exp\left( -\frac{\cosh u}\xi - \frac{\xi\lambda e^u}2\right)\,\dd\xi \right)^2 \lambda^{\alpha/2}\,\ddlam 
	\lesssim_\alpha  \left(\int_0^\infty \xi^{-2} e^{-1/\xi}\,\dd\xi \right)^2 = 1.
\]
On the other hand, for the remaining interval we have
\[\begin{split}
	\int_1^\infty& \left(\int_0^\infty \xi^{-2}  \exp\left( -\frac{\cosh u}\xi - \frac{\xi\lambda e^u}2\right)\,\dd\xi \right)^2 \lambda^{\alpha/2}\,\ddlam\\
	&\lesssim_\alpha \int_1^\infty \left(\int_0^\infty \xi^{-2} (\xi\lambda e^u)^{-(\alpha+2)/4}  \exp\left( -\frac{\cosh u}\xi\right)\,\dd\xi \right)^2 \lambda^{\alpha/2}\,\ddlam\\
	& \lesssim  \left(\int_0^\infty  \left(\frac{\eta}{e^u \cosh u }\right)^{(\alpha+2)/4} e^{-\eta} (\cosh u)^{-1}\,\dd\eta \right)^2\lesssim 1.
\end{split}\]
Combining the above finishes the proof.
\end{proof}

The functions $M_{t,u}$ encode the relation between the heat propagators on $X_\nu$ and $L_\nu$, already exploited in the proof of Proposition \ref{prop:6}.

\begin{lm}\label{lm:3}
For any $t>0$ and $u\in\RR$ there holds
\begin{equation*}
	K_{e^{-t\Delta_\nu}}(x,u) = K_{M_{t,u}(L_\nu)}(x) \qquad \qquad\text{for a.a. } x \in \Rnon.
\end{equation*}
\end{lm}
\begin{proof}
Fix $t>0$ and $u\in\RR$. By \eqref{eq:4} and Lemma \ref{lm:6} we have
\begin{equation*}
	K_{M_{t,u}(L_\nu)}(x)
	= \kappa_\nu^{-2} \int_0^\infty \int_0^\infty \Psi_t(\xi) \exp\left( -\frac{\cosh u}{\xi} - \frac{\xi e^u y^2}{2}\right) j_x^\nu(y) \,\dd\xi\,\dmu(y).
\end{equation*}
Because of \eqref{eq:15} we can apply Fubini's theorem and Lemma \ref{lm:1} to obtain
\[\begin{split}
	K_{M_{t,u}(L_\nu)}(x)
	&= \int_0^\infty  \Psi_t(\xi) \exp\left( -\frac{\cosh u}{\xi} \right) H_\nu^{-1}\left( y \mapsto e^{-\xi e^u y^2/2}\right)(x) \,\dd\xi\\
	&= \frac{2}{\Gamma(\nu/2)} \int_0^\infty \Psi_t(\xi) (2\xi e^u)^{-\nu/2} \exp\left( -\frac{\cosh u}{\xi} \right) \exp\left(-\frac{x^2}{2\xi e^u} \right) \,\dd\xi ,
\end{split}\]
and this is equal to $K_{e^{-t\Delta_\nu}}(x,u)$, see \eqref{eq:6}.
\end{proof}

\begin{rem}\label{rem:1}
Lemma \ref{lm:3} establishes in the context of the hypergroup $G_\nu$ the analogue of the relation between the heat kernels on a stratified group $N$ and the semidirect product group $G = N \rtimes \RR$, which is already known in the literature. In particular, for $n \in \NN \setminus \{0\}$, let $\Delta_{\RR^n}$ be the classical Laplacian on $\RR^n$ and $\widetilde{\Delta}_{\RR^n}= -\partial_u^2+e^{2u}\Delta_{\RR^n}$ the corresponding left-invariant Laplacian on $\RR^n \rtimes\RR$. By \cite[Proposition 4.3]{MOV},	
\begin{equation}\label{eq:heat_rel_eucl}
	K_{e^{-t\widetilde{\Delta}_{\RR^n}}}(x,u) = K_{M_{t,u}(\Delta_{\RR^n})}(x).
\end{equation}
Moreover, the Plancherel measures associated with $\Delta_{\RR^n}$ and $\widetilde{\Delta}_{\RR^n}$ can be explicitly calculated (cf.\ \cite[pp.~388-389]{MOV}) and, for all bounded Borel functions $F : \RR_+\to\CC$,
\[\begin{split}
	\int_{\RR^n} | K_{F(\Delta_{\RR^n})}(z) |^2\,\dz &\simeq_n \int_0^\infty  | F(\lambda) |^2 \lambda^{n/2}\,\ddlam,\\
	\int_\RR \int_{\RR^n} | K_{F (\widetilde{\Delta}_{\RR^n})}(z,u) |^2 \,\dz\,\du &\simeq_n \int_0^\infty | F(\lambda) |^2 \lambda^{[3/2,(n+1)/2]}\,\ddlam,
\end{split}\]
where we use the notation \eqref{eq:bipower}.
Of course, the first formula is analogous to the Plancherel formula \eqref{eq:17} for $L_\nu$. As we shall see below, an analogue of the second formula holds for $\Delta_\nu$.
\end{rem}

Let $\fctJ $ be the set of all finite linear combinations of decaying exponentials $\lambda\mapsto e^{-t\lambda}$, $\lambda\in \Rnon$, for some $t>0$, as in \cite[p.~389]{MOV}. The Stone--Weierstrass theorem yields that $\fctJ$ is uniformly dense in $C_0(\Rnon)$.

Define the linear operator $\Phi : \fctJ \to \bigcap_{\alpha>0} L^2(\Rnon\times\RR,\lambda^\alpha \ddlam\,\du)$ by setting
\[
\Phi (e^{-t\cdot}) = M_{t}
\]
for all $t > 0$, where $M_t$ is given in \eqref{eq:def_M}. We denote $(\Phi F)_u (\lambda)=\Phi F(\lambda,u)$.

\begin{lm}\label{lm:4}
For all $\alpha\in [1,\infty)$, the operator $\Phi$ extends to a bounded operator from $L^2(\RR_+,\lambda^{[3/2,(\alpha+1)/2]}\,\ddlam)$ to $L^2(\RR_+\times\RR, \,\lambda^{\alpha/2}\,\ddlam\,\du)$.
\end{lm}
\begin{proof}
This is already known for $\alpha\in\NN$, and was implicitly used in \cite{MOV}. For the reader's convenience we provide a short justification. 
	
Fix $\alpha\in\NN$, $\alpha\geq 1$. By Remark \ref{rem:1}, using \eqref{eq:heat_rel_eucl} and linearity, we deduce that, for all $F\in\fctJ$,
\[
	K_{F(\widetilde{\Delta}_{\RR^\alpha})}(z,u) = K_{(\Phi F)_u(\Delta_{\RR^\alpha})}(z),
\]
and therefore
\[\begin{split}
	\int_{\RR}\int_0^\infty  | \Phi F(\lambda,u)|^2 \lambda^{\alpha/2}\,\ddlam\,\du
	& \simeq \int_\RR \int_{\RR^\alpha} | K_{(\Phi F)_u (\Delta_{\RR^\alpha})}(z)|^2\,\dz\,\du\\
	&= \int_\RR \int_{\RR^\alpha} | K_{F (\widetilde{\Delta}_{\RR^\alpha})}(z)|^2\,\dz\,\du\\
	&\simeq \int_0^\infty | F(\lambda)|^2 \lambda^{[3/2,(\alpha+1)/2]}\,\ddlam.
\end{split}\]
Hence, if $\alpha$ is an integer, then $\Phi$ extends to a bounded operator between the spaces $L^2(\RR_+,\lambda^{[3/2,(\alpha+1)/2]}\,\ddlam)$ and $L^2(\RR_+\times\RR,\lambda^{\alpha/2}\,\ddlam \,\du)$.
	
The result for fractional $\alpha$ then follows by Stein's interpolation theorem for weighted $L^2$ spaces \cite{St}.
\end{proof}

\begin{cor}\label{cor:5}
If $F\in\fctJ$, then $F(\Delta_\nu)$ is a right $\diamond_\nu$-convolution operator, with convolution kernel $K_{F(\Delta_\nu)}\in L^1 \cap L^2(G_\nu)$. Moreover,
\begin{equation*}
	\Vert K_{F(\Delta_\nu)}\Vert_{L^2(G_\nu)}\lesssim \Vert F\Vert_{L^2(\RR_+,\lambda^{[3/2,(\nu+1)/2]}\,\ddlam)}.
\end{equation*}
In particular, the heat kernels associated with $\Delta_\nu$ belong to $L^2(G_\nu)$.
\end{cor}
\begin{proof}
By Proposition \ref{prop:6} and linearity it is clear that, for all $F \in \fctJ$, the operator $F(\Delta_\nu)$ is a convolution operator, with convolution kernel $K_{F(\Delta_\nu)} \in L^1(G_\nu)$. Moreover, again by linearity and Lemma \ref{lm:3}, we deduce that
\begin{equation}\label{eq:Delta_kernels}
K_{F(\Delta_\nu)}(x,u) = K_{(\Phi F)_u(L_\nu)}(x).
\end{equation}
So, by \eqref{eq:17}, and Lemma \ref{lm:4},
\[\begin{split}
	\int_{G_\nu} | K_{F(\Delta_\nu)}(x,u)|^2 \,\dmu(x)\,\du 
	&= \int_\RR \int_0^\infty | K_{(\Phi F)_u(L_\nu)}(x)|^2 \,\dmu(x)\,\du\\
	& \simeq \int_\RR \int_0^\infty | (\Phi F)_u(\lambda)|^2 \lambda^{\nu/2}\,\ddlam\,\du\\
	&\lesssim \int_0^\infty |F(\lambda)|^2 \lambda^{[3/2,(\nu+1)/2]}\,\ddlam,
\end{split}\]
as desired.
\end{proof}

We now want to extend Corollary \ref{cor:5} to a larger class of multipliers $F$. To this purpose, it is useful to establish some density properties of the class $\fctJ$.

\begin{lm}\label{lm:7}
Let $F\in L^2(\RR_+,\lambda^{[3/2,(\nu+1)/2]}\,\ddlam)$ be bounded.
\begin{enumerate}[label=(\roman*)]
	\item\label{lm:7(i)} There exists a uniformly bounded sequence $F_n\in C_c^\infty(\Rpos)$ converging to $F$ almost everywhere and in $L^2(\RR_+,\lambda^{[3/2,(\nu+1)/2]}\,\ddlam)$.
	\item\label{lm:7(ii)} If additionally $F\in C_0(\Rnon)$, then there exists a sequence $F_n\in\fctJ$ converging to $F$ uniformly and in $L^2(\RR_+,\lambda^{[3/2,(\nu+1)/2]}\,\ddlam)$.  
\end{enumerate}
\end{lm}
\begin{proof}
We start with \ref{lm:7(i)}. Given any $F\in L^2(\RR_+,\lambda^{[3/2,(\nu+1)/2]}\,\ddlam)$, we can approximate it first by the functions $F \ind_{[1/k,k]}$, $k \in \NN \setminus \{0\}$, which have compact support in $\Rpos$ and are dominated by $|F|$; in turn, each of these compactly supported functions can be approximated via mollifiers by $C_c^\infty(\Rpos)$ functions. Then a diagonal argument gives a sequence in $C_c^\infty(\Rpos)$ converging to $F$ in $L^2(\RR_+,\lambda^{[3/2,(\nu+1)/2]}\,\ddlam)$ and, up to extracting a subsequence, also almost everywhere. If additionally $F$ is bounded, then each approximant is also bounded by $\Vert F\Vert_{L^\infty}$.
		
For \ref{lm:7(ii)} it suffices to justify that $\fctJ$ is dense in the Banach space 
\begin{gather*}
	W=C_0(\Rnon)\cap  L^2\left(\RR_+,\lambda^{[3/2,(\nu+1)/2]}\,\ddlam\right) ,\\
	\Vert \cdot\Vert_W = \Vert\cdot \Vert_{L^\infty} +\Vert \cdot\Vert_{L^2(\RR_+,\lambda^{[3/2,(\nu+1)/2]}\,\ddlam)}.
\end{gather*}

Clearly, $C_c(\Rnon)$ is dense in $W$; this is easily seen by using compactly supported cutoffs and dominated convergence.
Thus, to conclude, it is enough to verify that the closure of $\fctJ$ in $W$ contains $C_c(\Rnon)$. Much as in the proof of \cite[Lemma 3.13]{Ma11}, let $G\in C_c(\Rnon)$ and set $g(x)\defeq G(x)e^{x}\in C_c(\Rnon)$. By the Stone--Weierstrass theorem we can find $g_k\in\fctJ$ converging uniformly to $g$. Thus, $G_k(x)\defeq g_k(x) e^{-x}\in\fctJ$ satisfy
\begin{equation*}
	| G_k(x)- G(x)| = e^{-x} |g_k(x) - g(x) |\leq e^{-x} \Vert g_k-g\Vert_{L^\infty}
\end{equation*}  
so $G_k$ converges uniformly to $G$. Moreover, since $|G_k(x)|\lesssim e^{-x}$, Lebesgue's dominated convergence theorem implies that $G_k$ converges to $G$ also in $ L^2(\RR_+,(\lambda^{3/2}+\lambda^{(\nu+1)/2})\,\ddlam)$. Hence, $C_c(\Rnon)$ is contained in the closure of $\fctJ$ in $W$.
\end{proof}

We can finally obtain the existence of the convolution kernel for a larger class of operators in the calculus for $\Delta_\nu$, and establish the existence of the associated Plancherel measure.

\begin{prop}\label{prop:2}
For all bounded Borel functions $F : \RR_+\to\CC$ that belong to $L^2(\RR_+,\lambda^{[3/2,(\nu+1)/2]}\,\ddlam)$, the operator $F(\Delta_\nu)$ is a $\diamond_\nu$-convolution operator and the corresponding kernel $K_{F(\Delta_\nu)}$ is in $L^2(G_\nu)$. Moreover, there exists a regular Borel measure $\sigma_\nu$, called the \emph{Plancherel measure} associated with $\Delta_\nu$, such that
\begin{equation}
\begin{split}\label{eq:19}
	\int_\RR \int_0^\infty | K_{F(\Delta_\nu)}(x,u)|^2\,\dmu(x) \,\du
	&= \int_0^\infty |F(\lambda)|^2\, \dd\sigma_\nu(\lambda)\\
	&\lesssim \int_0^\infty |F(\lambda)|^2 \,\lambda^{[3/2,(\nu+1)/2]}\, \ddlam
\end{split}
\end{equation}
for all such $F$.
Additionally, the null sets for $\sigma_\nu$ and for the spectral measure associated with $\Delta_\nu$ are the same.
\end{prop}
\begin{proof}	
Firstly, let $F\in L^2(\RR_+,\lambda^{[3/2,(\nu+1)/2]}\,\ddlam)$ be also in $C_0(\Rnon)$. By Lemma \ref{lm:7} \ref{lm:7(ii)} we find a sequence $F_n\in \fctJ$ converging to $F$ in $L^2(\RR_+,\lambda^{[3/2,(\nu+1)/2]}\,\ddlam)$ and uniformly. By the spectral theorem, $F_n(\Delta_\nu)\to F(\Delta_\nu)$ in the operator norm on $L^2(G_\nu)$.
	On the other hand, by Corollary \ref{cor:5}, the operator mapping $G\in\fctJ$ to $K_{G(\Delta_\nu)} \in L^2(G_\nu)$ extends to a bounded operator 
\[
T : L^2\left(\RR_+,\lambda^{[3/2,(\nu+1)/2]}\,\ddlam\right)\to L^2(G_\nu).
\]
Thus, $K_{F_n(\Delta_\nu)} \to TF$ in $L^2(G_\nu)$.
Combining the above, for all $f\in C_c(G_\nu)$, we see that,  $F_n(\Delta_\nu)f$ converges to $F(\Delta_\nu)f$ in $L^2(G_\nu)$, and also, by Young's inequality, $F_n(\Delta_\nu)f$ converges to $f\diamond_\nu TF$ in $L^2(G_\nu)$. This shows that $F(\Delta_\nu)$ is indeed a $\diamond_\nu$-convolution operator with kernel $K_{F(\Delta_\nu)}\defeq TF \in L^2(G_\nu)$.
	
In order to construct the Plancherel measure we follow the approach in the proof of \cite[Lemma 1]{Si2}. Let $E_\nu$ 
be the spectral measure associated with $\Delta_\nu$, by the spectral theorem. For any Borel function $G : \Rnon \to \CC$ the operator
\begin{equation*}
G(\Delta_\nu)=\int_0^\infty G(\lambda)\, \dd E_\nu(\lambda)
\end{equation*}
is a closed operator on $L^2(G_\nu)$, which is bounded whenever $G$ is bounded. Moreover, $\langle \dd E_\nu f,f\rangle$ is a regular Borel measure on $\Rnon$ for any $f\in L^2(G_\nu)$, and
\begin{equation}\label{eq:spectralcalculus}
\| G(\Delta_\nu) f\|_{L^2(G_\nu)}^2 = \int_0^\infty |G(\lambda)|^2 \langle \dd E_\nu(\lambda)f,f\rangle,
\end{equation}
where the right-hand side is finite if and only if $f$ is in the domain of $G(\Delta_\nu)$.
	
Now take $F\in L^2(\RR_+,\lambda^{[3/2,(\nu+1)/2]}\,\ddlam) \cap C_0(\Rnon)$.
Notice that, for any $f\in C_c(G_\nu)$, Lemma \ref{lm:8} gives
\[
	e^{-\Delta_\nu} F(\Delta_\nu)f = e^{-\Delta_\nu} ( f\diamond_\nu  K_{F(\Delta_\nu)}) =f\diamond_\nu ( e^{-\Delta_\nu} K_{F(\Delta_\nu)}).
\]
Clearly, $e^{-\Delta_\nu}$ and $F(\Delta_\nu)$ commute, so we also get
\begin{equation*}
	e^{-\Delta_\nu} F(\Delta_\nu)f = F(\Delta_\nu)  e^{-\Delta_\nu} f =f\diamond_\nu ( F(\Delta_\nu) K_{e^{-\Delta_\nu}} ).
\end{equation*}
Thus,
\begin{equation}\label{eq:associativity}
	e^{-\Delta_\nu} K_{F(\Delta_\nu)} = F(\Delta_\nu) K_{e^{-\Delta_\nu}}.
\end{equation}
Combining \eqref{eq:spectralcalculus} and \eqref{eq:associativity}, we deduce that
\[\begin{split}
	\int_{G_\nu} | K_{F(\Delta_\nu)}(x,u) |^2 \,\dmu(x)\,\du 
	&= \int_0^\infty e^{2\lambda} \langle \dd E_\nu(\lambda) e^{-\Delta_\nu} K_{F(\Delta_\nu)},e^{-\Delta_\nu} K_{F(\Delta_\nu)} \rangle \\
	&= \int_0^\infty e^{2\lambda} \langle \dd E_\nu(\lambda) F(\Delta_\nu) K_{e^{-\Delta_\nu}},F(\Delta_\nu) K_{e^{-\Delta_\nu}} \rangle \\
	& = \int_0^\infty |F(\lambda)|^2 e^{2\lambda}  \langle \dd E_\nu(\lambda)  K_{e^{-\Delta_\nu}}, K_{e^{-\Delta_\nu}} \rangle.
\end{split}\]
This gives the formula for the Plancherel measure:
\begin{equation}\label{eq:plancherel_def}
	\dd\sigma_\nu (\lambda)= e^{2\lambda} \langle \dd E_\nu(\lambda)K_{e^{-\Delta_\nu}},K_{e^{-\Delta_\nu}}\rangle.
\end{equation}
With this definition of the measure $\sigma_\nu$, the equality in \eqref{eq:19} is proved under the additional assumption that $F\in C_0(\Rnon)$.
	
Now we establish the inequality in \eqref{eq:19}. For any $G\in\fctJ$, by Corollary \ref{cor:5}, 
\[\begin{split}
	\int_0^\infty |G(\lambda)|^2\, \dd\sigma_\nu(\lambda)
	&=\int_\RR \int_0^\infty | K_{G(\Delta_\nu)}(x,u)|^2\,\dmu(x) \,\du\\
	& \lesssim  \int_0^\infty |G(\lambda)|^2 \,\lambda^{[3/2,(\nu+1)/2]}\,\ddlam. 
\end{split}\]
Since $\fctJ$ is dense in $L^2(\RR_+,\lambda^{[3/2,(\nu+1)/2]}\,\ddlam)$ we obtain the inequality in \eqref{eq:19} for all $G\in L^2(\RR_+,\lambda^{[3/2,(\nu+1)/2]}\,\ddlam)$. In particular, we deduce that $\sigma_\nu$ is absolutely continuous with respect to Lebesgue measure on $\RR_+$ and its density is bounded by a multiple of $\lambda^{[1/2,(\nu-1)/2]}$.
	
Now we prove that the null sets for $\sigma_\nu$ and the spectral measure are the same. By the definition \eqref{eq:plancherel_def} of $\sigma_\nu$, the only non-trivial implication is that if for $A\subseteq\RR_+$ we have $\sigma_\nu(A)=0$, then also $\ind_A(\Delta_\nu)=E_\nu(A)=0$.
	
Fix such a set $A$. Let $D$ be a countable dense subset of $L^2(G_\nu)$. Since $\sigma_\nu$ and the measures $\langle \dd E_\nu f,f\rangle$ for $f\in D$ are all regular Borel measures, we can find a sequence of compact sets $K_n$ and open sets $U_n$ such that $K_n\subseteq K_{n+1}\subseteq\ldots\subseteq A\subseteq\ldots\subseteq U_{n+1} \subseteq U_n$ and also $\sigma_\nu(U_n\setminus K_n)\to 0$ and $\langle E_\nu(U_n\setminus K_n)f,f\rangle\to0$, $f\in D$. Let $\varphi_n\in C_c(\RR)$ be such that $0\leq \varphi_n\leq 1$, $\supp\varphi_n\subseteq U_n$ and ${\varphi_n}_{|_{K_n}}\equiv 1$. Then, $\varphi_n\to \ind_A$ almost everywhere with respect to $\sigma_\nu$ and each $\langle \dd E_\nu  f,f\rangle$, $f\in D$.
	 
Since $0 \leq \varphi_n\leq \ind_{U_n}$ and $\sigma_\nu(U_n)$ tends to $0$, we also have that $\varphi_n\to 0$ in $L^2(\dd \sigma_\nu)$. As $\varphi_n\in C_c(\Rnon)$, we can apply \eqref{eq:19} and deduce that $K_{\varphi_n(\Delta_\nu)}\to 0$ in $L^2(G_\nu)$; since the $\varphi_n$ are uniformly bounded, by Young's convolution inequality (Lemma \ref{lm:young}) and a density argument we obtain that $\varphi_n(\Delta_\nu) \to 0$ in the strong operator topology on $L^2(G_\nu)$. On the other hand, for any $f\in D$,
\begin{equation*}
	 \Vert \varphi_n(\Delta_\nu) f- \ind_A (\Delta_\nu)f \Vert_{L^2(G_\nu)}
	 = \int_0^\infty |\varphi_n(\lambda)-\ind_A(\lambda)|^2 \langle \dd E_\nu(\lambda)f,f\rangle\to 0,
\end{equation*}
by the dominated convergence theorem, because $\varphi_n\to\ind_A $ almost everywhere with respect to $\langle \dd E_\nu f,f\rangle$, and moreover $|\varphi_n-\ind_A|\leq 1$ and $\langle \dd E_\nu f,f\rangle$ is a finite measure. From the density of $D$ we deduce that $\varphi_n(\Delta_\nu)$ converges to $\ind_A(\Delta_\nu)$ in the strong operator topology. Hence, by the uniqueness of the limit, $\ind_A(\Delta_\nu)=0$ as an operator, so $A$ is a null set for the spectral measure.
	 
It remains to prove the existence of the convolution kernel and the equality in \eqref{eq:19} for general $F$, namely for a bounded function $F\in L^2(\RR_+,\lambda^{[3/2,(\nu+1)/2]}\,\ddlam)$. We can apply Lemma \ref{lm:7} \ref{lm:7(i)} to find a uniformly bounded sequence of functions $F_n\in C_c^\infty(\Rpos)$ converging to $F$ almost everywhere (with respect to Lebesgue measure) and in $L^2(\RR_+,\lambda^{[3/2,(\nu+1)/2]}\,\ddlam)$. Since $\sigma_\nu$ is absolutely continuous with respect to Lebesgue measure, and the spectral measure has the same null sets, we obtain that $F_n(\Delta_\nu)$ converges to $F(\Delta_\nu)$ in the strong operator topology. Much as above, the kernels $K_{F_n(\Delta_\nu)}$ converge in $L^2(G_\nu)$, and the limit is the $\diamond_\nu$-convolution kernel of $F(\Delta_\nu)$. Finally, a density argument proves the equality in \eqref{eq:19} for $F$.	
\end{proof}

The next result is a generalization of Lemma \ref{lm:3}.

\begin{lm}\label{lm:5}
Let $F$ be bounded and in $L^2(\RR_+,\lambda^{[3/2,(\nu+1)/2]}\,\ddlam)$. Then
\begin{equation*}
	K_{F(\Delta_\nu)}(x,u)= K_{(\Phi F)_u(L_\nu)}(x)
\end{equation*}
for almost all $(x,u) \in G_\nu$.
\end{lm}
\begin{proof}
By \eqref{eq:17} and Lemma \ref{lm:4}, for any $H \in L^2(\RR_+,\lambda^{[3/2,(\nu+1)/2]}\,\ddlam)$,
\begin{multline*}
	\int_{G_\nu} | K_{(\Phi H)_u(L_\nu)}(x)|^2 \,\dmu(x)\,\du \\
	\simeq \int_\RR \int_0^\infty |(\Phi H)_u(\lambda)|^2 \lambda^{\nu/2}\,\ddlam \,\du
	\lesssim \Vert H\Vert^2_{L^2(\RR_+,\lambda^{[3/2,(\nu+1)/2]}\,\ddlam)}
\end{multline*}
So, the correspondence $T$ mapping $H$ to the function $(x,u) \mapsto K_{(\Phi H)_u(L_\nu)}(x)$ is a bounded operator $T : L^2(\RR_+,\lambda^{[3/2,(\nu+1)/2]}\,\ddlam) \to L^2(G_\nu)$.

Fix $F$ as in the statement. Let $F_n\in\fctJ$ be a sequence of functions converging to $F$ in $L^2(\RR_+,\lambda^{[3/2,(\nu+1)/2]}\,\ddlam)$. Proposition \ref{prop:2} implies that $K_{F_n(\Delta_\nu)}$ converges to $K_{F(\Delta_\nu)}$ in $L^2(G_\nu)$. At the same time, the boundedness of $T$ gives that $K_{(\Phi F_n)_u(L_\nu)}(x)$ converges to $K_{(\Phi F)_u(L_\nu)}(x)$ in $L^2(G_\nu)$. Hence, up to extracting a subsequence,
\begin{equation}\label{eq:20}
	K_{F_{n}(\Delta_\nu)}(x,u)\to K_{F(\Delta_\nu)}(x,u)\quad \text{and}\quad K_{(\Phi F_{n})_u(L_\nu)}(x) \to K_{(\Phi F)_u(L_\nu)}(x)
\end{equation}
almost everywhere on $G_\nu$. As $K_{F_{n_k}(\Delta_\nu)}(x,u)= K_{(\Phi F_{n_k})_u(L_\nu)}(x)$ by \eqref{eq:Delta_kernels}, we conclude that $K_{ F(\Delta_\nu)}(x,u)=K_{(\Phi F)_u(L_\nu)}(x)$ almost everywhere.
\end{proof}

\subsection{Riemannian distance and finite propagation speed}

Recall that $\Delta_\nu=\nabla_\nu^+\nabla_\nu$, where $\nabla_\nu$ is given in \eqref{eq:28}.
We equip the manifold $\mathring{G}_\nu$ with the Riemannian structure which makes the vector fields $\partial_u$ and $e^u \partial_x$ an orthonormal frame.
As these are the vector fields appearing in \eqref{eq:28} as the components of $\nabla_\nu$, the Riemannian distance $\dist$ on $\mathring{G}_\nu$ is the control distance associated with $\nabla_\nu$ in the sense of the Appendix (see Remark \ref{rem:distance}).

We emphasise that $\nabla_\nu$ does not depend on $\nu$, so also the Riemannian structure on $G_\nu$ and the corresponding distance do not. This is confirmed by the following statement, where we obtain an explicit formula for the Riemannian distance on $\mathring{G}_\nu$, analogous to the formula \eqref{eq:Gdistance} for the sub-Riemannian distance on the group $G$.

\begin{prop}\label{prop:Gnu_dist}
For any $\nu \geq 1$, the Riemannian distance on $\mathring{G}_\nu$ is given by
\begin{equation}\label{eq:dist}
	\dist((x,u),(x',u'))=\arccosh\left(\cosh(u-u')+\frac{|x-x'|^2}{2e^{u+u'}} \right).
\end{equation}
The same expression also gives a distance on the whole $G_\nu$.
\end{prop}
\begin{proof}
When $N = \RR^d$ is abelian and $\Delta_N$ is the standard Laplacian on $\RR^d$, then the semidirect product $G = \RR^d \rtimes \RR$ discussed in the introduction, equipped with the Riemannian metric which makes the vector fields \eqref{eq:liftedvectorfields} an orthonormal frame, is a realisation of the real hyperbolic space of dimension $d+1$. It is well known (see, e.g., \cite[Proposition 2.7]{MOV}) that the 
Riemannian distance between two points $(x,u),(x',u') \in \RR^d \rtimes \RR$ is given by the right-hand side of \eqref{eq:dist}.

Take now $d=1$. Then $\mathring{G}_\nu$ and $G_\nu$ can be thought of as an open subset and a closed subset of the hyperbolic plane $\RR \rtimes \RR$. So, certainly the expression \eqref{eq:dist} defines a distance on each of $G_\nu$ and $\mathring{G}_\nu$.
Moreover, the vector fields $\partial_u$ and $e^u \partial_x$ on $\mathring{G}_\nu$ are the restrictions to $\mathring{G}_\nu$ of the vector fields \eqref{eq:liftedvectorfields} on $\RR \rtimes \RR$.
Thus, the Riemannian metric tensor on $\mathring{G}_\nu$ is just the restriction of the corresponding tensor on the hyperbolic plane, and the length of a curve $\gamma$ in the Riemannian manifold $\mathring{G}_\nu$ is the same as the length of $\gamma$ thought of as a curve in the hyperbolic plane.
	
Recall that the hyperbolic plane is a complete Riemannian manifold, so the Riemannian distance between two points of the plane is the length of the unique length-minimising curve joining those points. Thus, in order to conclude that the formula \eqref{eq:dist} also gives the Riemannian distance on $\mathring{G}_\nu$, it is enough to show that $\mathring{G}_\nu$ is geodesically convex in the hyperbolic plane.

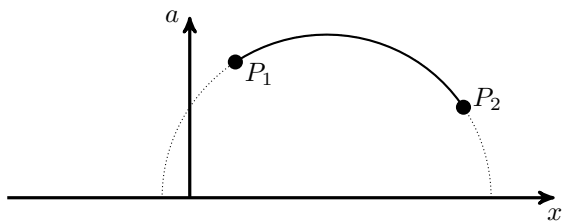
\begin{figure}[ht]
	\begin{tikzpicture}[scale=0.6,
	axis/.style={very thick, ->, >=stealth'}, 	
	important line/.style={thick}, 
	dashed line/.style={dashed, thin}, 
	]
	\draw[axis] (0,0) -- (12,0) node[below] {$x$};
	\draw[axis] (4,0) -- (4,4) node[left] {$a$} ;
	
	\draw[densely dotted] (7+sqrt{13},0) arc
	[
	start angle=0,
	end angle=180,
	x radius=sqrt{13},
	y radius =sqrt{13}
	] ;
	\draw [thick,domain=33.69:123.69] plot ({7+sqrt{13}*cos(\x)}, {sqrt{13}*sin(\x)});
	\node[label={[right] $P_2$}, style={circle,fill,inner sep=2pt}] at (10,2) {};
	\node[label={[below right] $P_1$}, style={circle,fill,inner sep=2pt}] at (5,3) {};
	\end{tikzpicture}
	\caption{The geodesic between $P_1$ and $P_2$.}
	\label{fig:1}
\end{figure}

This convexity property is easily seen if one works with the half-plane model $\{ (x,a) \tc x \in \RR, \, a \in \Rpos \}$ of the hyperbolic plane, corresponding to the change of variables $a = e^u$. Via this change of variables, $\mathring{G}_\nu$ corresponds to the open quadrant $\{(x,a) \tc x,a \in \Rpos \}$. As geodesics in the half-plane model are just segments of lines or circles perpendicular to the boundary $\{a = 0\}$ (see, e.g., \cite[Theorem 9.3]{CFKP}), it is clear that geodesics joining two points in $\mathring{G}_\nu$ are entirely contained in $\mathring{G}_\nu$  (see Figure \ref{fig:1}).
\end{proof}

From now on, we equip $G_\nu$ with the distance $\dist$ of Proposition \ref{prop:Gnu_dist}. Moreover, for all $(x,u) \in G_\nu$, we set
\[
|(x,u)|_{G_\nu} \defeq \dist((0,0),(x,u)) = \arccosh(\cosh(u)+e^{-u} x^2/2),
\]
and denote $B_{G_\nu}(0,r)=\{(x,u)\in G_\nu \tc |(x,u)|_{G_\nu}<r \}$ for all $r>0$. A function on $G_\nu$ will be called \emph{radial} if it has the form $f(|\cdot|_{G_\nu})$, i.e., if it depends only on the distance of its argument from the origin $(0,0)$ of $G_\nu$.

The following proposition concerning integration of radial functions on $G_\nu$ is a simple modification of \cite[Propositions 2.8 and 2.9]{MOV}.

\begin{prop}
For any Borel function $f:\Rnon\to \Rnon$ there holds
\begin{multline}\label{eq:radial_int}
	\int_{G_\nu} f(|(x,u)|_{G_\nu} )\,\dmu(x)\,\du \\
	= \int_{G_\nu} f(|(x,u)|_{G_\nu} )\,e^{-\nu u}\,\dmu(x)\,\du 
	= c_\nu \int_0^\infty f(r) (\sinh t)^\nu\,\dd r,
\end{multline}
where $c_\nu = 2^{\nu-1} \Beta(\nu/2,\nu/2)$. Moreover,
\begin{equation}\label{eq:5}
	\int_{G_\nu} f(|(x,u)|_{G_\nu}) e^{-u\nu} x^{\nu}\,\dmu(x)\,\du 
	\lesssim_\nu \int_{G_\nu} f(|(x,u)|_{G\nu} ) \, |(x,u)|_{G_\nu} \,\dmu(x)\,\du.
\end{equation}	
\end{prop}

Radial functions arise naturally when considering operators in the functional calculus for $\Delta_\nu$, up to a twist with the modular function $m$ of \eqref{eq:modular}.

\begin{cor}\label{cor:4}
For all bounded functions $F \in L^2(\RR_+,\lambda^{[3/2,(\nu+1)/2]}\,\ddlam)$, the function $m^{1/2} K_{F(\Delta_\nu)}$ is radial on $G_\nu$. Moreover, for all $r > 0$,
\begin{equation*}
	\int_{B_{G_\nu}(0,r)} | K_{F(\sqrt{\Delta_\nu})}(x,u)|^2 x^{\nu}\,\dmu(x)\,\du 
	\lesssim_\nu r \int_{B_{G_\nu}(0,r)} | K_{F(\sqrt{\Delta_\nu})}(x,u)|^2 \,\dmu(x)\,\du.
\end{equation*}
\end{cor}
\begin{proof}
Notice that we can rewrite \eqref{eq:6} as
\begin{equation*}
e^{\nu u/2} K_{e^{-t\Delta_\nu}}(x,u)
= \frac{2}{\Gamma(\nu/2)} \int_0^\infty \Psi_t(\xi) (2\xi)^{-\nu/2} \exp\left(-\frac{\cosh|(x,u)|_{G_\nu}}{\xi}\right)\,\dd\xi.
\end{equation*}
This means that $m^{1/2} K_{e^{-t\Delta_\nu}}$ is radial on $G_\nu$ for all $t > 0$. By linearity, this implies the radiality of $m^{1/2} K_{F(\Delta_\nu)}$ for all $F \in \fctJ$, and a density argument, as in \eqref{eq:20}, extends the result to all bounded $F \in L^2(\RR_+,\lambda^{[3/2,(\nu+1)/2]}\,\ddlam)$.

Now, let $r \geq 0$. By applying \eqref{eq:5} and \eqref{eq:radial_int} to the radial function
\[
	f_r(x,u)=e^{\nu u} \, |K_{F(\Delta_\nu)}(x,u)|^2 \, \ind_{B_{G_\nu}(0,r)}(x,u)
\]
we get
\[\begin{split}
	\int_{B_{G_\nu}(0,r)} | K_{F(\sqrt{\Delta_\nu})}(x,u)|^2 x^{\nu}\,\dmu(x)\,\du 
	&=\int_{G_\nu} f_r(x,u) e^{-\nu u} x^{\nu} \,\dmu(x)\,\du \\
	&\lesssim_\nu \int_{G_\nu} f_r(x,u) \, |(x,u)|_{G_\nu} \,\dmu(x)\,\du\\
	&\leq r \int_{G_\nu} f_r(x,u) \,\dmu(x)\,\du\\
	&= r \int_{G_\nu} f_r(x,u) \, e^{-\nu u}\,\dmu(x)\,\du\\
	&= r \int_{B_{G_\nu}(0,r)} | K_{F(\sqrt{\Delta_\nu})}(x,y)|^2 \,\dmu(x)\,\du,
\end{split}\]
as desired.
\end{proof}

Now we justify that $\Delta_\nu$ has the finite propagation speed property. As in the introduction, for any $r>0$, we set
\begin{equation}\label{eq:fctE}
\fctE_r = \{ F \in \Sz(\RR) \tc F \text{ even},\ \supp \hat F \subseteq [-r,r] \}.
\end{equation}

\begin{lm}\label{lm:13}
The family of operators $\{\cos(t\sqrt{\Delta_\nu})\}_{t\in\RR}$ has finite propagation speed with respect to the distance $\dist$, that is,
\begin{equation}\label{eq:fps}
	\supp (\cos(t\sqrt{\Delta_\nu}) f) \subseteq \{ (x,u) \in G_\nu \tc \dist((x,u),\supp f) \leq |t| \} \qquad\forall t \in \RR
\end{equation}
for all $f \in L^2(G_\nu)$.
Moreover, for all $r>0$,
\[
	\supp K_{F(\Delta_\nu)}\subseteq \overline{B_{G_\nu}}(0,r) \qquad \forall F \in \fctE_r.
\]
\end{lm}
\begin{proof}
In light of Lemma \ref{lm:12}, we can apply Proposition \ref{prop:8} to the operator $\Delta_\nu = \nabla^+ \nabla$ with Neumann domain on the manifold $\mathring{G}_\nu$, and obtain the finite propagation speed property \eqref{eq:fps}. To be precise, Proposition \ref{prop:8} gives the analogue of \eqref{eq:fps} with $\mathring{G}_\nu$ in place of $G_\nu$, and supports interpreted accordingly. However, $\mathring{G}_\nu$ is a dense open subset with full measure in $G_\nu$; so, for any $f \in L^2(G_\nu) = L^2(\mathring{G}_\nu)$, the support of $f$ in $G_\nu$ is the closure in $G_\nu$ of its support in $\mathring{G}_\nu$, and \eqref{eq:fps} follows as stated.
	
Now fix an $F \in \fctE_r(\RR)$ for some $r>0$. By the Fourier inversion formula (cf.\ \cite[Lemma 2.1]{CoSi}),
\begin{equation}\label{eq:fps_subordination}
	F(\sqrt{\Delta_\nu}) = \frac1{2\pi}\int_{-r}^r \hat{F}(t) \cos( t\sqrt{\Delta_\nu})\, \dd t,
\end{equation}
whence, by \eqref{eq:fps},
\begin{equation}\label{eq:36}
	\supp F(\sqrt{\Delta_\nu}) f\subseteq \{z\in G_\nu : \dist(z,\supp f) \leq r \},\qquad f\in L^2(G_\nu).
\end{equation}
	
Since $F(\sqrt{\Delta_\nu})$ is a $\diamond_\nu$-convolution operator, it is also an integral operator on $G_\nu$,
whose integral kernel $K$, much as in \eqref{eq:heat_conv_int}, is given by
\begin{equation}\label{eq:int_kernel}
K((x,u),(y,v))
= e^{-\nu v} \ell_{(y,v)^-} K_{F(\sqrt{\Delta_\nu})}(x,u).
\end{equation}
Now observe that \eqref{eq:36} can be equivalently restated as
\begin{equation*}
	\supp K\subseteq \{((x,u),(y,v))\in G_\nu\times G_\nu \tc \dist((x,u),(y,v))\leq r \}.
\end{equation*}
Let us write $K_{(y,v)}(x,u)=K((x,u),(y,v) )$. Thus, for any fixed $\varepsilon>0$,
\[
\supp K_{(y,v)} \subseteq \overline{B_{G_\nu}}(0,r+\varepsilon) \qquad \text{for almost all } (y,v) \in B_{G_\nu}(0,\varepsilon). 
\]
Now, by \eqref{eq:int_kernel} and Lemma \ref{lm:transl_Gnu}, it easily follows that
\[
K_{(y,v)} \to K_{(0,0)} = K_{F(\sqrt{\Delta_\nu})} \text{ in } L^2(G_\nu) \text{ as } (y,v) \to (0,0),
\]
whence we deduce that $\supp K_{(0,0)} \subseteq \overline{B_{G_\nu}}(0,r+\varepsilon)$ for any $\varepsilon > 0$, and therefore $\supp K_{(0,0)} \subseteq \overline{B_{G_\nu}}(0,r)$, as desired.
\end{proof}

\section{The multiplier theorem}\label{s:multipliers}

Here we revert to the setting of Section \ref{s:groups}. So $N$ is a $2$-step stratified group as in \eqref{eq:product_N}, that is, a direct product of an abelian group $N^{(0)}$ and M\'etivier groups $N^{(1)},\dots,N^{(\ell)}$, and $G = N \rtimes \RR$ is its semidirect product extension. Recall that $\vec d_1$ and $\vec d_2$ are the vectors of dimensions of the first layers and second layers of the M\'etivier groups $N^{(j)}$, $j=1,\dots,\ell$, while $d^{(0)}$ is the dimension of $N^{(0)}$; in particular, $d = d^{(0)} + |\vec d_1| + |\vec d_2|$ and $Q = d^{(0)} + |\vec d_1| + 2|\vec d_2|$ are the topological and homogeneous dimensions of $N$.

We aim at ``lifting'' the weighted estimate on $N$ contained in Proposition \ref{prop:3} to the semidirect product $G$. To this purpose, we first compare weighted norms for convolution kernels in the calculus for $\Delta_N$ to analogous norms in the calculus of the Bessel operator $L_\nu$, for appropriate choices of $\nu$.

\begin{prop}\label{prop:7}
Let $\vec{0}\preccurlyeq\al\prec \vec{d_2}$ and $\be\geq 0$. For any $F\in\Sz(\RR)$,
\begin{equation*}
	\int_N | K_{F(\Delta_N)}(z)|^2  |z|_N^\be \prod_{j=1}^{\ell } |z'_j |^{\al_j} \,\dz
	\lesssim_\alpha \int_0^\infty |K_{F(L_{Q-\val})}(x)|^2 x^{\be}\, \dd\mu_{Q-\val}(x). 
\end{equation*}
\end{prop}
\begin{proof}
The proof strongly relies on ideas from \cite{Si}.
	
Fix $\vec{0}\preccurlyeq\al\prec \vec{d_2}$, $\be \geq 0$ and $F\in\Sz(\RR)$. In view of \eqref{eq:17}, the case $\be=0$ is already covered by Proposition \ref{prop:3}, so we may assume $\be>0$. Denote $\nu=Q-\val$; observe that $\nu > Q-|\vec d_2| = d \geq 3$, as $N$ is a nonabelian $2$-step group.
	
Define $G \in \Sz_e(\Rnon)$ by $G(\lambda)=F(\lambda^2)$, and, for $r>0$,
\begin{equation*}
G_r \defeq H_{\nu} ( \ind_{[0,r)} H_{\nu}^{-1} G), \qquad G^r \defeq H_{\nu}( \ind_{[r,\infty)} H_{\nu}^{-1} G ).
\end{equation*}
Thus, $G=G_r+G^r$. Moreover, by \eqref{eq:4},
\begin{equation*}
K_{G_r(\sqrt{L_{\nu}})} =K_{G(\sqrt{L_\nu})} \ind_{[0,r)}, \qquad  K_{G^r(\sqrt{L_{\nu}})} =K_{G(\sqrt{L_\nu})} \ind_{[r,\infty)}.
\end{equation*}

By \eqref{eq:bessel} and 
\cite[eq.\ (10.9.4)]{DLMF},
\begin{equation*}
j_x^\nu(y) = \frac{1}{\sqrt{\pi}}  \int_{-1}^1 (1-\xi^2)^{\nu/2-3/2} e^{ixy\xi} \,\dd\xi, \qquad x,y \in \RR.
\end{equation*}
Hence, if $\Four$ denotes the Euclidean Fourier transform on $\RR$,
\begin{equation}\label{eq:21}
\Four(j_x^\nu)(\xi) = \frac{2\sqrt{\pi}}{x} \left(1- \left|\frac{\xi}{x}\right|^2\right)_+^{(\nu-3)/2} .
\end{equation}

Denote by $\widetilde{G}_r$ the even extension of $G_r$ to $\RR$; namely, by \eqref{eq:hankeltr},
\[
\widetilde{G}_r(y) = \int_0^r H_{\nu}^{-1} G(x) \, j^\nu_{x} (y)\,\dmu(x)
\]
for all $y \in \RR$ (recall that $j^\nu_x$ is even).
As $\nu > 3$, by \eqref{eq:21} we get
\begin{equation*}
\Four(\widetilde{G}_r)(\xi) = 2\sqrt{\pi} \int_0^r H_\nu^{-1} G(x) \left(1- \left|\frac{\xi}{x}\right|^2\right)_+^{(\nu-3)/2} \,x^{\nu-2} \,\dd x;
\end{equation*}
notice that $|\xi| \leq x \leq r$ in the above integral, otherwise the integrand vanishes. Thus, $\supp\Four(\widetilde{G}_r)\subseteq[-r,r]$.

Recall from \eqref{eq:fps_N} that $\Delta_N$ has finite propagation speed. Much as in \eqref{eq:fps_subordination}, from $\supp\Four(\widetilde{G}_r)\subseteq[-r,r]$ we then deduce that
\begin{equation*}
\supp K_{G_r(\sqrt{\Delta_N})}\subseteq\{z\in N \tc |z|_N \leq r\}.
\end{equation*}
Consequently, for $|z|_N>r$ we have $K_{G^r(\sqrt{\Delta_N})}(z)=K_{G(\sqrt{\Delta_N})}(z)$. Hence,
\[\begin{split}
	&\int_N | K_{G(\sqrt{\Delta_N})}(z)|^2  |z|_N^\be \prod_{j=1}^{\ell } |z'_j |^{\al_j} \,\dz \\
	&= \int_0^\infty \be r^{\be-1} \int_{|z|_N>r} | K_{G(\sqrt{\Delta_N})}(z)|^2 \prod_{j=1}^{\ell } |z'_j |^{\al_j}\,\dz\,\dd r\\
	&\leq \int_0^\infty \be r^{\be-1} \int_{N} | K_{G^r(\sqrt{\Delta_N})}(z)|^2 \prod_{j=1}^{\ell } |z'_j |^{\al_j}\,\dz\,\dd r.
\end{split}\]
Recall that $\nu = Q-\val$. We apply Proposition \ref{prop:3} and \eqref{eq:17} to obtain
\[\begin{split}
&\int_0^\infty \be r^{\be-1} \int_{N} | K_{G^r(\sqrt{\Delta_N})}(z)|^2 \prod_{j=1}^{\ell } |z'_j |^{\al_j}\,\dz\,\dd r \\
&\lesssim \int_0^\infty \be r^{\be-1} \int_{0}^\infty | K_{G^r(\sqrt{L_\nu})}(x)|^2 \,\dmu(x)\,\dd r\\
&=\int_0^\infty \be r^{\be-1} \int_{r}^\infty | K_{G(\sqrt{L_\nu})}(x)|^2 \,\dmu(x)\,\dd r\\
&= \int_{0}^\infty | K_{G(\sqrt{L_\nu})}(x)|^2 x^\be\,\dmu(x).
\end{split}\]
Combining the above finishes the proof.
\end{proof}

We can now lift the previous inequality to $G$ and $G_\nu$.

\begin{cor}\label{cor:2}
Let $\vec{0}\preccurlyeq\al\prec \vec{d_2}$ and $\be\geq 0$. For any $F\in\Sz(\RR)$,
\begin{multline}\label{eq:weighted_lifted}
	\int_G | K_{F(\Delta)}(z,u)|^2  |z|_N^\be \prod_{j=1}^{\ell } |z'_j |^{\al_j} \,\dz\,\du \\
	\lesssim \int_{G_{Q-\val}} |K_{F(\Delta_{Q-\val})}(x,u)|^2 x^{\be}\,\dd\mu_{Q-\val}(x)\,\du. 
\end{multline}
\end{cor}
\begin{proof}
Observe that Lemma \ref{lm:5} implies
\begin{multline}\label{eq:fubini_Gnu}
	\int_{G_{Q-\val}} |K_{F(\Delta_{Q-\val})}(x,u)|^2 x^{\be}\,\dd\mu_{Q-\val}(x)\,\du\\
	 =\int_\RR \int_0^\infty | K_{(\Phi F)_u(L_{Q-\val})}(x)|^2 x^\be \,\dd\mu_{Q-\val}(x)\,\du.
\end{multline}
Moreover, an analogue of Lemma \ref{lm:5} is true for $\Delta$ and $\Delta_N$ in place of $\Delta_\nu$ and $L_\nu$:
\[
	K_{ F(\Delta)}(z,u)=K_{(\Phi F)_u(\Delta_N)}(z)
\]
for almost all $(z,u) \in G$; this can be proved much in the same way as in Lemma \ref{lm:5}, and is implicitly used in \cite[Corollary 4.5]{MOV}. 
This gives
\begin{multline}\label{eq:fubini_G}
	\int_{G_{Q}} |K_{F(\Delta)}(z,u)|^2  |z|_N^{\be}\prod_{j=1}^{\ell } |z'_j |^{\al_j} \,\dz\,\du \\
	=\int_\RR \int_N | K_{(\Phi F)_u(\Delta_N)}(z)|^2 |z|_N^\be \prod_{j=1}^{\ell } |z'_j |^{\al_j} \,\dz\,\du.
\end{multline}
Now, Proposition \ref{prop:7} allows us to compare the right-hand sides of \eqref{eq:fubini_Gnu} and \eqref{eq:fubini_G} and deduce the desired inequality.
\end{proof}

In the case $\beta = 0$, the right-hand side of \eqref{eq:weighted_lifted} can be turned into a weighted $L^2$-norm of $F$ by Proposition \ref{prop:2}, thus yielding the weighted Plancherel estimate \eqref{eq:wplancherel} on $G$ discussed in the introduction.

We now combine the previous weighted estimate on $G$ together with finite propagation speed to deduce the following $L^1 \to L^2$ bound, corresponding to \eqref{eq:l1l2}. This improves \cite[Proposition 5.1]{MOV}, where the case $\nu=Q$ of the following bound is proved; of course, this improvement depends on our assumptions on $N$, which are more restrictive than those in \cite{MOV}.

Recall from \eqref{eq:fctE} the definition of $\fctE_r$.

\begin{prop}\label{prop:4}
Let $\nu \in (d,Q]$. Let $F \in \fctE_r$ for some $r>0$. 
Then,
\begin{equation*}
	\int_{G} | K_{F(\sqrt{\Delta})}(z,u)| \,\dz\,\du \lesssim_\nu r^{[(\nu+1)/2,3/2]} \left(\int_0^\infty |F(\lambda)|^2 \,\lambda^{[3,\nu+1]}\,\ddlam\right)^{1/2}.
\end{equation*}
\end{prop}
\begin{proof}
As $\nu \in (d,Q]$, we can choose $\al$ so that $\vec{0}\preccurlyeq\al\prec \vec{d_2}$ and $\nu = Q-\val$. Recall from \eqref{eq:fps_G} that $\Delta$ satisfies finite propagation speed, thus $\supp K_{F(\Delta)}\subseteq \overline{B_G}(0_G,r)$. 
	
We begin with the case $r\leq 1$. By the Cauchy--Schwarz inequality,
\begin{multline*}
	\Vert K_{F(\sqrt{\Delta})}\Vert_{L^1(G)} \\
	\leq \left(\int_{B_G(0_G,r)} \prod_{j=1}^{\ell } |z'_j |^{-\al_j} \,\dz\,\du\right)^{1/2} 
	\left(\int_G |K_{F(\Delta)}(z,u)|^2 \prod_{j=1}^{\ell } |z'_j |^{\al_j}\,\dz\,\du\right)^{1/2}.
\end{multline*}
Recall that $\al\prec \vec{d_2}\prec \vec{d_1}$ by \eqref{eq:metivier_dim}. Hence, Corollary \ref{cor:3}, Corollary \ref{cor:2} with $\be=0$, and Proposition \ref{prop:2} with $\nu=Q-\val$ yield
\begin{equation*}
	\Vert K_{F(\sqrt{\Delta})}\Vert_{L^1(G)} 
	\lesssim_\nu r^{(\nu+1)/2} \left(\int_0^\infty |F(\lambda)|^2 \,\lambda^{[3,\nu+1]}\,\ddlam\right)^{1/2}.
\end{equation*}
	
For $r>1$, notice that, by the Cauchy--Schwarz inequality,
\[\begin{split}
	\left(\int_{G} | K_{F(\sqrt{\Delta})}(z,u)| \,\dz\,\du\right)^2
	&\leq  \int_{B_G(0_G,r)} (1+|z|_N^{Q-\val})^{-1} \prod_{j=1}^{\ell } |z'_j |^{-\al_j} \,\dz\,\du\\
	&\quad\times\Biggl(\int_G |K_{F(\sqrt{\Delta})}(z,u)|^2 \prod_{j=1}^{\ell } |z'_j |^{\al_j} \,\dz\,\du\\
	&\quad\quad+\int_G |K_{F(\sqrt{\Delta})}(z,u)|^2 |z|^{Q-\val}_N \prod_{j=1}^{\ell } |z'_j |^{\al_j} \,\dz\,\du\Biggr).
\end{split}\]
Now, as above,
\begin{equation*}
	\int_G |K_{F(\sqrt{\Delta})}(z,u)|^2 \prod_{j=1}^{\ell } |z'_j |^{\al_j} \,\dz\,\du
	\lesssim_\nu \int_0^\infty |F(\lambda)|^2 \,\lambda^{[3,\nu+1]}\,\ddlam.
\end{equation*}
Moreover, by Corollary \ref{cor:3},
\begin{equation*}
	\int_{B_G(0,r)} (1+|z|_N^{Q-\val})^{-1} \prod_{j=1}^{\ell } |z'_j |^{-\al_j} \,\dz\,\du
	\lesssim_\nu r^2.
\end{equation*}
Further, by Corollary \ref{cor:2} with $\be=Q-\val=\nu$,
\[
	\int_G |K_{F(\sqrt{\Delta})}(z,u)|^2 |z|^{Q-\val}_N \prod_{j=1}^{\ell } |z'_j |^{\al_j} \,\dz\,\du
	\lesssim_\nu \int_{G_{\nu}} |K_{F(\Delta_{\nu})}(x,u)|^2 x^{\nu}\,\dmu(x)\,\du.
\]
By Lemma \ref{lm:13} we have $\supp K_{F(\Delta_{\nu})}\subseteq \overline{B_{G_{\nu}}}(0,r)$. Thus, Corollary \ref{cor:4} and Proposition \ref{prop:2} imply
\begin{equation*}
	\int_G |K_{F(\sqrt{\Delta})}(z,u)|^2 |z|^{Q-|\alpha|}_N \prod_{j=1}^{\ell } |z'_j |^{\al_j} \,\dz\,\du
	\lesssim_\nu r \int_0^\infty |F(\lambda)|^2 \,\lambda^{[3,\nu+1]}\,\ddlam,
\end{equation*}
and the desired estimate follows.
\end{proof}

We can now obtain the following improvement of \cite[Propositions 5.3 and 5.5]{MOV}.

\begin{prop}\label{prop:5}
Let $F\in L^2(\RR)$ be supported in $[-4,4]$. Then
\begin{equation*}
	\sup_{z\in G} \int_G | K_{F(t\Delta)}(y,z)| \, (1 + t^{-1/2}\dist(y,z))^\varepsilon \,\dd y
	\lesssim_{s,\varepsilon} \Vert F\Vert_{\sobolev{s}{2}}
\end{equation*}
for all $\varepsilon\geq 0$, and $s,t>0$ satisfying one of the following conditions:
\begin{itemize}
	\item $t\geq 1$ and $s>3/2+\varepsilon$;
	\item $t\leq 1$ and $s>d/2+\varepsilon$.
\end{itemize}
Moreover,
\begin{equation*}
	\int_G | K_{F(t\Delta)}(x,y)-K_{F(t\Delta)}(x,z)| \,\dd x 
	\lesssim_s t^{-1/2}\dist(y,z)\Vert F\Vert_{\sobolev{s}{2}} 
\end{equation*}
for all $y,z\in G$ and $s,t> 0$ satisfying the above condition with $\varepsilon=0$.
\end{prop}
\begin{proof}
The case $t \geq 1$ is already contained in \cite{MOV}. For $t \leq 1$, we choose $\nu \in (d,Q]$ so that $s > \nu/2+\varepsilon$; then the proof runs exactly in the same way as in \cite{MOV}, with Proposition \ref{prop:4} applied in place of \cite[Proposition 5.1]{MOV}.

We point out that the functions $f_\ell$ resulting from the frequency decomposition of $f(\lambda) = F(\lambda^2)$ given by \cite[Lemma 5.2]{MOV} are in the Schwartz class: indeed, from the original proof in \cite[Lemma (1.3)]{He99} we see that $\hat f_\ell = \chi_\ell \hat f$ for appropriate cutoffs $\chi_\ell \in C^\infty_c(\RR)$, and $\hat f$ is analytic because $f$ is compactly supported. So the Schwartz class assumption, implicit in the definition of $\fctE_r$, on the spectral multiplier in Proposition \ref{prop:4} here is not an obstacle.
\end{proof}

Based on the above we can conclude the proof of our main result.

\begin{proof}[Proof of Theorem \ref{thm:1}]
	The proof goes just as in \cite[proof of Theorem 1.1]{MOV}, using Proposition \ref{prop:5} in place of \cite[Propositions 5.3 and 5.5]{MOV}, and relies on the Calder\'on--Zygmund theory developed in \cite{HeSt,MOV}.
\end{proof}

\section{Appendix: Boundary conditions and finite propagation speed}\label{s:appendix}

In this section we recall some basic terminology and results related to self-adjoint extensions of smooth second-order differential operators in divergence form on smooth manifolds, which are used throughout the paper. While our main application is to second-order operators acting on scalar functions, their expression in divergence form naturally leads to considering first-order operators, such as gradient and divergence, which act on vector-valued functions, or more general sections of vector bundles. We find it therefore more natural to work directly in the setup of operators between spaces of sections of vector bundles.

As we shall see, this allows us to discuss in a unified and relatively simple way a number of results about self-adjointness and domains, as well as present a general derivation of the finite propagation speed property under Dirichlet or Neumann conditions based on the first-order approach of \cite{McIMo}.

We introduce some of the setup and notation from \cite{CM}, to which we refer for additional details. Let $M$ be a smooth manifold (without boundary) equipped with a smooth measure. Fix two smooth vector bundles $\bdlE,\bdlF$ on $M$ equipped with fibre inner products. We use notation such as $C^\infty(\bdlE), C^\infty_c(\bdlE), L^2(\bdlE)$ to denote the spaces of sections of $\bdlE$ which are smooth, smooth and compactly supported, square integrable. Let $\nabla : C^\infty(\bdlE)\to C^\infty(\bdlF)$ be a first-order differential operator, and $\nabla^+ : C^\infty(\bdlF)\to C^\infty(\bdlE)$ be its formal adjoint. As usual, we can extend $\nabla$ and $\nabla^+$ to spaces of distributions.

We define the \emph{maximal domain} and the \emph{minimal domain} for $\nabla$ on $L^2(\bdlE)$:
\begin{equation}\label{eq:dmax_dmin}
\Dmax(\nabla)=\{f\in L^2(\bdlE): \nabla f\in L^2(\bdlF) \}, \qquad \Dmin(\nabla)= \overline{C_c^\infty(\bdlE)}^{\Dmax(\nabla)},
\end{equation}
where $\Dmax(\nabla)$ is equipped with the graph norm of $\nabla$.
By a standard mollification technique (see, e.g., \cite[Propositions 5.5 and 6.1]{CM}), one can see that
\begin{equation}\label{eq:dmincomp}
\Dmin(\nabla)=\overline{\Dmaxc(\nabla)}^{\Dmax(\nabla)},
\end{equation}
where $\Dmaxc(\nabla)$ is the set of compactly supported elements of $\Dmax(\nabla)$.
Analogous considerations apply to the formal adjoint $\nabla^+$.
Moreover, formal and Hilbert space adjoints are related as follows:
\begin{equation}\label{eq:10}
( \nabla|_{\Dmin(\nabla)})^*=\nabla^+|_{\Dmax(\nabla^+)}\qquad \text{and}\qquad ( \nabla|_{\Dmax(\nabla)})^*=\nabla^+|_{\Dmin(\nabla^+)}.
\end{equation}

We are interested in the second-order divergence-form operator $\nabla^+\nabla$ associated with $\nabla$. We define the \emph{Dirichlet domain} and the \emph{Neumann domain} for $\nabla^+\nabla$ as
\begin{equation}\label{eq:Neu}
\begin{aligned}
\DDir(\nabla^+\nabla) &= \{f\in \Dmin(\nabla): \nabla f\in\Dmax(\nabla^+) \}, \\
\DNeu(\nabla^+\nabla) &= \{f\in \Dmax(\nabla): \nabla f\in\Dmin(\nabla^+) \} .
\end{aligned}
\end{equation}
In light of \eqref{eq:10} and \cite[Theorem X.25]{RS2}, both $\nabla^+\nabla|_{\DDir(\nabla^+\nabla)}$ and $\nabla^+\nabla|_{\DNeu(\nabla^+\nabla)}$ are nonnegative self-adjoint operators on $L^2(\bdlE)$.
As $(\nabla^+)^+ = \nabla$, analogous considerations apply to the operator $\nabla\nabla^+$.

\begin{rem}
The Dirichlet domain in \eqref{eq:Neu} is the same as the domain of the Friedrichs extension of $\nabla^+ \nabla|_{C^\infty_c(\bdlE)}$ (cf.\ \cite[Theorems X.23 and X.25]{RS2}). Of course, when $\nabla^+ \nabla|_{C^\infty_c(\bdlE)}$ is essentially self-adjoint, the Dirichlet and Neumann domains are the same; however, here we are also interested in the case where essential self-adjointness may fail.
\end{rem}

The \emph{control distance} $\dist_\nabla$ associated with $\nabla$ can be defined by (cf.\ \cite[p.~175]{CM})
\begin{equation*}
\dist_\nabla(x,y) =\inf \left\{\int_0^1 P_\nabla^\ast(\gamma'(t))\,\dd t \tc \gamma\in \AC([0,1];M),\ \gamma(0)=x,\ \gamma(1)=y  \right\}
\end{equation*}
for all $x,y \in M$, where we write $\AC([0,1];M)$ for the set of absolutely continuous curves in $M$ with domain $[0,1]$, and $P_\Delta$ for the fibre seminorm on $T^\ast M$ associated with $\nabla$, namely,
\begin{equation*}
P_\nabla (\xi) = | \sigma_1(\nabla)(\xi) |_{\op}, \qquad \xi \in T^*M,
\end{equation*}
 while $\sigma_1(\nabla) \in C^\infty(\Hom(T^\ast M, \Hom(\bdlE,\bdlF)))$ is the symbol of $\nabla$, and $P^\ast_\nabla$ denotes the extended fibre norm on $TM$ dual to $P_\nabla$ (see \cite[p.~153]{CM} for details).

\begin{rem}\label{rem:distance}
As discussed in \cite[Section 8.5]{CM}, this definition of the control distance $\dist_\nabla$ includes, for appropriate choices of $\nabla$, that of Riemannian and sub-Riemannian distances on $M$. In particular, assume that the bundles $\bdlE$ and $\bdlF$ are the trivial bundles of ranks $1$ and $r$, while $\nabla$ has the form
\[
\nabla = \begin{pmatrix} X_1  \\ \vdots \\ X_r \end{pmatrix}
\]
for a system $X_1,\dots,X_r$ of linearly independent smooth real vector fields on $M$. If $r=\dim M$ and $g$ is the Riemannian metric tensor on $M$ that makes $X_1,\dots,X_r$ an orthonormal frame, then $P^\ast_\nabla(v) = \sqrt{g(v,v)}$ for all $v \in TM$, and $\dist_\nabla$ is the Riemannian distance on $M$ induced by $g$. More generally, if $r\leq\dim M$ and the vector fields $X_1,\dots,X_r$ are bracket-generating, then $\dist_\nabla$ is the Carnot--Carath\'eodory distance associated to the system of vector fields (cf.\ also \cite{Mo,VSC}). In each of these cases, the control distance $\dist_\nabla$ induces on $M$ the manifold topology (on this, see also \cite[Proposition 4.23]{CM}).
\end{rem}

We say that a family  $\{W_t\}_{t\in\RR}$ of bounded operators on $L^2(\bdlE)$ has \emph{finite propagation speed} with respect to $\dist_\nabla$ if
\begin{equation}\label{eq:fps_family}
\supp W_t f\subseteq \{x\in M \tc \dist_\nabla(x,\supp f)\leq |t| \} \qquad \forall f\in L^2(\bdlE), \ t \in \RR.
\end{equation}
We are especially interested in the case $W_t = \cos(t \sqrt{\nabla^+ \nabla})$, corresponding to the \emph{wave propagator} associated to $\nabla^+ \nabla$; in this case, when \eqref{eq:fps_family} holds, we also say that $\nabla^+\nabla$ satisfies the finite propagation speed property.

Of course, the definition of the cosine family $\{\cos(t \sqrt{\nabla^+ \nabla})\}_{t \in \RR}$ in terms of spectral calculus on $L^2(\bdlE)$ becomes meaningful only once a self-adjoint extension of $\nabla^+ \nabla$ has been chosen, and indeed the validity of the finite propagation speed property may depend on this choice. This is easily seen, e.g., by taking $\nabla = -\partial_x$ on the interval $M = (0,1)$ with Lebesgue measure, and considering $\nabla^+ \nabla = -\partial_x^2$ with periodic boundary conditions: here finite propagation speed is violated, because, roughly speaking, due to periodicity, a wave that crosses one endpoint of the interval immediately re-enters from the other endpoint, effectively travelling at infinite speed (cf.\ \cite[Section X.1, Example 1]{RS2}).

Possibly to avoid such problems, several results on finite propagation speed in the literature are either stated under some completeness assumption on the underlying (sub-)Riemannian manifold (see, e.g., \cite{Mel} or the examples in \cite[Sections 4-8]{Si3}), effectively preventing a solution with compactly supported initial datum from reaching the boundary in finite time, or proved only for a restricted time interval depending on the datum (see, e.g., \cite[Section 7]{CM}). 

In contrast, the general approach of \cite{McIMo} can be readily used to deduce finite propagation speed for any second-order divergence form operator, for appropriate choices of the self-adjoint extension, without any completeness assumptions on the manifold. (In fact, the results of \cite{McIMo} go well beyond our setup, as they do not require self-adjointness or smooth coefficients, and also apply to operators on $L^p$ for $p \neq 2$.) The examples presented in \cite[Section 5]{McIMo} only consider elliptic differential operators; however, as the next statement shows, ellipticity can be replaced by a weaker topological assumption on the control distance, which holds in greater generality, including for sub-elliptic operators (see Remark \ref{rem:distance}).

\begin{prop}\label{prop:8}
Assume that the control distance associated with $\nabla$ induces the manifold topology on $M$. Consider $\nabla^+\nabla$ with either Neumann domain or Dirichlet domain. Then $\{\cos(t\sqrt{\nabla^+\nabla})\}_{t\in\RR}$ has finite propagation speed with respect to the control distance. 
\end{prop}
\begin{proof}
We are going to deduce the finite propagation speed property from \cite[Theorem 3.1]{McIMo}. To this purpose, we shall write, roughly speaking, the second-order operator $\nabla^+ \nabla$ as the square of a first-order differential operator.
	
More precisely, we define a ``matrix operator'' $\DD : C^\infty (\bdlE\oplus \bdlF)\to C^\infty(\bdlE\oplus \bdlF)$ by
\begin{equation*}
	\DD=\begin{pmatrix}
	0 & \nabla^+\\
	\nabla & 0
	\end{pmatrix}.
\end{equation*}
Clearly $\DD$ is formally self-adjoint, i.e.\ $\DD^+=\DD$, and moreover
\[
	\DD^2=\begin{pmatrix}
	\nabla^+ \nabla & 0\\
	0 & \nabla \nabla^+
	\end{pmatrix}.
\]
	
We now construct an appropriate self-adjoint extension of $\DD$. Let
\begin{equation}\label{eq:32}
	\Dmaxmin(\DD)=\Dmax(\nabla)\oplus\Dmin(\nabla^+).
\end{equation}
By \eqref{eq:10} we obtain 
\[\begin{split}
	(\DD|_{\Dmaxmin(\DD)})^\ast=\begin{pmatrix}
	0 & \nabla^+|_{\Dmin(\nabla^+)}\\
	\nabla_{\Dmax(\nabla)} & 0
	\end{pmatrix}^\ast
	&=\begin{pmatrix}
	0 & \nabla^+|_{\Dmin(\nabla^+)}\\
	\nabla_{\Dmax(\nabla)} & 0
	\end{pmatrix}\\
	&=\DD|_{\Dmaxmin(\DD)}.
\end{split}\]
Thus, $\DD|_{\Dmaxmin(\DD)}$ is self-adjoint. Consequently, the square operator $\DD^2|_{\Dmaxmin(\DD^2)}$ is self-adjoint as well (see \cite[Theorem X.25]{RS2}), where
\begin{equation*}
\begin{split}
	&\Dmaxmin(\DD^2) \\
	&\defeq \left\{\begin{pmatrix}
	f\\
	g
	\end{pmatrix}\in \Dmaxmin(\DD) \tc \DD\begin{pmatrix}
	f\\
	g
	\end{pmatrix}\in \Dmaxmin(\DD) \right\} \\
	&=
	\{ f \in \Dmax(\nabla) \tc \nabla f \in \Dmin(\nabla^+) \} \oplus \{ g \in \Dmin(\nabla^+) \tc \nabla^+ g \in \Dmax(\nabla) \} \\
	&=
	\DNeu(\nabla^+\nabla)\oplus \DDir(\nabla\nabla^+),
\end{split}
\end{equation*}
by \eqref{eq:32}, \eqref{eq:Neu} and the definition of $\DD$.
Hence,
\begin{equation*}
	\DD^2|_{\Dmaxmin(\DD^2)}=\begin{pmatrix}
	\nabla^+\nabla|_{\DNeu(\nabla^+\nabla)} & 0\\
	0 & \nabla\nabla^+|_{\DDir(\nabla\nabla^+)}
	\end{pmatrix},
\end{equation*} 
and consequently
\begin{equation}
	\label{eq:31}
	\cos\left(t\DD|_{\Dmaxmin(\DD)}\right)
	=\begin{pmatrix}
	\cos\left(t\sqrt{\nabla^+\nabla|_{\DNeu(\nabla^+\nabla)}}\right) & 0\\
	0 & \cos\left(t\sqrt{\nabla\nabla^+|_{\DDir(\nabla\nabla^+)}}\right)
	\end{pmatrix},
\end{equation}
since $\cos(t\cdot)$ is an even analytic function.

We want to apply \cite[Theorem 3.1]{McIMo} to the operator $\DD|_{\Dmaxmin(\DD)}$ on $L^2(\bdlE \oplus \bdlF)$, where the underlying manifold $M$ is equipped with the distance $\dist_\nabla$. As $\DD|_{\Dmaxmin(\DD)}$ is self-adjoint, it generates a $C_0$ group $(e^{it\DD|_{\Dmaxmin(\DD)}})_{t \in \RR}$ of unitary operators on $L^2(\bdlE \oplus \bdlF)$, so the first assumption of the theorem is satisfied.

We now check the second assumption of \cite[Theorem 3.1]{McIMo}. Let $\Lip(M)$ be the space of bounded real-valued $\dist_\nabla$-Lipschitz functions; in other words, $\eta\in \Lip(M)$ if and only if $\eta\in L^\infty(M;\RR)$ and
\begin{equation*}
\|\eta\|_{\Lip} \defeq \sup_{x,y \in M, \, x \neq y} \frac{| \eta(x)- \eta(y)|}{\dist_\nabla(x,y)} < \infty.
\end{equation*}
Since $\dist_\nabla$ is varietal, i.e.\ $\dist_\nabla$ induces the topology on the manifold $M$, by \cite[Propositions 5.2 and 5.4]{CM} we have the equivalent characterisation
\begin{equation}\label{eq:lipschitz}
\begin{gathered}
	\Lip(M) = \{ \eta \in L^\infty(M;\RR): \nabla^\sigma \eta \in L^\infty(\Hom(\bdlE,\bdlF)) \}, \\
		\|\eta\|_\Lip = \|\nabla^\sigma \eta\|_{L^\infty} \qquad \forall \eta \in \Lip(M),
\end{gathered}
\end{equation}
where $\nabla^\sigma$ denotes the ``symbol operator'' (see \cite[Section 2.2]{CM}), that is, the first-order differential operator that appears in the Leibniz rule for $\nabla$:
\begin{equation}\label{eq:leibniz}
	\nabla(\eta f) = (\nabla^\sigma \eta) f + \eta \nabla f,
\end{equation}
where $\eta$ is a scalar-valued function and $f$ is a section of $\bdlE$ with suitable differentiability and integrability properties
\cite[Proposition 3.7]{CM}.
By \eqref{eq:dmax_dmin}, \eqref{eq:lipschitz} and \eqref{eq:leibniz}, we deduce that
\begin{align*}
	\Lip(M)\cdot \Dmax(\nabla) &\subseteq \Dmax(\nabla), \\
	\Lip(M)\cdot \Dmaxc(\nabla) &\subseteq \Dmaxc(\nabla),
\end{align*}
whence, by \eqref{eq:dmincomp},
we also get	
\begin{equation*}
	\Lip(M)\cdot \Dmin(\nabla)\subseteq \Dmin(\nabla).
\end{equation*}
Since the control distances associated to the three operators $\nabla$, $\nabla^+$, $\DD$ are the same (see \cite[p.~181]{CM}), 
the corresponding Lipschitz space $\Lip(M)$ is the same, so similar inclusions are true for $\nabla^+$. Thus, by  \eqref{eq:32} we obtain
\begin{equation*}
	\Lip(M)\cdot \Dmaxmin(\DD)\subseteq \Dmaxmin(\DD).
\end{equation*}
In particular, for all $\eta \in \Lip(M)$ and $f \in \Dmaxmin(\DD)$, we have $\eta f \in \Dmaxmin(\DD)$; moreover, by \eqref{eq:leibniz} and \eqref{eq:lipschitz}, 
\[
\| [\eta,\DD] f \|_{L^2} = \| (\DD^\sigma \eta) f \|_{L^2} \leq \| \DD^\sigma \eta \|_{L^\infty} \| f \|_{L^2} = \| \eta \|_{\Lip} \| f \|_{L^2},
\]
and finally $[\eta,[\eta,\DD]] = [\eta,\DD^\sigma \eta] = 0$, as the scalar multiplication operator $\eta$ commutes with the matrix multiplication operator $\DD^\sigma \eta$.

Hence, we can apply \cite[Theorem 3.1]{McIMo} to $\DD$ on the domain $\Dmaxmin(\DD)$, and deduce finite propagation speed for the one-parameter group $\{e^{it\DD|_{\Dmaxmin(\DD)}}\}_{t\in\RR}$.
As $e^{it\DD}+e^{-it\DD}=2\cos(t\DD)$, by \eqref{eq:31} we obtain finite propagation speed for the cosine families $\left\{\cos\left(t\sqrt{\nabla^+\nabla|_{\DNeu(\nabla^+\nabla)}}\right)\right\}_{t\in\RR}$ and $\left\{\cos\left(t\sqrt{\nabla \nabla^+|_{\DDir(\nabla^+\nabla)}}\right)\right\}_{t\in\RR}$. This gives the desired result for $\nabla^+ \nabla$ with Neumann domain; by exchanging the roles of $\nabla$ and $\nabla^+$ in the above argument, we also obtain the result for $\nabla^+ \nabla$ with Dirichlet domain.
\end{proof}

\end{document}